\documentclass[11pt,reqno,a4paper]{amsart}
\usepackage{fbb}
\usepackage[utf8]{inputenc}
\usepackage[T1]{fontenc}
\usepackage{esint}

\usepackage[top=2.5cm,bottom=2.4cm,left=2.6cm,right=2.6cm,headsep=0.2in]{geometry}
\usepackage{microtype}

\usepackage{amsmath,amsthm,amssymb,amsfonts,mathtools,mathrsfs}
\usepackage[dvipsnames]{xcolor}

\usepackage{enumitem}
\setlist[itemize]{leftmargin=1.6em}
\setlist[enumerate]{leftmargin=1.6em}

\usepackage[noadjust]{cite}

\usepackage{hyperref}
\hypersetup{
 colorlinks=true,
 linkcolor=blue!50!green,
 citecolor=blue!50!green,
 urlcolor=blue!50!green,
 pdfauthor={Mayukh Mukherjee},
 pdftitle={Wavelength-scale optional stopping, critical Feynman-Kac gauges, and capacitary spectral inequalities}
}
\usepackage{aliascnt}

\makeatletter
\renewcommand{\paragraph}{\@startsection{paragraph}{4}%
  \z@\z@{-\fontdimen2\font}%
  {\normalfont\bfseries}}
\makeatother

\numberwithin{equation}{section}
\theoremstyle{plain}
\newtheorem{theorem}{Theorem}[section]

\newaliascnt{lemma}{theorem}
\newtheorem{lemma}[lemma]{Lemma}
\aliascntresetthe{lemma}

\newaliascnt{proposition}{theorem}
\newtheorem{proposition}[proposition]{Proposition}
\aliascntresetthe{proposition}

\newaliascnt{corollary}{theorem}
\newtheorem{corollary}[corollary]{Corollary}
\aliascntresetthe{corollary}

\theoremstyle{definition}
\newaliascnt{definition}{theorem}
\newtheorem{definition}[definition]{Definition}
\aliascntresetthe{definition}

\theoremstyle{remark}
\newaliascnt{remark}{theorem}
\newtheorem{remark}[remark]{Remark}
\aliascntresetthe{remark}

\theoremstyle{plain}
\newtheorem{mainthm}{Theorem}

\usepackage[capitalise,nameinlink]{cleveref}
\crefname{theorem}{Theorem}{Theorems}
\crefname{lemma}{Lemma}{Lemmas}
\crefname{proposition}{Proposition}{Propositions}
\crefname{corollary}{Corollary}{Corollaries}
\crefname{definition}{Definition}{Definitions}
\crefname{section}{Section}{Sections}
\crefname{equation}{Eq.}{Eqs.}
\crefname{remark}{Remark}{Remarks}
\crefname{mainthm}{Theorem}{Theorems}
\Crefname{equation}{Equation}{Equations}
\Crefname{theorem}{Theorem}{Theorems}
\Crefname{lemma}{Lemma}{Lemmas}
\Crefname{proposition}{Proposition}{Propositions}
\Crefname{corollary}{Corollary}{Corollaries}
\Crefname{definition}{Definition}{Definitions}
\Crefname{remark}{Remark}{Remarks}
\Crefname{mainthm}{Theorem}{Theorems}

\usepackage{fancyhdr}
\title[Wavelength-scale optional stopping]{Wavelength-scale optional stopping, critical Feynman--Kac gauges, and capacitary spectral inequalities}
\author{Mayukh Mukherjee}

\address{Department of Mathematics, Indian Institute of Technology Bombay,
Powai, Mumbai 400076, India}
\email{mukherjee@math.iitb.ac.in, mathmukherjee@gmail.com}


\begin{document}

\begin{abstract}
We develop two Brownian stopping methods: one uses hitting probabilities and capacity, while the other uses moments of weighted exit distributions. On a closed $m$-dimensional Riemannian manifold, $m\geq 2$, the first method gives low-energy spectral inequalities for equilibrium measures supported on sets of zero volume. Combining scales gives heat observability and null controllability from a full-support measure carried by a dense set of Hausdorff dimension $m-2$, which is minimal under a bounded $H^1$-trace condition. More generally, for every nondecreasing unbounded function $\Phi$ satisfying $\Phi(4E)\leq D\Phi(E)$, the measure can be chosen so that the optimal spectral constant is comparable to $\Phi(E)$ and the small-time control cost to $T^{-1/2}\Phi(T^{-1})^{1/2}$. For weighted exits, squaring the Feynman--Kac martingale introduces the doubled potential $2V$. For radial $V$, the resulting boundary second moment is log-convex in $\log r$ if $r^2V(r)$ is nondecreasing, while the regular solution of $(\tfrac12\Delta+2V)g=0$ remains positive; the factor $2$ is sharp when $V(0)>0$. On small geodesic balls, weighted second moments give an almost-monotone frequency and a weighted three-radius inequality. Bounds for the defining logarithmic derivative give doubling estimates; after a ground-state transform, these supply the doubling-index input for the Remez theorem of Logunov--Malinnikova. The normalised first-moment exit law instead gives an averaged boundary-variance identity for Laplace eigenfunctions. At $r=c\lambda^{-1/2}$, it shows that, outside a set carrying only $O(c^{2/m})$ of the total $\varphi^2$-mass, the component of $\{\varphi(x)\varphi>0\}\cap B(x,r)$ containing $x$ contains a concentric ball of nearly full radius, occupies nearly all of $B(x,r)$, and has nearly the same first Dirichlet eigenvalue as $B(x,r)$. First-moment stopping also controls positive superlevel components. Finally, reflected Brownian motion and boundary local time give $L^\infty$ estimates from the boundary into a collar for Steklov eigenfunctions on bounded Lipschitz domains, together with an $L^2$ estimate for $C^2$ domains.
\end{abstract}

\maketitle

\setcounter{tocdepth}{1}
\tableofcontents

\section{Introduction}\label{sec:intro}

This paper uses Brownian motion in two distinct ways. The first records whether a path hits a compact set before leaving a ball. Relative capacity controls this hitting probability, and killing by the corresponding equilibrium measure turns the hitting estimate into a local spectral gap. Combining these local estimates across cells and scales gives spectral inequalities, heat observability, and null controllability from measures carried by sets of zero volume.

The second method records the value of a solution at the point where Brownian motion leaves a ball. Let $(X,g)$ be an $m$-dimensional Riemannian manifold, $m\ge2$, let $\varphi$ satisfy $\Delta_g\varphi+\lambda\varphi=0$ with $\lambda>0$, and let $(B_t)$ have generator $\tfrac12\Delta_g$. Then $M_t=e^{(\lambda/2)t}\varphi(B_t)$ is a local martingale. For fixed $0<c<j_{(m-2)/2,1}$, put $R=c\lambda^{-1/2}$. When $R$ is sufficiently small relative to the local geometry, the stopped process is uniformly integrable, and
\begin{equation}\label{eq:intro-OST}
\varphi(x)=\mathbb E_x\!\left[
e^{(\lambda/2)\tau_{B(x,R)}}\varphi(B_{\tau_{B(x,R)}})\right].
\end{equation}
The corresponding weighted exit measure has total mass $q_m(c^2/2)(1+O_{\mathrm{bg},c}(R^2))$. After division by its total mass, its density with respect to normalised surface measure is $1+O_{\mathrm{bg},c}(R^2)$ (\cref{eq:q_m_def,prop:pole_isotropy_explicit}).

The first and second moments of this exit law give different information. When a second moment is used, we further require $c<2^{-1/2}j_{(m-2)/2,1}$. Squaring the martingale doubles the Feynman--Kac exponent, and the variance of its terminal value becomes a Green-kernel integral of $|\nabla\varphi|^2$ (\cref{thm:wavelength-stopping-identities}). In the radial Euclidean model this gives exact logarithmic convexity and explains the coefficient $2$; on small geodesic balls it gives an almost-monotone frequency. Normalising the singly weighted exit law instead gives the boundary variance used in the local nodal estimates. Applying the first-moment identity at the maximum of a superlevel component gives unconditional estimates for that component. Finally, in the Steklov problem, reflected Brownian motion and boundary local time provide the corresponding boundary argument.  

\subsection{Main results}\label{subsec:intro-geom}

\paragraph{A.\ Capacitary spectral inequalities and heat control at the critical dimension.}
For $m\ge2$, Brownian hitting at scale $R$ is governed by relative capacity rather than volume. Its natural size is $R^{m-2}$; in dimension $2$ this is scale-independent, and suitable sets of Hausdorff dimension zero have positive capacity. Polar sets cannot carry the finite-energy measures used below.

\begin{mainthm}[Capacitary spectral inequality and heat control]\label{thm:A}
Let $(X,g)$ be a closed Riemannian manifold of dimension $m\ge2$ and let $H=-\tfrac12\Delta_g$. There are $r_*=r_*(m,\mathrm{bg})>0$ and $c=c(m,\mathrm{bg})>0$ such that the following holds. Let $0<R\le r_*$, and suppose every ball $B(x_j,R)$ of a maximal $R$-separated net contains a compact set $K_j$ with
\[
\operatorname{Cap}_{B(x_j,2R)}(K_j)\ge\kappa R^{m-2},
\qquad 0<\kappa\le1.
\]
If $\mu_R$ is the sum of the corresponding equilibrium measures, then
\[
R^2\int_X|f|^2\,d\mu_R\ge c\kappa\|f\|_{L^2(X)}^2
\qquad
\bigl(f\in\operatorname{Ran}P_{[0,c\kappa R^{-2}]}(H)\bigr).
\]

Fix $0<R_0\le r_*$, set $T_0=R_0^2$ and $R_k=2^{-k}R_0$. For every $\theta>0$ there is a finite measure $\mu_\theta$, independent of $T$, with $H^1(X)\hookrightarrow L^2(\mu_\theta)$ and $\operatorname{supp}\mu_\theta=X$, carried by a dense zero-volume $F_\sigma$-set of Hausdorff dimension $m-2$, such that, for $0<T\le T_0$ and $u_0\in L^2(X)$,
\[
\|e^{-TH}u_0\|_{L^2(X)}^2\le
\frac{C_\theta}{T}\bigl(1+\log(T_0/T)\bigr)^\theta
\int_0^T\!\!\int_X|e^{-tH}u_0|^2\,d\mu_\theta\,dt .
\]
\end{mainthm}

By duality, $\partial_ty+Hy=v\,\mu_\theta$ is null-controllable at cost $C_\theta T^{-1/2}(1+\log(T_0/T))^{\theta/2}$.

On each fixed manifold and for all sufficiently small $R$, the theorem of Filbir and Mhaskar \cite{FilbirMhaskar2011} implies both inequalities in \cref{thm:capacity-sampling} for $E\le c_XR^{-2}$, independently of $\Lambda$. For equilibrium measures, this gives the fixed-manifold counterpart of the one-scale conclusion of \cref{thm:A}, with an energy window independent of $\kappa$. Relative to that theorem, the additional content here is the local $H^1$ estimate and the equilibrium-measure form inequality with constants controlled by the stated bounded-geometry data, together with the construction of a fixed full-support measure carried by a zero-volume set of Hausdorff dimension $m-2$ and the matching spectral, heat-observability, and null-control profiles; see \cref{rem:concurrent-comparison,rem:heat-control-context}.

Two local consequences locate the endpoint. When $\kappa$ is bounded below, the energy scale $R^{-2}$, the capacity scale $R^{m-2}$, and the normalisation $R^2d\mu_R$ are sharp; no optimality is asserted for the $\kappa$-dependence of the energy window (\cref{rem:sampling-sharp}). Above the endpoint, $s$-Ahlfors regular measures with $R/4$-dense support satisfy the analogous inequality with normalisation $R^{m-s}d\sigma$. At $s=m-2$ the construction instead uses the non-Ahlfors sets $K_*$.

The fixed-measure construction is not tied to a logarithm. Every nondecreasing, unbounded $\Phi:[T_0^{-1},\infty)\to[1,\infty)$ satisfying $\Phi(4s)\le D\Phi(s)$ for $s\ge T_0^{-1}$ and some $D<\infty$, and extended constantly below $T_0^{-1}$, is realised by a measure $\mu_\Phi$. Its spectral loss is $\Phi(E)$ for $E\ge T_0^{-1}$, and its control cost is $T^{-1/2}\Phi(T^{-1})^{1/2}$ for $0<T\le T_0$ (\cref{thm:prescribed-loss}). The same measure satisfies the coupling law
\[
\inf\sigma(H+\gamma\mu_\Phi)\asymp
\inf_{E\ge0}\left(E+\frac{\gamma}{\Phi(E)}\right).
\]
The corresponding upper bounds remain uniform over fixed symmetric uniformly elliptic classes (\cref{thm:uniform-elliptic-control}). To realise the mildly doubling profiles above, the observation measure must have full support; otherwise the spectral loss grows faster than every power. The carrier dimension $m-2$ is minimal among measures with bounded trace $H^1(X)\to L^2(\mu)$ (\cref{prop:divergence-full-support-necessary,prop:control-dim-sharp}).

The variational proof is shorter. The Brownian proof gives different information: it determines the law, up to exit from $\Omega$, of the additive functional whose Revuz measure is the equilibrium measure. For quasi-every $y$, $A^{K,\Omega}_{\tau_\Omega}$ is zero with probability $1-e_{K,\Omega}(y)$ and, conditional on being positive, is exponentially distributed with mean one. Equivalently, \cref{lem:equilibrium-additive-functional-law} gives
\[
\mathbb E_y e^{-\beta A^{K,\Omega}_{\tau_\Omega}}
=
1-\frac{\beta}{1+\beta}e_{K,\Omega}(y)
=
1-\frac{\beta}{1+\beta}
\mathbb P_y(\tau_K<\tau_\Omega)
\quad\text{for quasi-every }y.
\]
Thus the Feynman--Kac survival deficit at coupling $\beta$ is exactly $\beta/(1+\beta)$ times the probability of hitting $K$ before leaving $\Omega$; as $\beta\to\infty$, the survival probability tends to the probability of leaving without hitting $K$. In each cell, the capacity hypothesis makes the hitting probability $\gtrsim\kappa$ throughout the inner ball, while the expected killed occupation time is $O(R^2)$. The strong Markov property, applied after successive cell exits, then gives the global estimate in \cref{prop:equilibrium-killing-spectral-gap}:
\[
\inf\sigma(H+\beta\mu_R)
\ge
c\frac{\beta}{1+\beta}\kappa R^{-2},
\qquad \beta>0.
\]
At $\beta=1$, projection to $E\le c\kappa R^{-2}$ gives the spectral inequality. The exact law also gives the resolvent identity in \cref{prop:equilibrium-resolvent}; the same proposition identifies the exact attractive threshold $\alpha=1$ for $H_\Omega-\alpha\mu$. The additive-functional viewpoint is used again for the multiscale measure in \cref{prop:variational-additive-functional}.

The capacitary results use Brownian motion through the event of hitting a set before leaving a cell. We now turn to the values carried by Brownian motion when it exits a ball. The next two results use the second moment of the stopped Feynman--Kac martingale and hence involve the doubled potential. The radial Euclidean model can be treated exactly, while on a Riemannian manifold curvature produces controlled lower-order errors.

\paragraph{B.\ Exact convexity for radial potentials.}

\begin{mainthm}[Radial convexity and minimality of the coefficient]\label{thm:B}
Let $m\ge2$, let $V$ be continuous on $[0,R)$, and let $u\not\equiv0$ be a real-valued solution of $(\tfrac12\Delta+V(|x|))u=0$ in $B_R\subset\mathbb R^m$. Let $g$ be the regular radial solution of $(\tfrac12\Delta+2V)g=0$, $g(0)=1$, and set $R_*:=\sup\{0<r<R:g>0\text{ on }[0,r]\}$. For $0<r<R_*$ define
\[
H_u(r)=\mathbb E_0\!\left[e^{2\int_0^{\tau_r}V(|B_s|)\,ds}
u(B_{\tau_r})^2\right]
=\frac1{g(r)}\fint_{\partial B_r}u^2\,d\sigma .
\]
If $r\mapsto r^2V(r)$ is nondecreasing, then $t\mapsto\log H_u(e^t)$ is convex on $(-\infty,\log R_*)$ for every such solution; this hypothesis already implies $V\ge0$. In dimension $m=2$, the same conclusion holds under $V\ge0$ alone. The minimality statement needs only $V(0)>0$. Let $f_0$ be the regular radial solution of $(\tfrac12\Delta+V)f_0=0$, $f_0(0)=1$. For every $0\le b<2$, if $g_{bV}$ is the regular radial solution of $(\tfrac12\Delta+bV)g_{bV}=0$, $g_{bV}(0)=1$, then $f_0(r)^2/g_{bV}(r)$ fails to be log-convex in $\log r$ near $r=0$. Finally, if $V\ge0$ and $0<\rho\le R_*$, log-convexity on $(-\infty,\log\rho)$ for every nonzero solution is equivalent to
\[
\frac{rg'(r)}{g(r)}
\]
being nonincreasing on $(0,\rho)$.
\end{mainthm}

In logarithmic derivatives, the explicit potential terms cancel; for constant potentials the resulting formulas can be written in terms of Bessel zeros. This exact radial calculation is the model for the local result on a manifold. On a small geodesic ball, curvature introduces lower-order errors, but these errors still lead to an almost-monotone frequency at the wavelength scale.

\paragraph{C.\ Almost-monotonicity and quantitative unique continuation.}

\begin{mainthm}[Almost-monotonicity]
\label{thm:C}
Assume $B(x_0,R)\Subset X\setminus\partial X$ and bounded geometry of order $2$ there. Let $u\not\equiv0$ be a real-valued solution of $(\tfrac12\Delta_g+a)u=0$ there, where $a\ge0$, $R\le r_0(m,\mathrm{bg})$, and $aR^2\le c_0$. Fix $x\in B(x_0,R/2)$ and set
\[
\mathcal N(r):=
\frac{r\int_{\partial B(x,r)}u\,\partial_\nu u\,P_r\,d\sigma}
{\int_{\partial B(x,r)}u^2P_r\,d\sigma},
\]
where $P_r$ is the Poisson kernel at $x$ for $\tfrac12\Delta_g+2a$ on $B(x,r)$. Then $\mathcal N$ is locally absolutely continuous and
\[
\mathcal N'(r)\ge
-C(\kappa+a)r\bigl(1+\mathcal N(r)\bigr)
\]
for almost every $r\in(0,R/2)$. Moreover, for almost every $r\in(0,R/2)$,
\[
\frac r2\frac{d}{dr}\log
\mathbb E_x\!\left[e^{2a\tau_{x,r}}u(B_{\tau_{x,r}})^2\right]
=
\mathcal N(r)+O\bigl((\kappa+a)r^2\bigr).
\]
Here $\kappa$ is the curvature bound from \cref{lem:poisson-comp-small}, taken with $\bar r=R/2$, and all constants depend only on $m$ and the bounded-geometry data.
\end{mainthm}

Almost-monotonicity gives the weighted three-radius inequality in
\cref{prop:remez-from-ost}\textup{(ii)}. Integrating the defining
logarithmic derivative gives the doubling estimates in
part~\textup{(i)}. After a ground-state transform, these supply the
doubling-index input for the Remez theorem of Logunov--Malinnikova
\cite[Lemma~4.2 and Remark~4.3]{LogunovMalinnikova2018}, yielding
part~\textup{(iii)}. Under the larger-ball hypotheses of
\cref{prop:remez-from-ost}, let $N_\sharp$ be as defined there. If $N_\sharp<\infty$, then
\[
\int_{B(x_0,R/2)}u^2\,d\operatorname{vol}_g
\le
\left(\frac C\gamma\right)^{C(1+N_\sharp)}
\int_Eu^2\,d\operatorname{vol}_g
\]
whenever $0<\gamma\le1$ and $E\subset B(x_0,R/2)$ has relative volume at least $\gamma$.

\paragraph{D.\ Averaged boundary variance and its local consequences.}
Assume that $(X,g)$ is a closed connected Riemannian manifold of dimension $m\ge2$, and let $\varphi$ be a nonzero real-valued eigenfunction satisfying $\Delta_g\varphi+\lambda\varphi=0$, with $\lambda>0$.

\cref{thm:B,thm:C} concern weighted second moments. We now return to the singly weighted exit law. After normalisation, its boundary variance measures how far the boundary values of an eigenfunction are from being constant. Small variance therefore forces the normalised eigenfunction inside the ball to be close to the positive solution of the same equation with constant boundary data. \cref{thm:D} shows that this happens at centres carrying almost all of the $\varphi^2$-mass.

More precisely, if $r=c\lambda^{-1/2}$, then, outside a set carrying only $O(c^{2/m})$ of the total $\varphi^2$-mass, division by $\varphi(x)$ makes $\varphi$ $O(c^{(m-1)/m})$-close, in scale-invariant $C^k$ on every fixed smaller ball, to that positive solution. Consequently, the component through $x$ of $\{\varphi(x)\varphi>0\}\cap B(x,r)$ contains $B(x,(1-O(c^{2/m}))r)$, misses only an $O(c^2)$ fraction of the volume of $B(x,r)$, and has first Dirichlet eigenvalue within relative error $O(c^2)$ of that of $B(x,r)$. The starting point is an averaged identity for the weighted boundary variance: its integral equals $Q_c\|\varphi\|_2^2$, up to an error $O_g(c^2\lambda^{-1})\|\varphi\|_2^2$, where $Q_c=c^2/m+O_m(c^4)$.

\begin{mainthm}[Averaged boundary variance and typical nodal rigidity]
\label{thm:D}
Suppose $\lambda\ge1$. There is $c_{\mathrm{rig}}>0$, depending only on bounded geometry, such that the following holds for $0<c\le c_{\mathrm{rig}}$. Put $r=c\lambda^{-1/2}$ and, for $x\in X$, let $\tau_{x,r}$ be the first exit time from $B(x,r)$ and set
\[
h_{x,r}(y)=\mathbb E_y[e^{(\lambda/2)\tau_{x,r}}],\qquad
\nu_{x,r}(A)=
\frac{\mathbb E_x[e^{(\lambda/2)\tau_{x,r}}
\mathbf 1_{\{B_{\tau_{x,r}}\in A\}}]}{h_{x,r}(x)}.
\]
Set
\[
\delta_{x,r}=
\frac{\operatorname{Var}_{\nu_{x,r}}(\varphi)}
{\int\varphi^2\,d\nu_{x,r}},\qquad
\mathcal V_{\lambda,c}(x)=
h_{x,r}(x)^2\operatorname{Var}_{\nu_{x,r}}(\varphi),
\qquad
Q_c=q_m(c^2/2)^2-1.
\]
Then
\[
\left|\int_X\mathcal V_{\lambda,c}\,d\operatorname{vol}_g
-Q_c\|\varphi\|_2^2\right|
\le C_g\frac{c^2}{\lambda}\|\varphi\|_2^2,
\qquad
Q_c=\frac{c^2}{m}+O_m(c^4)\qquad(c\downarrow0).
\]
Moreover, with
\[
G_{\lambda,c}:=
\left\{x\in X:\delta_{x,r}\le Q_c^{(m-1)/m}\right\},
\]
one has
\[
\int_{X\setminus G_{\lambda,c}}\varphi^2\,d\operatorname{vol}_g
\le C_gQ_c^{1/m}\|\varphi\|_2^2.
\]
Every $x\in G_{\lambda,c}$ satisfies $\varphi(x)\ne0$. Let $\Omega_x$ be the component of $\{\varphi(x)\varphi>0\}\cap B(x,r)$ containing $x$. For every fixed $\theta\in(0,1)$ and $k\in\mathbb N_0$,
\[
\sum_{j=0}^kr^j
\left\|\nabla^j\left(
\frac{\varphi}{\varphi(x)}-
\frac{h_{x,r}}{h_{x,r}(x)}
\right)\right\|_{L^\infty(B(x,\theta r))}
\le C_{g,\theta,k}Q_c^{(m-1)/(2m)},
\]
and
\[
\begin{aligned}
B\bigl(x,(1-C_gQ_c^{1/m})r\bigr)&\subset\Omega_x,\\
\frac{\operatorname{vol}_g(B(x,r)\setminus\Omega_x)}
{\operatorname{vol}_g(B(x,r))}&\le C_gQ_c,\\
1\le\frac{\mu_1(\Omega_x)}{\mu_1(B(x,r))}&\le1+C_gQ_c.
\end{aligned}
\]
\end{mainthm}

At a prescribed centre $x_0$, the pointwise input is stability of Cauchy--Schwarz for the first-moment stopping formula. Equality forces $\varphi$ to be a constant multiple of $h_{x_0,r}$; if the normalised boundary variance is at most $\varepsilon$, then the profile error, relative radius loss, and volume and spectral errors are respectively $O(\varepsilon^{1/2})$, $O(\varepsilon^{1/(m-1)})$, and $O(\varepsilon^{m/(m-1)})$. The square-root profile loss is necessary for general local solutions, while the remaining powers are attained by genuine eigenfunctions (\cref{cor:cs-stability,prop:sharpness-capacity,cor:egp-sharpness}). The variance odds are exactly a weighted gradient energy of $\varphi/h_{x_0,r}$, and the underlying local theorem allows a normalised one-sided $L^p$ boundary deficit $\varepsilon_p$, with radius loss $O(\varepsilon_p^{p/(m-1)})$ and volume and spectral errors $O(\varepsilon_p^{mp/(m-1)})$ (\cref{prop:boundary-variance-energy,thm:boundary-deficit-rigidity}). The inclusion and volume conclusions pass to the global same-sign nodal domain; the spectral comparison concerns only the component truncated by $B(x,r)$.

\paragraph{Further first-moment consequences: superlevel components.}
Theorem D is a typical-centre result and permits a small exceptional set. At the maximum of any positive superlevel component, the same first-moment identity gives estimates without either a typicality assumption or a small-variance hypothesis. Let $C$ be a component of $\{\varphi>\ell\}$, $\ell\ge0$, let $M_C=\max_C\varphi$, choose $x_C\in C$ with $\varphi(x_C)=M_C$, put $\vartheta_C=\ell/M_C$, and let $\tau_C$ be the first exit time from $C$. Then
\[
\mathbb E_{x_C}\int_0^{\tau_C}\frac{\varphi(B_s)}{M_C}\,ds
=\frac{2(1-\vartheta_C)}{\lambda},
\qquad
\mathbb E_{x_C}e^{(\lambda/2)\tau_C}=\vartheta_C^{-1}
\quad(\ell>0),
\]
and, for $0<\rho\le r_0(m,\mathrm{bg})$,
\[
\frac{\operatorname{vol}_g(B(x_C,\rho)\setminus C)}
{\operatorname{vol}_g(B(x_C,\rho))}
\le C_g\min\left\{1,\frac{\lambda\rho^2}{1-\vartheta_C}\right\},
\qquad
\operatorname{vol}_g(C)\ge
c_g(1-\vartheta_C)^{m/2}\lambda^{-m/2}.
\]
These conclusions are proved in \cref{thm:superlevel-stopping}. For a bounded Euclidean component with $0<\vartheta_C<1$, Brownian symmetrisation gives the sharp form
\[
|C|\ge\omega_m\left(\frac{2\beta_m(\vartheta_C)}{\lambda}\right)^{m/2},
\qquad q_m(\beta_m(\vartheta_C))=\vartheta_C^{-1},
\]
with equality for concentric superlevel components of the regular radial Helmholtz solution (\cref{cor:euclidean-superlevel}). The volume bounds extend to uniformly elliptic equations with bounded nonnegative potential, and the resulting elliptic theorem applies, without interface regularity, to optimal composite membranes (\cref{thm:elliptic-superlevel}).

\paragraph{The reflected setting: Steklov boundary-to-collar estimates.}
The preceding arguments stop Brownian motion when it reaches the boundary of an interior domain. In the Steklov problem, Brownian motion is instead reflected at the boundary, and boundary local time takes the place of elapsed time. If $u$ is a Steklov eigenfunction with eigenvalue $\sigma$, then $e^{\sigma L_t}u(B_t)$ is a local martingale. For $\sigma\ge1$ on a $C^2$ domain, with $\rho_\sigma=\min\{\rho_*,(4\sigma)^{-1}\}$ and $\rho_*>0$ a sufficiently small fixed collar width, this gives
\[
\sup_{\partial X}|u|\le2\sup_{\{d=\rho_\sigma\}}|u|,
\qquad
\|u\|_{L^2(\partial X)}^2
\le6\|u\|_{L^2(\{d=\rho_\sigma\})}^2,
\]
and the scale $\sigma^{-1}$ is optimal (\cref{thm:steklov_obs}). The martingale identity itself uses only the reflecting process and its boundary local time. On a bounded Lipschitz domain, a quantitative collar estimate for that local time gives the boundary-to-collar bound with the universal constant $c/(c-1)$, $c>1$, at boundary distance $(c\sigma C_{\mathrm{Lip}})^{-1}$, where $C_{\mathrm{Lip}}$ depends only on $m$ and the Lipschitz character (\cref{thm:steklov-boundary-collar-supremum}). On smooth compact manifolds with smooth boundary, Wang--Zhang
\cite[Corollary~1 and Lemma~6]{WangZhang2024} already prove the
corresponding qualitative $L^2$ and $L^\infty$ lower comparisons on
parallel hypersurfaces at distances $t\le \sigma^{-1}$. The contribution
here is the explicit boundary-local-time proof and constants under $C^2$
regularity, together with the bounded-Lipschitz sup-norm estimate. For a
general Lipschitz boundary the Dirichlet-to-Neumann operator need not admit
the classical pseudodifferential description, and the smooth microlocal
methods of \cite{HislopLutzer2001,GalkowskiToth2019,WangZhang2024} do not
cover this setting.

\section{Optional stopping at the wavelength scale}
\label{sec:setup}

We assume that $X$ is complete and let $(B_t)$ be Brownian motion with generator $\tfrac12\Delta_g$. All processes below are stopped on leaving relatively compact domains, so no global stochastic-completeness assumption is needed.

Let $(X,g)$ be an $m$-dimensional Riemannian manifold, let $U\subset X$ be open, and fix an integer $k_0\ge0$. We say that $(X,g)$ has \emph{bounded geometry of order $k_0$ on $U$} if there are $i_0>0$ and $K_0<\infty$ such that
\[
\operatorname{inj}_g(x)\ge i_0\quad(x\in U),\qquad
|\nabla^j\mathrm{Rm}_g|_g\le K_0\quad\text{on }U
\quad(0\le j\le k_0).
\]
For $k_0=0$ these are the injectivity-radius and curvature bounds. On a closed manifold one may take $U=X$; on a noncompact manifold one may work in any region, such as a thick part, where these bounds hold uniformly.

\subsection{The wavelength scale and uniform integrability}
\label{subsec:eigenscale-UI}

For open $D\subset X$, let $\mu_1(D)$ be the first Dirichlet eigenvalue of $-\tfrac12\Delta_g$ on $D$, and let
\[
\tau_D:=\inf\{t\ge0:B_t\notin D\}.
\]

\begin{lemma}[Ball eigenvalue and exponential exit-time moments]
\label{lem:survival-gap}
Assume bounded geometry on $U$. There is $r_{\mathrm{geo}}>0$, depending only on $i_0,K_0,m$, such that, whenever $D=B(x_0,r)\Subset U$ and $0<r\le r_{\mathrm{geo}}$, the following hold:
\begin{enumerate}
\item\label{it:gap} There are $0<c_{\mathrm{gap}}\le C_{\mathrm{gap}}<\infty$, depending only on $i_0,K_0,m$, such that
\[
c_{\mathrm{gap}}r^{-2}\le\mu_1(D)\le C_{\mathrm{gap}}r^{-2}.
\]
\item\label{it:mgf} If $0<a<\mu_1(D)$ and $\theta=a/\mu_1(D)$, then
\[
\sup_{x\in D}\mathbb E_x[e^{a\tau_D}]\le C_\theta<\infty,
\]
where $C_\theta$ depends only on $i_0,K_0,m,\theta$. Moreover, for constants $A_1,C_1>0$ depending only on $i_0,K_0,m$,
\[
ar^2\le A_1\quad\Longrightarrow\quad
\sup_{x\in D}\mathbb E_x[e^{a\tau_D}]\le1+C_1ar^2.
\]
\end{enumerate}
\end{lemma}

\begin{proof}
After rescaling normal coordinates by $r^{-1}$, the metric and volume density are uniformly comparable with their Euclidean counterparts. Comparison of Rayleigh quotients proves \textup{(1)}.

For \textup{(2)}, the Dirichlet heat-kernel bound at $t_0\asymp r^2$, followed by spectral decay, gives $\mathbb P_x(\tau_D>t)\le Ce^{-\mu_1(D)t}$ for $t\ge t_0$. Insert this in
\[
\mathbb E_x[e^{a\tau_D}]
=1+a\int_0^\infty e^{at}\mathbb P_x(\tau_D>t)\,dt
\]
and split at $t_0$. This proves the asserted dependence on $a/\mu_1(D)$; if $ar^2\le A_1$ is small, then $a/\mu_1(D)\le\tfrac12$, and the same calculation gives $1+Car^2$.
\end{proof}

\begin{remark}[Choice of the wavelength-scale radius]
\label{rem:cES-choice}
Fix $\theta_{\mathrm{ES}}\in(0,1)$ and set
\[
c_{\mathrm{ES}}:=
\min\left\{\frac{r_{\mathrm{geo}}}{4},
\frac12\sqrt{\theta_{\mathrm{ES}}c_{\mathrm{gap}}}\right\}.
\]
For $\lambda\ge1$, put $R=c_{\mathrm{ES}}\lambda^{-1/2}$ and assume $B(x_0,2R)\Subset U$. Then
\[
\frac{\lambda}{\mu_1(B(x_0,2R))}
\le\frac{4c_{\mathrm{ES}}^2}{c_{\mathrm{gap}}}
\le\theta_{\mathrm{ES}}<1.
\]
Domain monotonicity therefore gives $\lambda<\mu_1(B(x_0,R))$, while $B(x_0,2R)$ leaves room for local cutoffs. After all small-ball constants $r_0,c_0$ used below have been fixed, we decrease $c_{\mathrm{ES}}$, without changing notation, so that $c_{\mathrm{ES}}\le r_0$ and $c_{\mathrm{ES}}^2/2\le c_0$.
\end{remark}

\subsection{Wavelength-scale stopping identities}
\label{subsec:variance}

\begin{theorem}[Stopping identities at the wavelength scale]
\label{thm:wavelength-stopping-identities}
Let $\lambda\ge1$, let $\varphi\in C^\infty(X)$ satisfy $\Delta_g\varphi+\lambda\varphi=0$, and let $D=B(x_0,R)$ be as in \cref{rem:cES-choice}. Set
\[
M_t:=e^{(\lambda/2)t}\varphi(B_t).
\]
Then $(M_{t\wedge\tau_D})_{t\ge0}$ is a uniformly integrable martingale bounded in $L^2$. For every $x\in D$,
\begin{enumerate}
\item[\textup{(i)}]
\[
\varphi(x)=
\mathbb E_x[e^{(\lambda/2)\tau_D}\varphi(B_{\tau_D})].
\]
\item[\textup{(ii)}]
\begin{equation}\label{eq:variance-identity}
\operatorname{Var}_x(M_{\tau_D})
=\mathbb E_x\!\left[\int_0^{\tau_D}
e^{\lambda s}|\nabla\varphi|_g^2(B_s)\,ds\right]
=\int_DG_D^{(\lambda)}(x,y)|\nabla\varphi(y)|_g^2\,
d\operatorname{vol}_g(y),
\end{equation}
where $G_D^{(\lambda)}$ is the positive kernel of $(-\tfrac12\Delta_D-\lambda)^{-1}$.
\end{enumerate}
More generally, \textup{(i)} holds on any relatively compact ball $D$ for which $\lambda/2<\mu_1(D)$.
\end{theorem}

\begin{proof}
It\^o's formula makes the stopped process a local martingale, while \cref{lem:survival-gap}\textup{(2)} and
\[
|M_{t\wedge\tau_D}|
\le\|\varphi\|_{L^\infty(\overline D)}e^{(\lambda/2)\tau_D},
\]
make it $L^2$-bounded and uniformly integrable; optional stopping proves \textup{(i)}. For the final clause, killed-semigroup ultracontractivity and spectral decay give the required $L^1$ exponential moment whenever $\lambda/2<\mu_1(D)$, and the same argument applies. Finally,
\[
d\langle M\rangle_t=e^{\lambda t}|\nabla\varphi|_g^2(B_t)\,dt
\qquad(t<\tau_D).
\]
The It\^o isometry gives the first equality in \eqref{eq:variance-identity}, and Tonelli identifies the occupation kernel with $\int_0^\infty e^{\lambda s}p_D(s,x,y)\,ds=G_D^{(\lambda)}(x,y)$.
\end{proof}

\section{Exponentially weighted exit distributions on small balls}
\label{sec:toolbox}

The proofs of \cref{thm:C,thm:D} use a small-ball comparison for exponentially weighted exit distributions. The Euclidean formula also appears in the constant-potential calculation in \cref{sec:critical-gauge}. 

Put $\nu=(m-2)/2$. The Euclidean exit-time moment below gives the leading term in the small-ball comparison.

\paragraph{The Euclidean exit-time moment.}
Write $j_{\nu,1}$ for the first positive zero of the Bessel function $J_\nu$. For $0\le\beta<\tfrac12j_{\nu,1}^2$, define
\begin{equation}\label{eq:q_m_def}
q_m(\beta):=
\mathbb E^{\mathbb R^m}_0
\left[e^{\beta\tau_{B(0,1)}}\right].
\end{equation}
Equivalently, $q_m(\beta)$ is the value at the origin of the unique solution
\[
\left(\tfrac12\Delta_{\mathbb R^m}+\beta\right)u=0
\quad\text{in }B(0,1),\qquad
u=1\quad\text{on }\partial B(0,1).
\]

For $m\ge2$, $\nu=(m-2)/2$, and $0<\beta<\tfrac12j_{\nu,1}^2$, uniqueness and rotational invariance make the solution radial. The regular solution, normalised to equal $1$ at $\rho=1$, is
\[
u(\rho)=
\frac{\rho^{-\nu}J_\nu(\sqrt{2\beta}\,\rho)}
{J_\nu(\sqrt{2\beta})}.
\]
Taking $\rho\downarrow0$ gives
\begin{equation}\label{eq:q_m_bessel}
q_m(\beta)=
\frac{(\sqrt{2\beta}/2)^\nu}
{\Gamma(\nu+1)J_\nu(\sqrt{2\beta})}.
\end{equation}
This extends continuously to $q_m(0)=1$. The first two terms of the Bessel series, together with $2(\nu+1)=m$, give
\begin{equation}\label{eq:q_m_small_beta}
q_m(\beta)=1+\frac{\beta}{m}+O_m(\beta^2)
\qquad(\beta\downarrow0).
\end{equation}
In particular, $q_2(\beta)=J_0(\sqrt{2\beta})^{-1}$.

\begin{lemma}[Bounds for exponentially weighted Poisson kernels]
\label{lem:poisson-comp-small}
Let $m\ge2$, and assume bounded geometry of order $2$ on $U$. Set
\[
c_0:=\min\bigl\{\tfrac14c_{\mathrm{gap}},\tfrac18j_{\nu,1}^2\bigr\}.
\]
There are constants $r_0\in(0,r_{\mathrm{geo}}]$ and $0<c\le C<\infty$, depending only on the bounded-geometry data and $m$, such that the following holds. Fix $x\in U$, $a\ge0$, and $0<\bar r\le r_0$ such that $B(x,\bar r)\Subset U$ and $a\bar r^2\le c_0$, and set
\[
\kappa=\sup_{B(x,\bar r)}
\bigl(|\mathrm{Rm}|+r_0|\nabla\mathrm{Rm}|
+r_0^2|\nabla^2\mathrm{Rm}|\bigr).
\]
For $0<r\le\bar r$, put $S_r=\partial B(x,r)$ and $\tau_r=\tau_{B(x,r)}$, and define $P_r$ by
\[
\mathbb E_x\!\left[e^{2a\tau_r}F(B_{\tau_r})\right]
=\int_{S_r}F(\xi)P_r(\xi)\,d\sigma_g(\xi)
\]
for bounded Borel $F$.

\emph{(i) Pointwise bounds.} For every $\xi\in S_r$,
\begin{equation}\label{eq:poisson-comp}
cr^{1-m}\le P_r(\xi)\le Cr^{1-m}.
\end{equation}

\emph{(ii) Comparison with the Euclidean kernel.} Let $S_xX=\{\omega\in T_xX:|\omega|=1\}$, $c_m=|S^{m-1}|^{-1}$, and
\[
\Theta(r,\omega)=
\frac{P_r(\exp_x(r\omega))}
{c_mr^{1-m}q_m(2ar^2)}-1.
\]
The map $r\mapsto\Theta(r,\cdot)$ is locally absolutely continuous in $L^\infty(S_xX)$, and
\begin{equation}\label{eq:poisson-pole-expansion}
\|\Theta(r,\cdot)\|_{C^2(S_xX)}\le C\kappa r^2,\qquad
\|\partial_r\Theta(r,\cdot)\|_{L^\infty(S_xX)}
\le C\kappa r,
\end{equation}
where the second estimate holds for almost every $r$. After decreasing $r_0$, if necessary, $1+\Theta\ge\tfrac12$.

\emph{(ii$'$) Radial and tangential derivatives.} Set $\rho=d(x,\cdot)$, and, for $\xi=\exp_x(r\omega)$, define
\[
\mathcal R_r(\xi)=
\frac d{dr}P_r(\exp_x(r\omega))
+(\Delta_g\rho)(\xi)P_r(\xi).
\]
Then, for almost every $r\in(0,\bar r)$,
\begin{equation}\label{eq:poisson-transport-bounds}
\left\|\frac{\mathcal R_r}{P_r}\right\|_{L^\infty(S_r)}
+\left\|\frac{\nabla_{S_r}P_r}{P_r}\right\|_{L^\infty(S_r)}
\le C(\kappa+a)r.
\end{equation}
For fixed $r$, define $\Phi_r(\exp_x(s\omega))=P_r(\exp_x(r\omega))$, $0<s\le r$. Then
\begin{equation}\label{eq:frozen-weight-calculus}
\Phi_r\asymp r^{1-m},\qquad
\frac{s|\nabla_{S_s}\Phi_r|}{\Phi_r}
+\frac{s^2|\Delta_{S_s}\Phi_r|}{\Phi_r}
\le C(\kappa+a)r^2.
\end{equation}

\emph{(iii) Surface averages.} For real-valued $u\in L^2(S_r)$, set $H_u(x,r)=\int_{S_r}u^2P_r\,d\sigma_g$. Then
\begin{equation}\label{eq:H-comp-surfmean}
\frac{c}{\sigma_g(S_r)}\int_{S_r}u^2\,d\sigma_g
\le H_u(x,r)\le
\frac{C}{\sigma_g(S_r)}\int_{S_r}u^2\,d\sigma_g.
\end{equation}
\end{lemma}

\begin{proof}
The choice of $c_0$ gives $2a<\mu_1(B(x,r))$ and $2ar^2<\tfrac12j_{\nu,1}^2$ for $0<r\le\bar r$, so $P_r$ and $q_m(2ar^2)$ are well defined.

Let $F_r(z)=\exp_x(rz)$ and write $F_r^*g=r^2g_r$ on $B_1$. Normal-coordinate estimates give
\[
\|g_r-\delta\|_{C^4(B_1)}
+r\|\partial_rg_r\|_{C^3(B_1)}
\le C\kappa r^2,
\]
as well as
\[
|g_r-\delta|\le C\kappa r^2|z|^2,
\qquad
|\partial_zg_r|\le C\kappa r^2|z|.
\]
Moreover,
\[
 |r\partial_rg_r|\le C\kappa r^2|z|^2,
 \qquad
 |\partial_z(r\partial_rg_r)|\le C\kappa r^2|z|.
\]
For $0\le\beta\le2c_0$, let $P^{(\beta)}_{g_r}$ be the Poisson kernel at the origin for $\tfrac12\Delta_{g_r}+\beta$ on $B_1$, with respect to the surface measure of $g_r$, and set
\[
\widetilde\Theta(r,\beta,\omega)
=\frac{P^{(\beta)}_{g_r}(0,\omega)}{c_mq_m(\beta)}-1.
\]
Rescaling space and time gives
\[
r^{m-1}P_r(\exp_x(r\omega))
=P^{(\beta)}_{g_r}(0,\omega),
\qquad \beta=2ar^2.
\]

Let $G_{r,\beta}$ and $G_{0,\beta}$ be the positive Dirichlet Green kernels of $(-\tfrac12\Delta_{g_r}-\beta)^{-1}$ and $(-\tfrac12\Delta-\beta)^{-1}$. Put
\[
L_{r,\beta}=-\tfrac12\Delta_{g_r}-\beta,\qquad
L_{0,\beta}=-\tfrac12\Delta-\beta,
\]
and
\[
w_{r,\beta}
=G_{r,\beta}(\cdot,0)-G_{0,\beta}(\cdot,0).
\]
Since $g_r(0)=\delta$ and its first derivatives vanish at the origin, the delta singularities cancel, and
\[
L_{r,\beta}w_{r,\beta}
=(L_{0,\beta}-L_{r,\beta})G_{0,\beta}(\cdot,0),
\qquad
w_{r,\beta}|_{\partial B_1}=0,
\]
where the right-hand side is bounded by $C\kappa r^2|z|^{2-m}$.  Write $R_{r,\gamma}=L_{r,\gamma}^{-1}$ and set $f_{r,\beta}=(L_{0,\beta}-L_{r,\beta})G_{0,\beta}(\cdot,0)$. Thus $w_{r,\beta}=R_{r,\beta}f_{r,\beta}$ in the distributional, equivalently Green-potential, sense. Suppose first that $m\ge3$. The unshifted Green estimate \cite[Theorem~1.1, estimate~(1.8)]{GruterWidman1982} and repeated convolution show that, for every fixed $k\ge1$,
\[
 |R_{r,0}^k f_{r,\beta}(z)|\le C_k\kappa r^2
 \begin{cases}
  1+|z|^{2(k+1)-m},&2(k+1)<m,\\
  1+\log(2/|z|),&2(k+1)=m,\\
  1,&2(k+1)>m.
 \end{cases}
\]
Choose an integer $N>m/2$.  Let $f_{r,\beta}^{(\epsilon)}=\chi_\epsilon f_{r,\beta}$, where $\chi_\epsilon$ is a smooth radial cutoff, equal to $0$ on $B(0,\epsilon)$ and to $1$ outside $B(0,2\epsilon)$.  Finite iteration of the resolvent identity gives
\[
 R_{r,\beta}f_{r,\beta}^{(\epsilon)}
 =\sum_{j=0}^{N-1}\beta^jR_{r,0}^{j+1}f_{r,\beta}^{(\epsilon)}
 +\beta^N R_{r,\beta}R_{r,0}^Nf_{r,\beta}^{(\epsilon)}.
\]
The preceding convolution majorants are independent of $\epsilon$ and give locally uniform convergence off the pole for every term.  Since $2N>m$, the kernel of $R_{r,0}^N$ is uniformly bounded; as $\|f_{r,\beta}^{(\epsilon)}-f_{r,\beta}\|_{L^1(B_1)} \le C\kappa r^2\epsilon^2$, it follows that
\[
 \|R_{r,0}^N(f_{r,\beta}^{(\epsilon)}-f_{r,\beta})\|_\infty\longrightarrow0.
\]
Positivity of the Green kernels supplies the same majorants when $f_{r,\beta}$ changes sign.  Dominated convergence therefore gives the same identity with $f_{r,\beta}$ in place of $f_{r,\beta}^{(\epsilon)}$. Every term in the finite sum is $O(\kappa r^2)$ on a fixed boundary collar, while $\|R_{r,0}^Nf_{r,\beta}\|_{L^\infty(B_1)}\le C\kappa r^2$. Moreover, if $h\in L^\infty(B_1)$, then
\[
 |R_{r,\beta}h(y)|
 \le \|h\|_\infty
 \mathbb E_y^{g_r}\!\left[\int_0^{\tau_{B_1}}e^{\beta t}\,dt\right]
 \le C\|h\|_\infty,
 \]
uniformly for $0\le\beta\le2c_0$.  Indeed, $\mu_1(B_1,g_r)\ge c_{\mathrm{gap}}$ and $\beta\le2c_0\le c_{\mathrm{gap}}/2$, so the last bound follows from the killed-survival estimate used in \cref{lem:survival-gap}.  Thus the remainder has the same collar bound.  The expansion also records the global pole behaviour
\[
 |w_{r,\beta}(z)|\le C\kappa r^2
 \begin{cases}
  1,&m=3,\\
  1+\log(2/|z|),&m=4,\\
  1+|z|^{4-m},&m\ge5.
 \end{cases}
\]

For $m=2$, the right-hand side is bounded, and the uniform $L^\infty$ resolvent estimate gives a global $O(\kappa r^2)$ bound directly.  Put $\rho=|z|$ and, for $0<\rho\le1/4$, set
\[
 \Psi_m(\rho)=
 \begin{cases}
  1,&m=2,3,\\
  1+\log(2/\rho),&m=4,\\
  1+\rho^{4-m},&m\ge5.
 \end{cases}
\]
The global pole bounds and scaled interior estimates on dyadic annuli give, for $j=1,2$ and $|z|=\rho$,
\[
 \rho^j|\nabla^jw_{r,\beta}(z)|
 \le C\kappa r^2\Psi_m(\rho).
\]
The normal-coordinate estimates and the corresponding Euclidean Green bounds also give, with the $r$-derivative taken at fixed $\beta$,
\[
 |\partial_\beta f_{r,\beta}(z)|+|r\partial_rf_{r,\beta}(z)|
 \le C\kappa r^2
 \begin{cases}
  1,&m=2,\\
  1+\rho^{2-m},&m\ge3.
 \end{cases}
\]
In nondivergence form, the coefficients of the second- and first-order parts of $r\partial_rL_{r,\beta}$ are bounded respectively by $C\kappa r^2\rho^2$ and $C\kappa r^2\rho$.  After decreasing $r_0$ so that $\kappa r_0^2\le1$, the annular estimate shows that $(r\partial_rL_{r,\beta})w_{r,\beta}$ obeys the same displayed source bound.

Choose the truncations above with cutoffs independent of $r$ and $\beta$. For the truncated problems, subtracting the equations at two parameter values gives, with the cutoff superscript suppressed,
\[
 \frac{w_{r,\beta'}-w_{r,\beta}}{\beta'-\beta}
 =R_{r,\beta'}\left(
 \frac{f_{r,\beta'}-f_{r,\beta}}{\beta'-\beta}
 +w_{r,\beta}\right)
\]
and
\[
 \frac{w_{s,\beta}-w_{r,\beta}}{s-r}
 =R_{s,\beta}\left(
 \frac{f_{s,\beta}-f_{r,\beta}}{s-r}
 -\frac{L_{s,\beta}-L_{r,\beta}}{s-r}w_{r,\beta}\right).
\]
The preceding source bounds and the finite resolvent iteration control these difference quotients uniformly in the cutoff.  Dominated convergence in their Green representations identifies the limits and gives
\[
 L_{r,\beta}\partial_\beta w_{r,\beta}
 =w_{r,\beta}+\partial_\beta f_{r,\beta},
 \qquad
 L_{r,\beta}(r\partial_rw_{r,\beta})
 =r\partial_rf_{r,\beta}-(r\partial_rL_{r,\beta})w_{r,\beta}.
\]
The right-hand sides have the same admissible pole bound as $f_{r,\beta}$, so the same bootstrap gives uniform collar bounds for both parameter derivatives.

Fix $1/2<\rho_0<\rho_1<1$.  Boundary Schauder estimates on $\mathcal C=\{\rho_1<|z|\le1\}$, using the larger collar $\mathcal C'=\{\rho_0<|z|\le1\}$, give
\[
 \|w_{r,\beta}\|_{C^3(\mathcal C)}
 +\|\partial_\beta w_{r,\beta}\|_{C^1(\mathcal C)}
 +\|r\partial_rw_{r,\beta}\|_{C^1(\mathcal C)}
 \le C\kappa r^2.
\]
By the Gauss lemma, the outward unit normal on $\partial B_1$ is $\partial_\rho$ for both $g_r$ and the Euclidean metric.  Green's formula therefore gives
\[
 P^{(\beta)}_{g_r}(0,\omega)-c_mq_m(\beta)
 =-\tfrac12\partial_\rho w_{r,\beta}(\omega).
\]
Applying the preceding collar estimates to this identity gives
\[
\|\widetilde\Theta\|_{C^2(S^{m-1})}
+r\|\partial_r\widetilde\Theta\|_{L^\infty(S^{m-1})}
+\|\partial_\beta\widetilde\Theta\|_{L^\infty(S^{m-1})}
\le C\kappa r^2,
\]
uniformly for $0\le\beta\le2c_0$, where the $r$-derivative is taken with $\beta$ fixed.  The difference-quotient argument also gives local $C^1$ dependence on $r$. Since $\Theta(r,\omega)=\widetilde\Theta(r,2ar^2,\omega)$,
\[
\|\partial_r\Theta\|_\infty
\le C\kappa r+4ar\,C\kappa r^2
\le C\kappa r.
\]
This proves \eqref{eq:poisson-pole-expansion}. Since $q_m$ is bounded above and below on $[0,2c_0]$, decreasing $r_0$ proves \eqref{eq:poisson-comp}. The $C^2$ bound above, together with the comparison of the metric on $S_s$ with $s^2$ times the round metric, gives \eqref{eq:frozen-weight-calculus}.

At fixed $\omega$,
\[
\frac d{dr}\log P_r(\exp_x(r\omega))
=\frac{1-m}{r}
+4ar\frac{q_m'(\beta)}{q_m(\beta)}
+\frac{\partial_r\Theta}{1+\Theta}.
\]
Since $q_m'/q_m$ is bounded on $[0,2c_0]$ and the normal-coordinate estimates give
\[
\Delta_g\rho=(m-1)r^{-1}+O(\kappa r),
\]
this proves the first estimate in \eqref{eq:poisson-transport-bounds}. The second follows from \eqref{eq:poisson-pole-expansion} and $|\nabla_{S_r}f|\le Cr^{-1}|\nabla_\omega f|$. Finally, $\sigma_g(S_r)\asymp r^{m-1}$, so integrating \eqref{eq:poisson-comp} against $u^2$ proves \eqref{eq:H-comp-surfmean}.
\end{proof}

\begin{proposition}[Poisson kernel on a small geodesic ball]
\label{prop:pole_isotropy_explicit}
Let $m\ge2$, assume bounded geometry of order $2$ on an open set $U\subset X\setminus\partial X$, and suppose $B(x,2r_{\mathrm{geo}})\Subset U$. Set
\[
\kappa=\sup_{B(x,2r_{\mathrm{geo}})}
\bigl(|\mathrm{Rm}|+r_{\mathrm{geo}}|\nabla\mathrm{Rm}|
+r_{\mathrm{geo}}^2|\nabla^2\mathrm{Rm}|\bigr).
\]
For every $0<\beta_{\max}<\tfrac12j_{\nu,1}^2$, there are $r_*\in(0,r_{\mathrm{geo}}]$ and $C<\infty$, depending only on the bounded-geometry data, $m$, and $\beta_{\max}$, such that if $0<r\le r_*$, $\alpha\ge0$, and $\beta=\alpha r^2\le\beta_{\max}$, then $\alpha<\mu_1(B(x,r))$. Writing
\[
S_r=\partial B(x,r),\qquad \tau_r=\tau_{B(x,r)},
\]
the Poisson kernel $P^{(\alpha)}_{x,r}$ at $x$ for $\tfrac12\Delta_g+\alpha$ on $B(x,r)$ satisfies
\begin{equation}\label{eq:poisson_isotropic}
\left|
\frac{\sigma_g(S_r)P^{(\alpha)}_{x,r}(x,\xi)}{q_m(\beta)}-1
\right|\le C\kappa r^2
\qquad(\xi\in S_r).
\end{equation}
Consequently, for every $F\in L^1(S_r)$, $F\ge0$,
\begin{equation}\label{eq:poisson_two_sided}
\begin{aligned}
(1-C\kappa r^2)\frac{q_m(\beta)}{\sigma_g(S_r)}
\int_{S_r}F\,d\sigma_g
&\le\mathbb E_x[e^{\alpha\tau_r}F(B_{\tau_r})]\\
&\le(1+C\kappa r^2)\frac{q_m(\beta)}{\sigma_g(S_r)}
\int_{S_r}F\,d\sigma_g .
\end{aligned}
\end{equation}
\end{proposition}

\begin{proof}
Put $\lambda_0=\tfrac12j_{\nu,1}^2$ and $\eta=\lambda_0-\beta_{\max}>0$. For $T_r(z)=\exp_x(rz)$, write $T_r^*g=r^2g_r$. Normal-coordinate estimates and the minimax principle give
\[
\|g_r-\delta\|_{C^2(B_1)}\le C\kappa r^2,
\qquad
\bigl|r^2\mu_1(B(x,r))-\lambda_0\bigr|\le C\kappa r^2.
\]
After decreasing $r_*$,
\[
r^2\mu_1(B(x,r))\ge\beta_{\max}+\frac{\eta}{2},
\]
so $\alpha<\mu_1(B(x,r))$ and the rescaled operators have a uniform spectral gap for $0\le\beta\le\beta_{\max}$.

The fixed-domain Green-kernel perturbation argument in the proof of \cref{lem:poisson-comp-small}, with $[0,2c_0]$ replaced by $[0,\beta_{\max}]$, therefore gives, for the rescaled Poisson kernel $\Pi_{r,\beta}$,
\[
\left\|
\Pi_{r,\beta}
-\frac{q_m(\beta)}{|S^{m-1}|}
\right\|_{L^\infty(S^{m-1})}
\le C\kappa r^2.
\]
Finally,
\[
P^{(\alpha)}_{x,r}(x,T_r\omega)
=r^{1-m}\Pi_{r,\beta}(\omega),\qquad
\sigma_g(S_r)=r^{m-1}|S^{m-1}|(1+O(\kappa r^2)).
\]
Since $q_m$ is bounded above and below on $[0,\beta_{\max}]$, these formulas prove \eqref{eq:poisson_isotropic}; integration proves \eqref{eq:poisson_two_sided}.
\end{proof}

\section{Spectral inequalities on sets of zero volume}

We work in dimension $m\ge2$. On a ball of radius $R$, the Green kernel has the scale $d(x,y)^{2-m}$ when $m\ge3$ and a logarithmic singularity when $m=2$. In both cases relative capacity at scale $R$ has natural size $R^{m-2}$.

\subsection{The capacitary spectral inequality at scale \texorpdfstring{$R$}{R}}
\label{subsec:capsampling}

Throughout this subsection, $X$ is closed with bounded geometry, $m\ge2$, $H=-\tfrac12\Delta_g$, and $P_{[0,E]}$ denotes its spectral projection. For an open set $\Omega\subset X$, write
$$
\mathcal E_\Omega(v,w)
=\frac12\int_\Omega\langle\nabla v,\nabla w\rangle_g\,d\mathrm{vol}_g,
\qquad \mathcal E_\Omega(v)=\mathcal E_\Omega(v,v),
$$
and omit the subscript when the domain is clear. We prove the estimates for real functions; the complex case follows by applying them to real and imaginary parts.

Choose $0<r_*=r_*(m,\mathrm{bg})\le\tfrac12r_{\mathrm{geo}}$ so that the Poincar\'e, cutoff, Green, covering, and Brownian heat-kernel estimates used below hold uniformly at this scale. Fix $0<R\le r_*$, choose a maximal $R$-separated set $\{x_j\}$, and put
$$
B_j=B(x_j,R),\qquad B_j^*=B(x_j,2R).
$$
Then the $B_j$ cover $X$, and the $B_j^*$ have uniformly bounded overlap.

Let $G_j$ be the positive Dirichlet Green kernel of $-\tfrac12\Delta_g$ on $B_j^*$. If $\nu$ is a finite positive measure of finite Green energy, supported on a compact subset of $B_j^*$, then $G_j\nu\in W^{1,2}_0(B_j^*)$ and
$$
\mathcal E_{B_j^*}(w,G_j\nu)=\int\widetilde w\,d\nu,
\qquad
\mathcal E_{B_j^*}(G_j\nu)
=\iint G_j(y,y')\,d\nu(y)\,d\nu(y'),
$$
for $w\in W^{1,2}_0(B_j^*)$, where $\widetilde w$ is its quasi-continuous representative. Moreover, the local Green estimates give, for $y\ne y'$,
$$
G_j(y,y')\le
\begin{cases}
C d(y,y')^{2-m},&m\ge3,\\
C\bigl(1+\log_+(R/d(y,y'))\bigr),&m=2.
\end{cases}
$$
All integrals below against finite-energy measures use quasi-continuous representatives.

Fix $\Lambda\ge1$. For each $j$, let $K_j\subset B_j$ be nonempty and compact. For a probability measure $\nu$ supported on $K_j$, set
$$
\mathcal I_j(\nu):=\iint G_j(y,y')\,d\nu(y)\,d\nu(y').
$$
In what follows, $\nu_j$ denotes a probability measure supported on $K_j$ and satisfying
$$
\mathcal I_j(\nu_j)\le\Lambda R^{2-m}.
$$

Here $\operatorname{Cap}_{B_j^*}$ is the variational capacity for $\mathcal E_{B_j^*}$. Green-energy duality gives
\[
\operatorname{Cap}_{B_j^*}(K_j)
=\left(\inf_{\nu\in\mathcal P(K_j)}\mathcal I_j(\nu)\right)^{-1},
\]
with the convention $1/\infty=0$. When $\operatorname{Cap}_{B_j^*}(K_j)>0$, the normalised equilibrium measure attains the infimum. Hence such a $\nu_j$ exists exactly when
\[
\operatorname{Cap}_{B_j^*}(K_j)\ge\Lambda^{-1}R^{m-2}.
\]
Every Borel set carrying such a $\nu_j$ is nonpolar. For $m\ge3$, every set of Hausdorff dimension less than $m-2$ is polar; in dimension $2$, nonpolar compact sets of dimension zero exist.

\begin{lemma}[Local spectral inequality]\label{lem:local-sampling}
There is $C=C(m,\mathrm{bg})$ such that every such measure $\nu_j$ and every $u\in W^{1,2}(B_j^*)$ satisfy
$$
\int_{B_j}|u|^2\,d\mathrm{vol}_g
\le C(1+\Lambda)R^2\int_{B_j^*}|\nabla u|^2\,d\mathrm{vol}_g
+CR^m\left|\int\widetilde u\,d\nu_j\right|^2.
$$
\end{lemma}

\begin{proof}
Let $\bar u$ be the average of $u$ on $B_j^*$, put $v=u-\bar u$, and choose $\chi_j\in C_c^\infty(B_j^*)$ equal to $1$ on $B_j$, with $|\nabla\chi_j|\le C/R$. By Poincar\'e's inequality, $\mathcal E_{B_j^*}(\chi_jv)\le C\int_{B_j^*}|\nabla u|^2$. For $U_j=G_j\nu_j$, since $\nu_j$ is supported where $\chi_j=1$,
$$
\left|\int\widetilde u\,d\nu_j-\bar u\right|
\le\mathcal E_{B_j^*}(\chi_jv)^{1/2}\mathcal I_j(\nu_j)^{1/2},
\qquad
\left|\int\widetilde u\,d\nu_j-\bar u\right|^2
\le C\Lambda R^{2-m}\int_{B_j^*}|\nabla u|^2\,d\mathrm{vol}_g.
$$
Together with
$$
\int_{B_j}|u|^2\,d\mathrm{vol}_g
\le CR^2\int_{B_j^*}|\nabla u|^2\,d\mathrm{vol}_g+CR^m|\bar u|^2,
$$
and $|\bar u|\le|\int\widetilde u\,d\nu_j| +|\int\widetilde u\,d\nu_j-\bar u|$, this proves the lemma.
\end{proof}

\begin{theorem}[Capacitary spectral inequality]\label{thm:capacity-sampling}
There are $c_0,c_1>0$, depending only on $m$ and the bounded-geometry constants, with the following property. Let $0<R\le r_*$ and $\Lambda\ge1$, and for every $j$ let $\nu_j$ be a probability measure supported on a compact set $K_j\subset B_j$ and satisfying
$$
\mathcal I_j(\nu_j)\le\Lambda R^{2-m}.
$$
If
$$
0\le E\le\frac{c_1}{1+\Lambda}R^{-2},
$$
then every $f\in\operatorname{Ran}P_{[0,E]}(H)$ satisfies
$$
\sum_jR^m\left|\int f\,d\nu_j\right|^2
\ge c_0\|f\|_{L^2(X)}^2,
\qquad
R^m\int_X|f|^2\,d\nu_R\ge c_0\|f\|_{L^2(X)}^2,
$$
where $\nu_R=\sum_j\nu_j$. The constant $c_0$ is independent of $\Lambda$; only the allowed energy range depends on $\Lambda$.
\end{theorem}

\begin{proof}
Summing \cref{lem:local-sampling} with $u=f$, and using the covering and bounded-overlap properties, we obtain
$$
\|f\|_2^2
\le C(1+\Lambda)R^2\|\nabla f\|_2^2
+C\sum_jR^m\left|\int f\,d\nu_j\right|^2.
$$
Since $X$ is closed and the spectral interval is bounded, $f$ is smooth, so its quasi-continuous representative is $f$. Moreover,
$$
\frac12\|\nabla f\|_2^2=\langle Hf,f\rangle\le E\|f\|_2^2.
$$
We choose $c_1$ so that $2Cc_1\le\tfrac12$. The gradient term can then be absorbed into the left-hand side, proving the first inequality with $c_0=(2C)^{-1}$. The second follows from
$$
\left|\int f\,d\nu_j\right|^2\le\int|f|^2\,d\nu_j,
$$
since every $\nu_j$ is a probability measure.
\end{proof}

\begin{corollary}[Equilibrium measures]\label{cor:zero-volume}
Suppose every cell contains a compact set $K_j\subset B_j$ such that
$$
\operatorname{Cap}_{B_j^*}(K_j)\ge\kappa R^{m-2},
\qquad 0<\kappa\le1,
$$
and let $\mu_j$ be its equilibrium measure in $B_j^*$. Set $\mu_R=\sum_j\mu_j$. There is $c=c(m,\mathrm{bg})>0$ such that, if $0\le E\le c\kappa R^{-2}$ and $f\in\operatorname{Ran}P_{[0,E]}(H)$, then
$$
R^2\int_X|f|^2\,d\mu_R\ge c\kappa\|f\|_{L^2(X)}^2.
$$
This includes the Ahlfors-regular examples of every dimension $s\in(m-2,m]$ described below; those with $s<m$ have zero volume.
\end{corollary}

\begin{proof}
Put
$$
\nu_j=\frac{\mu_j}{\operatorname{Cap}_{B_j^*}(K_j)}.
$$
Then $\nu_j$ is a probability measure and
$$
\mathcal I_j(\nu_j)
=\operatorname{Cap}_{B_j^*}(K_j)^{-1}
\le\kappa^{-1}R^{2-m}.
$$
Thus \cref{thm:capacity-sampling} applies with $\Lambda=\kappa^{-1}$ whenever $E\le\tfrac12c_1\kappa R^{-2}$. Its second inequality gives
$$
\begin{aligned}
c_0\|f\|_2^2
&\le R^m\sum_j\int|f|^2\,d\nu_j\\
&\le\kappa^{-1}R^2\int_X|f|^2\,d\mu_R.
\end{aligned}
$$
Taking $c\le\min\{c_0,c_1/2\}$ proves the result.
\end{proof}

\paragraph{A small-values consequence.}
There is $c=c(m,\mathrm{bg})>0$ such that, if $0<\kappa\le1$, $\delta\ge0$, $0\ne f\in\operatorname{Ran}P_{[0,E]}(H)$ with $E\le c\kappa R^{-2}$, and every $B_j$ contains a compact set $K_j\subset\{|f|\le\delta\,\operatorname{vol}_g(X)^{-1/2}\|f\|_2\}$ with $\operatorname{Cap}_{B_j^*}(K_j)\ge\kappa R^{m-2}$, then $\delta\ge c\sqrt\kappa$.
Indeed, if $\mu_R$ is the sum of the corresponding equilibrium measures, capacity monotonicity and the cardinality of the net give $R^2\mu_R(X)\le C\operatorname{vol}_g(X)$.
Thus $R^2\int_X|f|^2\,d\mu_R\le C\delta^2\|f\|_2^2$, whereas \cref{cor:zero-volume} gives the lower bound $c\kappa\|f\|_2^2$.

\paragraph{A Brownian proof by equilibrium killing.}
We next give a Brownian proof. It has three steps: identify the law of the equilibrium additive functional, obtain a one-cell survival estimate, and iterate that estimate to produce a global spectral gap. Projection to the low-energy space then gives the original spectral inequality. For a closed set $K$, let $\tau_K:=\inf\{t\ge0:B_t\in K\}$ be its hitting time; $\tau_\Omega$ keeps its meaning as the exit time from $\Omega$. Let $K\Subset\Omega\subset X$, with $\overline\Omega\subsetneq X$, and assume that the part form $(\mathcal E_\Omega,W^{1,2}_0(\Omega))$ is transient. Write $e_{K,\Omega}$ for the equilibrium potential and $\mu_{K,\Omega}$ for its equilibrium measure. Recall that $e_{K,\Omega}\in W^{1,2}_0(\Omega)$, equals $1$ quasi-everywhere on $K$, is harmonic off $K$, and
$$
e_{K,\Omega}(y)=\mathbb P_y(\tau_K<\tau_\Omega)
$$
quasi-everywhere. Moreover,
$$
\mathcal E_\Omega(e_{K,\Omega},\psi)
=\int\widetilde\psi\,d\mu_{K,\Omega},
\qquad \psi\in W^{1,2}_0(\Omega),
$$
and
$$
\mu_{K,\Omega}(K)=\mathcal E_\Omega(e_{K,\Omega})
=\operatorname{Cap}_\Omega(K).
$$
Let $A_t^{K,\Omega}$ be the positive continuous additive functional of Brownian motion killed on leaving $\Omega$ whose Revuz measure is $\mu_{K,\Omega}$ \cite[Theorem~5.1.4]{FOT11}.

\begin{lemma}[Law of the equilibrium additive functional]
\label{lem:equilibrium-additive-functional-law}
For quasi-every $y\in\Omega$,
$$
\operatorname{Law}_y\left(A^{K,\Omega}_{\tau_\Omega}\right)
=(1-e_{K,\Omega}(y))\delta_0
+e_{K,\Omega}(y)e^{-s}\mathbf 1_{\{s>0\}}\,ds.
$$
Consequently, for every $\beta\ge0$,
$$
\mathbb E_y\left[e^{-\beta A^{K,\Omega}_{\tau_\Omega}}\right]
=1-\frac{\beta}{1+\beta}e_{K,\Omega}(y).
$$
Thus $\beta=1$ removes one half of the probability of hitting $K$ before leaving $\Omega$, while $\beta\to\infty$ gives the probability of leaving without hitting $K$.
\end{lemma}

\begin{proof}
Let $G_\Omega$ be the Dirichlet Green operator and put $\mu=\mu_{K,\Omega}$ and $e=e_{K,\Omega}$. The Revuz potential formula \cite[Lemma~5.1.3]{FOT11}, followed by monotone convergence as $\alpha\downarrow0$ for the killed process, gives
\[
\mathbb E_y\int_0^{\tau_\Omega}f(B_t)\,dA_t^{K,\Omega}
=G_\Omega(f\mu)(y)
\]
for bounded nonnegative Borel $f$ and quasi-every $y\in\Omega$.

Set
\[
h_1=G_\Omega\mu=e,
\qquad
h_{n+1}=G_\Omega(\widetilde h_n\mu).
\]
Since $\widetilde e=1$ $\mu$-almost everywhere, induction gives $h_n=e$ for every $n\ge1$. Applying the Markov property to the ordered-time integral therefore gives
\[
\mathbb E_y\left[
\left(A^{K,\Omega}_{\tau_\Omega}\right)^n\right]
=n!h_n(y)=n!e(y),
\qquad n\ge1.
\]
Hence, for $0\le\gamma<1$, monotone convergence yields
\[
\mathbb E_y\left[e^{\gamma A^{K,\Omega}_{\tau_\Omega}}\right]
=1+\sum_{n=1}^\infty\gamma^ne(y)
=1+e(y)\frac{\gamma}{1-\gamma}.
\]
This is the moment generating function of
\[
(1-e(y))\delta_0+e(y)e^{-s}\mathbf1_{\{s>0\}}\,ds
\]
and, being finite in a neighbourhood of zero, determines the law. Taking its Laplace transform gives
\[
\mathbb E_y\left[e^{-\beta A^{K,\Omega}_{\tau_\Omega}}\right]
=1-\frac{\beta}{1+\beta}e(y),
\qquad \beta\ge0.
\]
\end{proof}

\begin{proposition}[Resolvent and critical coupling for the equilibrium measure]
\label{prop:equilibrium-resolvent}
Let $H_\Omega$ be the Dirichlet operator on $\Omega$, and write $\mu=\mu_{K,\Omega}$, $e=e_{K,\Omega}$, and $A=A^{K,\Omega}$. For $\beta\ge0$ and quasi-every $y\in\Omega$,
$$
\mathbb E_y\left[\int_0^{\tau_\Omega}e^{-\beta A_t}\,dA_t\right]
=\frac{e(y)}{1+\beta}.
$$
Equivalently, in the form sense,
$$
(H_\Omega+\beta\mu)^{-1}\mu=\frac{e}{1+\beta}.
$$
Moreover, the Green trace
$$
Pf=\left.\widetilde{G_\Omega(f\mu)}\right|_K
\quad\mu\text{-almost everywhere}
$$
extends from bounded functions to a positive self-adjoint contraction on $L^2(\mu)$ and satisfies $P1=1$; in particular, $P$ is Markov. If $\operatorname{Cap}_\Omega(K)>0$, the coupling constant $1$ is also exact:
$$
\text{for }\alpha\ge0,\qquad
\mathcal E_\Omega(u)-\alpha\int|\widetilde u|^2\,d\mu\ge0
\quad\text{for every }u\in W^{1,2}_0(\Omega)
\quad\Longleftrightarrow\quad \alpha\le1.
$$
At $\alpha=1$, equality is attained by $e$, which satisfies $(H_\Omega-\mu)e=0$ weakly; for $\alpha>1$, the same function gives a negative Rayleigh quotient.
\end{proposition}

\begin{proof}
For $\beta>0$, integration with respect to the continuous increasing function $A_t$ and \cref{lem:equilibrium-additive-functional-law} give
$$
\mathbb E_y\left[\int_0^{\tau_\Omega}e^{-\beta A_t}\,dA_t\right]
=\frac{1-\mathbb E_y[e^{-\beta A_{\tau_\Omega}}]}{\beta}
=\frac{e(y)}{1+\beta}.
$$
The case $\beta=0$ is the first moment in the proof of that lemma.

For bounded real $\psi$, Picone's identity applied to $e+\varepsilon$, followed by $\varepsilon\downarrow0$, gives
$$
\int|\widetilde\psi|^2\,d\mu\le\mathcal E_\Omega(\psi),
\qquad \psi\in W^{1,2}_0(\Omega)\cap L^\infty(\Omega).
$$
Truncation and real and imaginary parts give the inequality for every $\psi\in W^{1,2}_0(\Omega)$. Thus the trace into $L^2(\mu)$ is a contraction for the energy norm, and the form of $H_\Omega+\beta\mu$ is closed on $W^{1,2}_0(\Omega)$. Denote the trace by $T$. Since $\widetilde e=1$ $\mu$-almost everywhere, $e/(1+\beta)$ satisfies
$$
\mathcal E_\Omega(v,\psi)
+\beta\int\widetilde v\,\widetilde\psi\,d\mu
=\int\widetilde\psi\,d\mu,
\qquad \psi\in W^{1,2}_0(\Omega),
$$
and uniqueness proves the resolvent identity.

With the energy inner product, the Green identity gives $G_\Omega(f\mu)=T^*f$, so $P=TT^*$ is a self-adjoint contraction on $L^2(\mu)$. If $f\ge0$, positivity of the Dirichlet Green operator gives $G_\Omega(f\mu)\ge0$ quasi-everywhere, so $P$ is positivity preserving. Finally, $G_\Omega\mu=e$ and $\widetilde e=1$ $\mu$-almost everywhere give $P1=1$. Hence $0\le f\le1$ implies $0\le Pf\le1$, and $P$ is Markov. The trace inequality proves the coupling assertion for $0\le\alpha\le1$. Conversely,
$$
\mathcal E_\Omega(e)-\alpha\int|\widetilde e|^2\,d\mu
=(1-\alpha)\operatorname{Cap}_\Omega(K),
$$
and the equilibrium identity gives $(H_\Omega-\mu)e=0$ weakly.
\end{proof}

\begin{remark}[Relation to earlier work]
Measure-valued limits occur at the endpoint of eigenvalue optimisation over Schr\"odinger potentials \cite[Chapter~8]{Henrot2006}. With the normalisation above, the attractive coupling at that endpoint is explicit and has the equilibrium potential as a null state. For a closed smooth surface $\Sigma\subset\mathbb R^3$ with electrostatic capacity $C_{\rm el}(\Sigma)$, normalised so that a sphere of radius $R$ has capacity $R$, and unit equilibrium charge density $\varrho$, Exner--Fraas show that, in the $-\Delta$ normalisation, the coupling $4\pi C_{\rm el}(\Sigma)\varrho$ is critical \cite[Theorem~4.1\textup{(a)}]{ExnerFraas2009}. Trace inequalities for regular Dirichlet forms are studied in \cite{BenAmor2004}; large-coupling resolvent estimates, including equilibrium measures, in \cite{BrascheDemuth2005,BenAmorBrasche2008}. The resolvent identity used here comes from the explicit law of the equilibrium additive functional in \cref{lem:equilibrium-additive-functional-law}. It requires neither compactness of the trace operator nor discreteness of the coupling spectrum.
\end{remark}

\begin{lemma}[A capacity estimate on one ball]
\label{lem:one-step-capacity-killing}
There are $c,C>0$, depending only on $m$ and the bounded-geometry constants, such that if $0<R\le r_*$, $B=B(x,R)$, $B^*=B(x,2R)$, and $K\subset B$ is compact with
$$
\operatorname{Cap}_{B^*}(K)\ge\kappa R^{m-2},
\qquad \kappa>0,
$$
then, writing $A_t^K=A_t^{K,B^*}$, for every $\beta>0$ and quasi-every $y\in B$,
$$
\mathbb E_y\left[e^{-\beta A^K_{\tau_{B^*}}}\right]
\le1-c\frac{\beta}{1+\beta}\kappa,
\qquad
\mathbb E_y\left[\int_0^{\tau_{B^*}}e^{-\beta A^K_t}\,dt\right]\le CR^2.
$$
\end{lemma}

\begin{proof}
Taking $b=c_{\mathrm{mgf}}R^{-2}$ with $c_{\mathrm{mgf}}>0$ sufficiently small in \cref{lem:survival-gap} for $B^*$, and using $e^{b\tau_{B^*}}\ge1+b\tau_{B^*}$, gives
$$
\mathbb E_y\left[\int_0^{\tau_{B^*}}e^{-\beta A^K_t}\,dt\right]
\le\mathbb E_y[\tau_{B^*}]\le CR^2.
$$
For the first inequality, \cref{lem:equilibrium-additive-functional-law} gives
$$
1-\mathbb E_y\left[e^{-\beta A^K_{\tau_{B^*}}}\right]
=\frac{\beta}{1+\beta}e_{K,B^*}(y).
$$
The local Green lower bound and a Harnack chain of bounded length inside $B^*$ yield
$$
G_{B^*}(y,z)\ge cR^{2-m},
\qquad y,z\in B,\quad y\ne z,
$$
also when $m=2$, where the right-hand side is a positive constant. The equilibrium measure charges no polar sets, and points are polar for $m\ge2$; hence $\mu_{K,B^*}$ has no atoms. Therefore, for quasi-every $y\in B$,
$$
e_{K,B^*}(y)
=\int_KG_{B^*}(y,z)\,d\mu_{K,B^*}(z)
\ge cR^{2-m}\operatorname{Cap}_{B^*}(K)
\ge c\kappa.
$$
Finally, a cutoff equal to $1$ on $B$ shows that $\operatorname{Cap}_{B^*}(\overline B)\le CR^{m-2}$. Thus $\kappa\le C$, and $c$ may be decreased so that $c\kappa\le\tfrac12$.
\end{proof}

\begin{proposition}[Spectral gap from Brownian killing by equilibrium measures]
\label{prop:equilibrium-killing-spectral-gap}
In the setting of \cref{cor:zero-volume}, let $\mu_R=\sum_j\mu_j$. For every $\beta>0$, let $H+\beta\mu_R$ denote the operator associated with the closed form
$$
\mathfrak h_{\beta\mu_R}[u]=\mathcal E(u)+\beta\int_X|\widetilde u|^2\,d\mu_R,
\qquad
\mathcal D(\mathfrak h_{\beta\mu_R})
=\{u\in W^{1,2}(X):\widetilde u\in L^2(\mu_R)\}.
$$
Then
$$
\inf\sigma(H+\beta\mu_R)
\ge c\frac{\beta}{1+\beta}\kappa R^{-2},
$$
where $c=c(m,\mathrm{bg})>0$ is independent of $\beta$.
\end{proposition}

\begin{proof}
Fix $\beta>0$.
The net is finite, and each $\mu_j$ is a finite smooth measure. Regard $\mu_j$ as a measure on $X$ by extension by zero, and let $A^{(j)}$ be its positive continuous additive functional. Outside a common properly exceptional set,
$$
A=\sum_jA^{(j)}
$$
is the additive functional with Revuz measure $\mu_R$. The form above is closed, and, for bounded Borel $f$,
$$
S_t^{(\beta)}f(x):=\mathbb E_x\left[e^{-\beta A_t}f(B_t)\right]
=e^{-t(H+\beta\mu_R)}f(x)
\quad\text{for quasi-every }x
$$
\cite[Lemma~6.1.1 and Theorem~6.1.1]{FOT11}.

For $T>0$, define
$$
u_T(x)=\int_0^T S_t^{(\beta)}1(x)\,dt
=\mathbb E_x\left[\int_0^T e^{-\beta A_t}\,dt\right].
$$
Let $M_T\le T$ be the quasi-essential supremum of $u_T$. For quasi-every $x$, choose $j$ with $x\in B_j$ and put $\sigma=\tau_{B_j^*}$. By the strong Markov property,
$$
\begin{aligned}
u_T(x)
&=\mathbb E_x\left[\int_0^{\sigma\wedge T}e^{-\beta A_t}\,dt\right]
+\mathbb E_x\left[e^{-\beta A_\sigma}u_{T-\sigma}(B_\sigma);\,\sigma<T\right]\\
&\le CR^2+\left(1-c\frac{\beta}{1+\beta}\kappa\right)M_T.
\end{aligned}
$$
Indeed, $A_t\ge A_t^{(j)}$ for $t\le\sigma$, and $A^{(j)}_{t\wedge\sigma}$ agrees with $A^{K_j,B_j^*}_{t\wedge\sigma}$ for the killed process. Hence both bounds follow from \cref{lem:one-step-capacity-killing}. Taking the quasi-essential supremum and rearranging gives
$$
M_T\le C\frac{1+\beta}{\beta\kappa}R^2.
$$
Put $M=C(1+\beta)(\beta\kappa)^{-1}R^2$ and $Q_T=\int_0^T S_t^{(\beta)}\,dt$. Positivity gives $\|Q_T\|_{L^\infty\to L^\infty}\le M$, and self-adjointness gives $\|Q_T\|_{L^1\to L^1}\le M$. Hence $\|Q_T\|_{L^2\to L^2}\le M$. By spectral calculus,
$$
\sup_{\lambda\in\sigma(H+\beta\mu_R)}
\int_0^Te^{-t\lambda}\,dt\le M.
$$
Letting $T\to\infty$ gives $\inf\sigma(H+\beta\mu_R)\ge M^{-1}\ge c\frac{\beta}{1+\beta}\kappa R^{-2}$.
\end{proof}

\paragraph{Projection to low energies.}
Let $\lambda_R:=\inf\sigma(H+\mu_R)\ge c_*\kappa R^{-2}$ by \cref{prop:equilibrium-killing-spectral-gap}, and suppose $E\le\tfrac12c_*\kappa R^{-2}$. Every $f\in\operatorname{Ran}P_{[0,E]}(H)$ is smooth and therefore belongs to the form domain of $H+\mu_R$. Hence
$$
\int_X|f|^2\,d\mu_R
\ge(\lambda_R-E)\|f\|_2^2
\ge\frac{c_*}{2}\kappa R^{-2}\|f\|_2^2,
$$
which recovers \cref{cor:zero-volume}. This is the same variational projection step as in \cite[Theorem~1.1]{BLS2011}, here applied directly at form level to the measure perturbation.

\paragraph{Ahlfors-regular examples.}
Let $m\ge2$ and $m-2<s\le m$. Let $S\subset X$ be closed and carry a positive Borel measure $\sigma$ such that
$$
c_s\rho^s\le\sigma\bigl(S\cap B(y,\rho)\bigr)\le C_s\rho^s,
\qquad y\in S,\quad 0<\rho\le r_0.
$$
Suppose $S$ is $\delta$-dense, where $0<\delta\le(r_*\wedge r_0)/4$. For $R\in[4\delta,r_*\wedge r_0]$, choose $y_j\in S$ with $d(x_j,y_j)\le\delta$ and set
$$
K_j=S\cap\overline{B(y_j,R/2)},
\qquad
\nu_j=\frac{\sigma|_{K_j}}{\sigma(K_j)}.
$$
Then
$$
\mathcal I_j(\nu_j)\le\Lambda_sR^{2-m},
\qquad
\Lambda_s=\frac{C(m,\mathrm{bg})C_s}
{c_s\bigl(s-(m-2)\bigr)}.
$$
Consequently, if
$$
E\le c(m,\mathrm{bg})\frac{c_s}{C_s}
\bigl(s-(m-2)\bigr)R^{-2},
$$
then every $f\in\operatorname{Ran}P_{[0,E]}(H)$ satisfies
$$
R^{m-s}\int_S|f|^2\,d\sigma
\ge c(m,\mathrm{bg})c_s\|f\|_{L^2(X)}^2.
$$
In particular, if
$$
E\le c(m,\mathrm{bg})\frac{c_s}{C_s}
\bigl(s-(m-2)\bigr)\delta^{-2},
$$
then every $f\in\operatorname{Ran}P_{[0,E]}(H)$ satisfies
$$
\|f\|_{L^2(X)}^2
\le\frac{C(m,\mathrm{bg})}{c_s}\delta^{m-s}
\int_S|f|^2\,d\sigma.
$$
Whenever $R\asymp E^{-1/2}$ lies in the stated range, it gives $R^{m-s}\asymp E^{-(m-s)/2}$. As $s\downarrow m-2$, $\Lambda_s$ diverges and the energy window obtained from the Green-energy estimate shrinks.
Indeed, by $\delta$-density, $K_j\Subset B_j$ and
$$
\sigma(K_j)\ge c_s(R/2)^s.
$$
Shell integration of the Green bound, logarithmic when $m=2$, gives in both cases
$$
\int_{K_j}G_j(y,y')\,d\sigma(y')
\le\frac{C(m,\mathrm{bg})C_s}{s-(m-2)}R^{s+2-m}.
$$
Consequently,
$$
\mathcal I_j(\nu_j)
\le\frac{C(m,\mathrm{bg})C_s}
{c_s\bigl(s-(m-2)\bigr)}R^{2-m}.
$$
The energy window and, by bounded overlap,
$$
c_0\|f\|_2^2
\le R^m\sum_j\int|f|^2\,d\nu_j
\le\frac{C(m,\mathrm{bg})}{c_s}R^{m-s}
\int_S|f|^2\,d\sigma,
$$
follow from \cref{thm:capacity-sampling}.

\begin{remark}[Sharpness]\label{rem:sampling-sharp}
The powers of $R$ in the energy window, the capacity scale, and the normalisations in \cref{thm:capacity-sampling} are sharp when the relative capacity is bounded below. \emph{Energy:} on $\mathbb T^m=(\mathbb R/2\pi\mathbb Z)^m$, the function $f_N(x)=\sin(Nx_1)$ has eigenvalue $N^2/2$. For large $N$, set $R=\pi/N$. Every $B_j$ contains a nodal hypersurface disk of radius comparable to $R$, whose capacity relative to $B_j^*$ is at least $cR^{m-2}$. The sum of their equilibrium measures is supported where $f_N=0$, so $\int_Xf_N^2\,d\mu_R=0$. Thus the scale $R^{-2}$ is optimal when the relative capacity is bounded below. \emph{Capacity:} the observation measures have finite Green energy, so every Borel set carrying one of them is nonpolar. For $m\ge3$, every compact set of dimension less than $m-2$ is polar. In dimension $2$, \cref{lem:critical-prototype} attains the endpoint on a nonpolar compact set of dimension zero. \emph{Normalisation:} taking $f=1$ forces the factor $R^m$ for probability measures, since $\sum_jR^m\asymp\mathrm{vol}_g(X)$. If $\mu_j(K_j)\simeq\kappa R^{m-2}$, then $R^2\mu_j(K_j)\simeq\kappa R^m$.

The dimension restriction comes from finite Green energy and the bounded trace $H^1(X)\to L^2(\mu)$; it does not apply to general sampling measures without these requirements \cite{OrtegaCerdaPridhnani2012,FilbirMhaskar2011}. Likewise, no optimality is claimed for the displayed $\Lambda$-dependence: on each fixed manifold, the theorem of Filbir and Mhaskar gives an energy window $E\le c_XR^{-2}$ independently of $\Lambda$; see \cref{rem:concurrent-comparison}.
\end{remark}

\subsection{Heat observability and null control at the critical dimension}
\label{subsec:heat-control}

Repeating at every dyadic scale the construction obtained by killing Brownian motion with equilibrium measures gives a finite observation measure carried by a dense zero-volume set of Hausdorff dimension $m-2$ and with support $X$. Here $X$ is closed, $m\ge2$, $H=-\tfrac12\Delta_g$, $R_0\le r_*$, $T_0=R_0^2$, and $R_k=2^{-k}R_0$.

\begin{lemma}[A compact set of dimension $m-2$]
\label{lem:critical-prototype}
For every $m\ge2$ there is a compact $K_*\subset B_{\mathbb R^m}(0,1/4)$ such that
$$
\dim_{\mathrm H}K_*=m-2,
\qquad
\operatorname{Cap}_{B(0,1)}(K_*)>0.
$$
Moreover, $\mathcal H^{m-2}(K_*)=\infty$.
\end{lemma}

\begin{proof}
Put $\ell_n=e^{-n^2}$. Starting from one interval, replace every level-$n$ interval by two level-$(n+1)$ intervals of length $\ell_{n+1}$ placed at its endpoints, and give each level-$n$ interval mass $2^{-n}$. Embed the resulting Cantor set $C_0$ in $\mathbb R^2$, and denote its probability measure by $\nu_0$. If two independent points have last common level $n$, which has probability $2^{-(n+1)}$, then their distance is at least $c\ell_n$. Hence
\[
\iint\log_+\frac1{|z-z'|}\,d\nu_0(z)\,d\nu_0(z')
\le C+\sum_{n\ge0}2^{-n}n^2<\infty.
\]
Moreover, if $\ell_n\le\varepsilon<\ell_{n-1}$, then the covering number satisfies $\mathcal N(C_0,\varepsilon)\le2^n$, and therefore $\overline{\dim}_{\mathrm B}C_0=0$.

For $m=2$, put $\widetilde K=C_0$. For $m\ge3$, put $s=m-2$ and $\widetilde K=[-1,1]^s\times C_0$. For $a>0$,
\[
\iint_{[-1,1]^s\times[-1,1]^s}
\frac{dx\,dx'}{\bigl(|x-x'|^2+a^2\bigr)^{s/2}}
\le C_s\left(1+\log_+\frac1a\right).
\]
Thus normalised Lebesgue measure times $\nu_0$ has finite $s$-energy. Projection onto $[-1,1]^s$ gives $\dim_{\mathrm H}\widetilde K\ge s$, while the product covering estimate and $\overline{\dim}_{\mathrm B}C_0=0$ give the reverse inequality.

Dilate $\widetilde K$ into $B(0,1/4)$ and call the resulting set $K_*$. The preceding energy estimates, the local Green bounds, and Green-energy duality give
\[
\operatorname{Cap}_{B(0,1)}(K_*)>0.
\]
When $m=2$, $C_0$ is infinite, so $\mathcal H^0(K_*)=\infty$. When $m\ge3$, every slice $[-1,1]^s\times\{z\}$ has the same positive $\mathcal H^s$-measure; finite unions of distinct slices therefore show that $\mathcal H^s(K_*)=\infty$.
\end{proof}

\begin{definition}[Measures at all scales]\label{def:multiscale-measure}
For each $k\ge0$, let $\{x^{(k)}_j\}_j$ be a maximal $R_k$-separated set and write
$$
B_{k,j}=B(x^{(k)}_j,R_k),
\qquad B^*_{k,j}=B(x^{(k)}_j,2R_k).
$$
Fix $K_*$ as in \cref{lem:critical-prototype}, choose linear isometries $A_{k,j}:\mathbb R^m\to T_{x^{(k)}_j}X$, and set
$$
K_{k,j}=\exp_{x^{(k)}_j}(R_kA_{k,j}K_*)
\subset B(x^{(k)}_j,R_k/4).
$$
Rescaling normal coordinates distorts Dirichlet energies by uniformly bounded factors, so the Euclidean capacity scaling of $K_*$ gives
$$
cR_k^{m-2}\le
\operatorname{Cap}_{B^*_{k,j}}(K_{k,j})
\le CR_k^{m-2},
$$
with constants independent of $k,j$. Let $\mu_{k,j}$ be the equilibrium measure of $K_{k,j}$ in $B^*_{k,j}$. For $\theta>0$, define
$$
\eta_k=R_k^2\sum_j\mu_{k,j},
\qquad
\mu_\theta=c_\theta\sum_{k\ge0}(k+1)^{-1-\theta}\eta_k,
\qquad
c_\theta^{-1}=\sum_{k\ge0}(k+1)^{-1-\theta}.
$$
\end{definition}

The capacity bounds and $\#\{j\}\asymp\mathrm{vol}_g(X)R_k^{-m}$ show that $\mu_\theta$ is finite. It is carried by $S=\bigcup_{k,j}K_{k,j}$, which is a dense $F_\sigma$ set with $\mathrm{vol}_g(S)=0$ and $\dim_{\mathrm H}S=m-2$, while $\operatorname{supp}\mu_\theta=X$. For open $U\ne\varnothing$, take $B(x,r)\subset U$ and choose $k,j$ with $5R_k/4<r$ and $d(x,x_j^{(k)})\le R_k$. Then $K_{k,j}\subset U$ and $\mu_\theta(U)>0$.

\begin{lemma}[Uniform trace bound]\label{lem:multiscale-trace}
Uniformly in $k$,
\begin{enumerate}
\item[\textup{(i)}] $c\,\mathrm{vol}_g(X)\le\eta_k(X)\le C\,\mathrm{vol}_g(X)$;
\item[\textup{(ii)}] $\int_X|\widetilde u|^2\,d\eta_k \le C\bigl(\|u\|_{L^2(X)}^2+R_k^2\|\nabla u\|_{L^2(X)}^2\bigr)$ for every $u\in H^1(X)$.
\end{enumerate}
The constants depend only on $m$, bounded geometry, and $K_*$; hence $H^1(X)$ embeds continuously into $L^2(\mu_\theta)$.
\end{lemma}

\begin{proof}
Write $R=R_k$. Maximal separation and the capacity bounds give $\#\{j\}\asymp\mathrm{vol}_g(X)R^{-m}$ and $\mu_{k,j}(K_{k,j})\asymp R^{m-2}$, proving \textup{(i)}.

For \textup{(ii)}, \cref{prop:equilibrium-resolvent}, applied to $K_{k,j}\Subset B^*_{k,j}$, gives
$$
\int|\widetilde v|^2\,d\mu_{k,j}
\le \mathcal E_{B^*_{k,j}}(v),
\qquad v\in H^1_0(B^*_{k,j}).
$$
Choose $\chi_{k,j}\in C_c^\infty(B^*_{k,j})$ equal to $1$ near $K_{k,j}$ with $|\nabla\chi_{k,j}|\le C/R$. Apply the last inequality to $\chi_{k,j}u$, multiply by $R^2$, and sum:
$$
\int_X|\widetilde u|^2\,d\eta_k
\le CR^2\sum_j\mathcal E(\chi_{k,j}u)
\le C\bigl(\|u\|_2^2+R^2\|\nabla u\|_2^2\bigr)
$$
by bounded overlap. Summing against $c_\theta(k+1)^{-1-\theta}$ completes the proof.
\end{proof}

\begin{lemma}[One-scale estimate]\label{lem:one-scale}
There is $C_*=C_*(m,\mathrm{bg},K_*)$ such that, for every $u\in H^1(X)$ and $k\ge0$,
\begin{equation}\label{eq:one-scale}
\|u\|_{L^2(X)}^2\le C_*\left(R_k^2\|\nabla u\|_{L^2(X)}^2
+\int_X|\widetilde u|^2\,d\eta_k\right),
\end{equation}
where $\widetilde u$ is the quasi-continuous representative.
\end{lemma}

\begin{proof}
The Brownian killing argument of \cref{prop:equilibrium-killing-spectral-gap}, applied at scale $R_k$, gives
$$
\inf\sigma\left(H+\sum_j\mu_{k,j}\right)\ge cR_k^{-2}.
$$
The Rayleigh quotient therefore yields
$$
cR_k^{-2}\|u\|_2^2
\le\frac12\|\nabla u\|_2^2+\sum_j\int_X|\widetilde u|^2\,d\mu_{k,j}.
$$
Multiplication by $R_k^2$ and the definition of $\eta_k$ prove the result.
\end{proof}

\begin{theorem}[Observability and null control at the critical dimension]
\label{thm:heat-control}
There is $C_\theta=C(\theta,m,\mathrm{bg},K_*)>0$ such that the following hold.
\begin{enumerate}
\item[\textup{(i)}] For every $E\ge0$ and $f\in\operatorname{Ran}P_{[0,E]}(H)$,
$$
\|f\|_{L^2(X)}^2\le C_\theta\bigl(1+\log(2+T_0E)\bigr)^\theta
\int_X|f|^2\,d\mu_\theta.
$$

\item[\textup{(ii)}] For every $0<T\le T_0$ and $u_0\in L^2(X)$,
$$
\|e^{-TH}u_0\|_{L^2(X)}^2\le
\frac{C_\theta}{T}\bigl(1+\log(T_0/T)\bigr)^\theta
\int_0^T\!\int_X|e^{-tH}u_0|^2\,d\mu_\theta\,dt.
$$

\item[\textup{(iii)}] Define $B:L^2(\mu_\theta)\to H^{-1}(X)$ by
$$
\langle Bv,\varphi\rangle=\int_Xv\,\overline{\widetilde\varphi}\,d\mu_\theta.
$$
For every $0<T\le T_0$ and $y_0\in L^2(X)$, there is $v\in L^2((0,T);L^2(\mu_\theta))$ such that the weak solution of $\partial_ty+Hy=Bv$, $y(0)=y_0$, satisfies $y(T)=0$ and
$$
\|v\|_{L^2(dt\,d\mu_\theta)}\le
\left(\frac{C_\theta}{T}\right)^{1/2}
\bigl(1+\log(T_0/T)\bigr)^{\theta/2}\|y_0\|_{L^2(X)}.
$$
\end{enumerate}
\end{theorem}

\begin{proof}
For \textup{(i)}, since $\|\nabla f\|_2^2\le2E\|f\|_2^2$, choose the least $k\ge0$ such that $2C_*R_k^2E\le\tfrac12$. Since $R_k^2=4^{-k}T_0$, $k+1\le C(1+\log(2+T_0E))$. For every $\ell\ge k$, \eqref{eq:one-scale} gives
$$
\int_X|f|^2\,d\eta_\ell\ge(2C_*)^{-1}\|f\|_2^2.
$$
Since $\sum_{\ell\ge k}(\ell+1)^{-1-\theta} \ge c(\theta)(k+1)^{-\theta}$, consequently
$$
\int_X|f|^2\,d\mu_\theta
\ge c(k+1)^{-\theta}\|f\|_2^2.
$$
This proves \textup{(i)}.

For \textup{(ii)}, set $u(t)=e^{-tH}u_0$, $F(t)=\|u(t)\|_2^2$, and $G(t)=\|u(t)\|_{L^2(\mu_\theta)}^2$. For smooth $u_0$ and fixed $t>0$, $u(t-s,B_s)$, $0\le s\le t$, is a Brownian martingale with quadratic variation $\int_0^s|\nabla u(t-r,B_r)|^2\,dr$. The It\^o isometry, integrated over the starting point and using invariance of volume, yields
$$
F(0)-F(t)=\int_0^t\|\nabla u(s)\|_2^2\,ds.
$$
Approximation extends this identity to every $u_0\in L^2(X)$. Together with the trace bound it gives
$$
F'(t)=-\|\nabla u(t)\|_2^2\quad\text{for almost every }t>0,
\qquad
\int_0^TG(t)\,dt\le C(T+1)\|u_0\|_2^2.
$$
We fix $\gamma=2+\theta$. For $0<t\le T_0$, let $k(t)$ be the least $k\ge0$ such that $C_*\gamma R_k^2\le t$. Then $k(t)+1\le C_\theta(1+\log(T_0/t))$. For every $\ell\ge k(t)$, \eqref{eq:one-scale} gives
$$
F(t)\le\frac{t}{\gamma}(-F'(t))
+C_*\int_X|u(t)|^2\,d\eta_\ell.
$$
Multiplying by $c_\theta(\ell+1)^{-1-\theta}$, summing over $\ell\ge k(t)$, and using the same sum estimate gives
$$
F'(t)+\frac{\gamma}{t}F(t)
\le\frac{C_\theta}{t}\bigl(1+\log(T_0/t)\bigr)^\theta G(t).
$$
Put $L(t)=1+\log(T_0/t)$. The function $t^{\gamma-1}L(t)^\theta$ is increasing because its logarithmic derivative is $t^{-1}(\gamma-1-\theta/L(t))\ge0$. Multiplying by $t^\gamma$ and integrating from $0$ to $T$ gives
$$
F(T)\le\frac{C_\theta}{T}L(T)^\theta\int_0^TG(t)\,dt.
$$
Here $t^\gamma F(t)\to0$ as $t\downarrow0$. This proves \textup{(ii)}. The trace estimate makes the final-state control map bounded, and its adjoint is $z\mapsto\widetilde{e^{-(T-t)H}z}$. Standard Hilbert-space duality now gives the control statement with the square root of this constant, proving \textup{(iii)}.
\end{proof}

\noindent\textbf{Prescribing the observability loss.} Let $a_k\ge0$, $\sum_{k\ge0}a_k=1$, and set $\mu_a=\sum_{k\ge0}a_k\eta_k$ and $A_k=\sum_{\ell\ge k}a_\ell$. Since the trace bounds for $\eta_k$ are uniform, $\mu_a$ is finite, $H^1(X)\hookrightarrow L^2(\mu_a)$, and $\mu_a$ is carried by the dense zero-volume set of Hausdorff dimension $m-2$ from \cref{def:multiscale-measure}. For a finite Borel measure $\mu$, let $\mathcal C_\mu(E)$ be the least constant such that
$$
\|f\|_{L^2(X)}^2\le\mathcal C_\mu(E)\int_X|f|^2\,d\mu,
\qquad 0\ne f\in\operatorname{Ran}P_{[0,E]}(H),
$$
and let $\mathcal H_\mu(T)$ be the least $K$ such that, for every $u_0\in L^2(X)$,
$$
\|e^{-TH}u_0\|_2^2\le\frac K T
\int_0^T\!\int_X|e^{-tH}u_0|^2\,d\mu\,dt.
$$
When $H^1(X)\hookrightarrow L^2(\mu)$, define the control operator as in \cref{thm:heat-control} and set
$$
\operatorname{Cost}_\mu(T)=\sup_{\|y_0\|_2=1}
\inf\left\{\|v\|_{L^2(dt\,d\mu)}:
v\in L^2((0,T);L^2(\mu)),\ y(T)=0\right\},
$$
where $y$ is the corresponding solution with initial value $y_0$.

\begin{lemma}[Bounds on balls]\label{lem:eta-ball-growth}
There are $c,C>0$ and $C_0\ge4$, depending only on bounded geometry and $K_*$, such that, for every $k\ge0$, $x\in X$, and $0<r\le R_0$,
$$
\eta_k(B(x,r))\le CR_k^2r^{m-2}\quad(r\le R_k),\qquad
\eta_k(B(x,r))\le Cr^m\quad(r\ge R_k),
$$
and $\eta_k(B(x,r))\ge cr^m$ whenever $r\ge C_0R_k$. If
$$
k(r)=\min\{k\ge0:C_0R_k\le r\},
$$
then $R_{k(r)}\asymp r$ and $\mu_a(B(x,r))\ge cA_{k(r)}r^m$.
\end{lemma}

\begin{proof}
Suppose $r\le R_k$. The ball $B(x,r)$ meets only boundedly many supports of the measures $\mu_{k,j}$. If $\mu_{k,j}(B(x,r))>0$, choose $x'\in K_{k,j}\cap B(x,r)$ outside the polar exceptional set, where the equilibrium potential equals $1$. Both $x'$ and the support of $\mu_{k,j}|_{B(x,r)}$ are at distance at least $R_k$ from $\partial B^*_{k,j}$, so the interior Green bound gives
$$
1\ge\int_{B(x,r)}G_{B^*_{k,j}}(x',y)\,d\mu_{k,j}(y)
\ge cr^{2-m}\mu_{k,j}(B(x,r)).
$$
Multiplication by $R_k^2$ and summation prove the first bound. If $r\ge R_k$, at most $C(r/R_k)^m$ supports meet $B(x,r)$, and each contributes at most $CR_k^m$. If $r\ge C_0R_k$, maximality of the net and volume comparison give at least $c(r/R_k)^m$ centres in $B(x,r/2)$; their supports lie in $B(x,r)$ and each contributes at least $cR_k^m$. This proves the remaining bounds. Applying the last one for every $\ell\ge k(r)$ and summing with weights $a_\ell$ proves the final assertion.
\end{proof}

The reverse estimates for prescribed profiles require test functions that are spectrally localised at energy $R^{-2}$ and spatially concentrated near one point. The next lemma fixes such a family.

\begin{lemma}[Spectral cutoff kernels]\label{lem:spectral-cutoff-kernels}
Fix once and for all $q_0\in C_c^\infty([0,4))$ with $q_0=1$ on $[0,1]$ and $0\le q_0\le1$, and set $q=q_0^2$. For $x_*\in X$ and $0<R\le R_0$, set $G_R=q(R^2H)(x_*,\cdot)$. Then
$$
G_R\in\operatorname{Ran}P_{[0,4R^{-2}]}(H),
\qquad
\|G_R\|_{L^2(X)}^2\asymp R^{-m}.
$$
Moreover, for every integer $N\ge1$, every $0\le t\le R^2$, and every $y\in X$,
$$
|e^{-tH}G_R(y)|
\le C_NR^{-m}\left(1+\frac{d(x_*,y)}R\right)^{-N},
$$
and, for every $\rho>0$,
$$
\int_{X\setminus B(x_*,\rho)}
\left(|e^{-tH}G_R|^2
+R^2|\nabla e^{-tH}G_R|^2\right)\,d\mathrm{vol}_g
\le C_NR^{-m}\left(1+\frac{\rho}{R}\right)^{-2N}.
$$
The constants depend only on $m$, the order-zero bounded-geometry constants, and $N$.
\end{lemma}

\begin{proof}
Fix $0\le t\le R^2$, write $\tau=t/R^2$, and set $B_\rho=B(x_*,\rho)$.  Put
\[
h_\tau(s)=e^{-\tau s}q(s),
\qquad
b_\tau(s)=e^sh_\tau(s).
\]
Since $-\Delta_g=2H$, define, for $j=0,1$,
\[
F_{\tau,j}(\xi)
=
\left(\frac{\xi^2}{2}\right)^j
b_\tau\left(\frac{\xi^2}{2}\right).
\]
Then
\[
F_{\tau,j}(R\sqrt{-\Delta_g})
=(R^2H)^j b_\tau(R^2H).
\]
The families $F_{\tau,j}$ are bounded in $C_c^\infty(\mathbb R)$, uniformly for $0\le\tau\le1$.  Their Fourier transforms are therefore uniformly rapidly decreasing: for every $N$,
\[
\sup_{0\le\tau\le1}
\int_a^\infty
\left(
 |\widehat F_{\tau,0}(s)|
 +|\widehat F_{\tau,1}(s)|
\right)\,ds
\le C_N(1+a)^{-N},
\qquad a\ge0.
\]
After the spectral rescaling $\xi\mapsto R\xi$, \cite[Proposition~1.1]{CGT1982} and finite propagation give, for every $x\in X$ and $\rho>0$,
\[
\sum_{j=0}^1
\left\|
\mathbf 1_{X\setminus B(x,\rho)}
(R^2H)^j b_\tau(R^2H)
\mathbf 1_{B(x,\rho/2)}
\right\|_{2\to2}
\le
C_N\left(1+\frac{\rho}{R}\right)^{-N}.
\]
Indeed, the Fourier-tail term in that proposition begins at $\rho/(2R)$.

Let
\[
p=p_{R^2}(x_*,\cdot).
\]
The semigroup identity and the Gaussian heat-kernel bounds give
\[
\|p\|_2^2=p_{2R^2}(x_*,x_*)\le CR^{-m},
\qquad
\|\mathbf 1_{B_\rho^c}p\|_2
\le CR^{-m/2}e^{-c\rho^2/R^2}.
\]
Since $h_\tau(s)=e^{-s}b_\tau(s)$,
\[
w:=h_\tau(R^2H)(x_*,\cdot)
=b_\tau(R^2H)p
=e^{-tH}G_R.
\]
Split $p$ into its parts on $B_{\rho/2}$ and $B_{\rho/2}^c$.  The preceding Cheeger-Gromov-Taylor estimate controls the first part after applying $b_\tau(R^2H)$, while the Gaussian tail and the uniform $L^2$ operator bound control the second.  Repeating the argument with $j=1$ gives
\[
\|\mathbf 1_{B_\rho^c}w\|_2
+
\|\mathbf 1_{B_\rho^c}R^2Hw\|_2
\le
C_NR^{-m/2}
\left(1+\frac{\rho}{R}\right)^{-N}.
\]

Suppose first that $\rho>2R$.  Choose $\chi$ equal to $0$ on $B_{\rho/2}$ and to $1$ on $B_\rho^c$, with $|\nabla\chi|\le C/\rho$.  Taking real parts after testing $Hw$ against $\chi^2w$ gives
\[
\begin{aligned}
\int_{B_\rho^c}|\nabla w|^2
&\le
C\|\mathbf 1_{B_{\rho/2}^c}Hw\|_2
 \|\mathbf 1_{B_{\rho/2}^c}w\|_2\\
&\quad
+C\rho^{-2}
 \|\mathbf 1_{B_{\rho/2}^c}w\|_2^2.
\end{aligned}
\]
Multiplying by $R^2$ and using the last off-diagonal estimates proves
\[
\int_{B_\rho^c}
\left(|w|^2+R^2|\nabla w|^2\right)\,d\mathrm{vol}_g
\le
C_NR^{-m}
\left(1+\frac{\rho}{R}\right)^{-2N}.
\]
For $\rho\le2R$, the same conclusion follows from the global $L^2$ bound and
\[
R^2\|\nabla w\|_2^2
=
2R^2\langle Hw,w\rangle
\le8\|w\|_2^2,
\]
because $w$ has spectral support in $[0,4R^{-2}]$.

For the pointwise estimate, put
\[
a_\tau(s)=e^{-\tau s/2}q_0(s).
\]
Repeating the preceding argument with $e^sa_\tau(s)$ shows that the kernel $K_{\tau,R}$ of $a_\tau(R^2H)$ satisfies
\[
\|\mathbf 1_{X\setminus B(x,\rho)}
 K_{\tau,R}(x,\cdot)\|_2
\le
C_NR^{-m/2}
\left(1+\frac{\rho}{R}\right)^{-N}
\]
uniformly in $x$, $\tau$, and $\rho\ge0$.  Since $h_\tau=a_\tau^2$ and $K_{\tau,R}$ is symmetric,
\[
w(y)
=
\int_X
K_{\tau,R}(x_*,z)K_{\tau,R}(z,y)\,d\mathrm{vol}_g(z).
\]
Let $d=d(x_*,y)$.  Split the integral into the region where $d(x_*,z)\ge d/2$ and its complement; on the complement, $d(z,y)>d/2$.  Cauchy--Schwarz and the last $L^2$ estimate give
\[
|w(y)|
\le
C_NR^{-m}
\left(1+\frac{d(x_*,y)}{R}\right)^{-N}.
\]

It remains to estimate the norm of $G_R$.  The support of $q$ gives
\[
G_R\in\operatorname{Ran}P_{[0,4R^{-2}]}(H),
\]
and
\[
\|G_R\|_2^2
=(q^2)(R^2H)(x_*,x_*).
\]
Since $q^2(s)\le e^4e^{-s}$ for $s\ge0$,
\[
\|G_R\|_2^2
\le e^4p_{R^2}(x_*,x_*)
\le CR^{-m}.
\]
For the reverse inequality, $q^2\ge\mathbf 1_{[0,1]}$, and the spectral theorem gives, for every $A>0$,
\[
\begin{aligned}
\|G_R\|_2^2
&\ge P_{[0,R^{-2}]}(H)(x_*,x_*)\\
&\ge
p_{AR^2}(x_*,x_*)
-e^{-A/2}p_{AR^2/2}(x_*,x_*).
\end{aligned}
\]
Choose $A$ large enough that the second term is at most half the first, and then decrease the fixed upper scale $R_0$ if necessary so that $AR_0^2$ lies in the range of the uniform two-sided diagonal heat-kernel bounds.  It follows that
\[
\|G_R\|_2^2\ge cR^{-m},
\]
which completes the proof.
\end{proof}

\begin{theorem}[Prescribed observability profiles]
\label{thm:prescribed-loss}
Let $\Phi:[T_0^{-1},\infty)\to[1,\infty)$ be nondecreasing and unbounded, with $\Phi(4s)\le D\Phi(s)$ for $s\ge T_0^{-1}$. There is a finite measure $\mu_\Phi$ such that:
\begin{enumerate}
\item[\textup{(i)}] $\mu_\Phi$ has full support, is carried by a zero-volume set of Hausdorff dimension $m-2$, and satisfies $H^1(X)\hookrightarrow L^2(\mu_\Phi)$. Moreover,
$$
\inf_{x\in X}\frac{\mu_\Phi(B(x,r))}{r^m}
\asymp\Phi(r^{-2})^{-1},\qquad 0<r\le R_0.
$$
\item[\textup{(ii)}] $\mathcal C_{\mu_\Phi}(E)\asymp\Phi(E)$ for $E\ge T_0^{-1}$.
\item[\textup{(iii)}] $\mathcal H_{\mu_\Phi}(T)\asymp\Phi(T^{-1})$ and $\operatorname{Cost}_{\mu_\Phi}(T) \asymp T^{-1/2}\Phi(T^{-1})^{1/2}$ for $0<T\le T_0$.
\end{enumerate}
The constants depend only on $m$, the order-zero bounded-geometry constants, $K_*$, $D$, and $\Phi(T_0^{-1})$.
\end{theorem}

\begin{proof}
Set
$$
A_k=\frac{\Phi(T_0^{-1})}{\Phi(R_k^{-2})},
\qquad a_k=A_k-A_{k+1}.
$$
Then $A_0=1$, $A_k\downarrow0$, $a_k\ge0$, and $\sum_{k\ge0}a_k=1$. In \cref{def:multiscale-measure}, choose every net to contain a fixed point $x_*$. In its $x_*$-cell, shrink $K_*$ by $1/4$ and translate it into
$$
\overline{B}(x_*,7R_k/16)\setminus B(x_*,R_k/4).
$$
The capacity, trace, one-scale, and ball estimates remain valid with modified constants, and
$$
\operatorname{dist}(x_*,\operatorname{supp}\eta_k)\ge R_k/4.
$$
Define $\mu_\Phi=\sum_{k\ge0}a_k\eta_k$. The trace estimate and the assertion that $\mu_\Phi$ is carried by a zero-volume set of Hausdorff dimension $m-2$ follow as before. Since $a_k>0$ for arbitrarily large $k$, the measure has full support. The preceding lemma gives the required lower bound on balls. For $0<r<R_0/8$, choose $k\ge1$ so that $R_k/8\le r<R_{k-1}/8$. The support separation removes the terms with $\ell\le k$, while the upper ball bounds give $\mu_\Phi(B(x_*,r))\le CA_{k+1}r^m$. The range $R_0/8\le r\le R_0$ is absorbed into the constants, and the condition on $\Phi$ proves \textup{(i)}.

For the upper spectral bound, choose the least $k\ge0$ such that $2C_*R_k^2E\le\tfrac12$. Applying \eqref{eq:one-scale} at every $\ell\ge k$ and summing with weights $a_\ell$ gives $\mathcal C_{\mu_\Phi}(E)\le CA_k^{-1}\le C\Phi(E)$. For heat observability, choose $\gamma$ with $4^{\gamma-1}\ge2D$ and let $k(t)$ be the least index with $C_*\gamma R_{k(t)}^2\le t$. Then
$$
\sup_{0<t\le T}t^{\gamma-1}A_{k(t)}^{-1}
\le CT^{\gamma-1}A_{k(T)}^{-1}.
$$
Repeating the proof of \cref{thm:heat-control} and using $\sum_{\ell\ge k(t)}a_\ell=A_{k(t)}$ gives $\mathcal H_{\mu_\Phi}(T)\le C\Phi(T^{-1})$.

For the reverse bounds, set $T_k=R_k^2$, $g_k=G_{R_k}/\|G_{R_k}\|_2$, and $u_k(t)=e^{-tH}g_k$. For $\ell<k$, $\operatorname{dist}(x_*,\operatorname{supp}\eta_\ell)\ge R_\ell/4$. Choose $\chi_\ell$ equal to $1$ on $\operatorname{supp}\eta_\ell$, zero on $B(x_*,R_\ell/8)$, and with $|\nabla\chi_\ell|\le C/R_\ell$. \Cref{lem:multiscale-trace}\textup{(ii)}, \cref{lem:spectral-cutoff-kernels}, and, for $\ell\ge k$, the spectral theorem give uniformly for $0\le t\le T_k$,
$$
\int_X|u_k(t)|^2\,d\eta_\ell
\le\begin{cases}
C,&\ell\ge k,\\
C_N(R_k/R_\ell)^{2N-2},&\ell<k.
\end{cases}
$$
Choose $N$ so that $2^{2N-2}>2D$. Since $A_\ell\le D^{k-\ell}A_k$ and $a_\ell\le A_\ell$ for $\ell<k$, while $\sum_{\ell\ge k}a_\ell=A_k$, summation gives
$$
\int_X|e^{-tH}g_k|^2\,d\mu_\Phi\le CA_k,
\qquad 0\le t\le T_k.
$$
At $t=0$, this and $g_k\in\operatorname{Ran}P_{[0,E_k]}(H)$, $E_k=4R_k^{-2}$, give $\mathcal C_{\mu_\Phi}(E_k)\ge cA_k^{-1}\ge c\Phi(E_k)$. Since $E_{k+1}=4E_k$, monotonicity and $\Phi(4s)\le D\Phi(s)$ give the lower bound for $E\ge4T_0^{-1}$; the constant eigenfunction covers $T_0^{-1}\le E<4T_0^{-1}$.

For $0<T\le T_0$, choose $k$ with $T_{k+1}<T\le T_k$. Then $\|e^{-TH}g_k\|_2^2\ge e^{-8}$, while the last display bounds the observation integral by $CTA_k$. Hence $\mathcal H_{\mu_\Phi}(T)\ge cA_k^{-1}\ge c\Phi(T^{-1})$. Finally, let $K_T$ be the final-state map from the control to $L^2(X)$. Its adjoint is
\[
(K_T^*z)(s)=\widetilde{e^{-(T-s)H}z}
\quad\text{in }L^2(\mu_\Phi).
\]
Douglas' range-inclusion lemma identifies the squared worst minimal-control norm with the least constant in the corresponding adjoint observability inequality. By the definition of $\mathcal H_{\mu_\Phi}(T)$, this gives
$$
\operatorname{Cost}_{\mu_\Phi}(T)^2
=\frac{\mathcal H_{\mu_\Phi}(T)}{T},
$$
which proves the remaining assertion.
\end{proof}

\begin{theorem}[Observation and control for uniformly elliptic operators]\label{thm:uniform-elliptic-control}
Fix $0<a_-\le a_+<\infty$ and $0<\rho_-\le\rho_+<\infty$. Let $\mathbf A$ be a measurable $g$-symmetric endomorphism of $TX$, let $\rho$ be measurable, and let $W\in L^\infty(X)$ be real-valued, with
$$
a_-|\xi|_g^2\le\langle\mathbf A\xi,\xi\rangle_g\le a_+|\xi|_g^2,
\qquad \rho_-\le\rho\le\rho_+,
\qquad W\ge0.
$$
On $Y_\rho:=L^2(X,\rho\,d\mathrm{vol}_g)$ let $H_{\mathbf A,W,\rho}$ be the self-adjoint operator associated with
$$
\mathfrak h_{\mathbf A,W}[u]
=\frac12\int_X\langle\mathbf A\nabla u,\nabla u\rangle_g\,d\mathrm{vol}_g
+\int_XW|u|^2\,d\mathrm{vol}_g,
\qquad u\in H^1(X).
$$
Write $P^{\mathbf A,W,\rho}$ for its spectral projections. The measures above may be chosen independently of $\mathbf A,\rho,W$.
\begin{enumerate}
\item[\textup{(i)}] The conclusions of \cref{thm:capacity-sampling} hold for $f\in\operatorname{Ran}P^{\mathbf A,W,\rho}_{[0,E]}$ when $E\le ca_-\rho_+^{-1}(1+\Lambda)^{-1}R^{-2}$, with the lower bounds $c\rho_+^{-1}\|f\|_{Y_\rho}^2$. In particular, \cref{cor:zero-volume} becomes
$$
R^2\int_X|\widetilde f|^2\,d\mu_R
\ge\frac{c\kappa}{\rho_+}\|f\|_{Y_\rho}^2
$$
when $E\le ca_-\kappa\rho_+^{-1}R^{-2}$.

\item[\textup{(ii)}] The spectral, heat-observability, and null-control conclusions of \cref{thm:heat-control} hold with $H,L^2(X)$ replaced by $H_{\mathbf A,W,\rho},Y_\rho$, and with the same $\mu_\theta$. The control equation is understood variationally, with $Y_\rho\hookrightarrow H^1(X)^*$ through its weighted inner product.

\item[\textup{(iii)}] For every $\Phi$ in \cref{thm:prescribed-loss}, the same $\mu_\Phi$ satisfies its three upper bounds after these replacements; their orders are optimal uniformly over the coefficient class.
\end{enumerate}
The constants depend only on the displayed class bounds and the geometric data, not on the individual coefficients or on $\|W\|_\infty$. The support, trace, volume, and dimension properties of the measures are unchanged.
\end{theorem}

\begin{proof}
The proof rests on two comparisons. First,
$$
\|u\|_{Y_\rho}^2\le\rho_+\|u\|_2^2,
\qquad
\mathfrak h_{\mathbf A,W}[u]\ge\frac{a_-}{2}\|\nabla u\|_2^2.
$$
Thus \cref{lem:one-scale} gives
\begin{equation}\label{eq:uniform-elliptic-one-scale}
\|u\|_{Y_\rho}^2
\le\frac{2C_*\rho_+}{a_-}R_k^2\mathfrak h_{\mathbf A,W}[u]
+C_*\rho_+\int_X|\widetilde u|^2\,d\eta_k.
\end{equation}
Second, summing \cref{lem:local-sampling} gives
$$
\|u\|_{Y_\rho}^2
\le\frac{C\rho_+(1+\Lambda)}{a_-}R^2
\mathfrak h_{\mathbf A,W}[u]
+C\rho_+\sum_jR^m\left|\int_X\widetilde u\,d\nu_j\right|^2.
$$
If $f\in\operatorname{Ran}P^{\mathbf A,W,\rho}_{[0,E]}$, the spectral theorem gives
\[
\mathfrak h_{\mathbf A,W}[f]\le E\|f\|_{Y_\rho}^2.
\]
Substitution in the preceding local estimate, followed by absorption under the stated restriction on $E$, proves \textup{(i)}. Substitution in \eqref{eq:uniform-elliptic-one-scale} and summation over the scales $\ell\ge k$ proves the spectral assertion in \textup{(ii)}, as in the spectral part of the proof of \cref{thm:heat-control}. This step uses only the form inequalities above; no regularity of $\mathbf A$ or $\rho$ is required.

For $u(t)=e^{-tH_{\mathbf A,W,\rho}}u_0$, semigroup form theory gives, for almost every $t>0$,
\[
F'(t)=-2\mathfrak h_{\mathbf A,W}[u(t)],
\qquad F(t):=\|u(t)\|_{Y_\rho}^2.
\]
Together with \eqref{eq:uniform-elliptic-one-scale}, this is the differential identity used in the proof of \cref{thm:heat-control}. With $\gamma=2+\theta$, choose $k(t)$ so that $C_*\gamma(\rho_+/a_-)R_{k(t)}^2\le t$. Multiplying \eqref{eq:uniform-elliptic-one-scale} by the tail weights and summing gives the same differential inequality as in that proof, with constants depending only on $a_-$ and $\rho_+$, and hence heat observability. The uniform trace estimate makes the control operator $B:L^2(\mu_\theta)\to H^1(X)^*$ bounded. After time reversal, the observability estimate and the Douglas range-inclusion argument used in the proof of \cref{thm:prescribed-loss} give the null control with the asserted cost.

For $\mu_\Phi$, summing over the same tail of scales uses $\sum_{\ell\ge k}a_\ell=A_k$, and $\Phi(4s)\le D\Phi(s)$ controls the fixed shift of scale. This proves the three upper bounds in \textup{(iii)}. The nonnegative term $\int W|u|^2$ only strengthens the form bounds, which explains why no constant depends on $\|W\|_\infty$. Finally, the coefficient class contains $\mathbf A=a_-I$, $\rho=\rho_-$, $W=0$. For this subclass, $H_{\mathbf A,0,\rho}=(a_-/\rho_-)H$, so after this fixed energy and time rescaling \cref{thm:prescribed-loss} supplies the matching lower bounds.
\end{proof}

\paragraph{The associated diffusion.}
For fixed $\mathbf A$ and $\rho$, let $X_t^{\mathbf A,\rho}$ be the diffusion associated with the form $\tfrac12\int\langle\mathbf A\nabla u,\nabla u\rangle_g$ on $L^2(X,\rho\,d\mathrm{vol}_g)$. Here $\mathbf A$ determines local covariance, $\rho$ changes the time parametrisation, and $W$ appears as killing. If $\mu^{\mathbf A}_{K,\Omega}$ is the equilibrium measure and $A_t^{K,\mathbf A,\rho}$ is the positive continuous additive functional with this Revuz measure relative to $\rho\,d\mathrm{vol}_g$, then
$$
\mathbb E_y^{\mathbf A,\rho}
e^{-\beta A^{K,\mathbf A,\rho}_{\tau_\Omega}}
=1-\frac{\beta}{1+\beta}
\mathbb P_y^{\mathbf A,\rho}(\tau_K<\tau_\Omega)
$$
for quasi-every $y$, exactly as in \cref{lem:equilibrium-additive-functional-law}. The corresponding capacities satisfy
\[
a_-\operatorname{Cap}_g\le\operatorname{Cap}_{\mathbf A}
\le a_+\operatorname{Cap}_g.
\]
The theorem instead uses the fixed geometric measures, so its observation measure is independent of $\mathbf A$, $\rho$, and $W$.

\begin{corollary}[Perforated domains and instantaneous hitting]
\label{cor:full-volume-perforations}
For each $k\ge0$, set
$$
F_k=\bigcup_jK_{k,j},\qquad \Omega_k=X\setminus F_k,
\qquad S=\bigcup_{k\ge0}F_k,
$$
with $K_{k,j}$ as in \cref{def:multiscale-measure}. Then:
\begin{enumerate}
\item[\textup{(i)}] The set $\Omega_k$ is open,
$$
\operatorname{vol}_g(\Omega_k)=\operatorname{vol}_g(X),\qquad
\dim_{\mathrm H}\partial\Omega_k=m-2,
$$
and
$$
cR_k^{-2}\le\mu_1(\Omega_k)\le CR_k^{-2}.
$$
The lower bound is uniform over the coefficient class of \cref{thm:uniform-elliptic-control}: if $H_{\mathbf A,W,\rho}^{\Omega_k}$ is the Dirichlet realisation on $\Omega_k$, then
$$
\inf\sigma(H_{\mathbf A,W,\rho}^{\Omega_k})
\ge \frac{ca_-}{\rho_+}R_k^{-2}.
$$
\item[\textup{(ii)}] One has
$$
\{u\in H^1(X):\widetilde u=0\text{ quasi-everywhere on }S\}=\{0\}.
$$
\item[\textup{(iii)}] Although $X\setminus S$ has full volume, Brownian motion started at every $x\in X\setminus S$ satisfies
$$
\mathbb P_x(\tau_S^+=0)=1,
\qquad
\tau_S^+:=\inf\{t>0:B_t\in S\}.
$$
In fact, almost surely $\{t>0:B_t\in S\}$ is dense in $(0,\infty)$ but has Lebesgue measure zero.
\end{enumerate}
\end{corollary}

\begin{proof}
Each $F_k$ is a finite union of compact sets of dimension $m-2$, hence is compact, has zero volume and empty interior, and $\partial\Omega_k=F_k$. If $u\in H^1_0(\Omega_k)$ is extended by zero to $X$, then $\widetilde u=0$ quasi-everywhere on $F_k$. Since $\eta_k$ is carried by $F_k$, \eqref{eq:one-scale} gives
$$
\|u\|_2^2\le C_*R_k^2\|\nabla u\|_2^2,
$$
which proves the lower bound. For the reverse bound, fix a centre $x_j^{(k)}$, choose a unit vector $v$, and put $y=\exp_{x_j^{(k)}}(R_kv/2)$. The separation of the centres and $K_{k,i}\subset B(x_i^{(k)},R_k/4)$ give $B(y,R_k/8)\subset\Omega_k$. Domain monotonicity and the ball eigenvalue bound prove $\mu_1(\Omega_k)\le CR_k^{-2}$.

For the operators $H_{\mathbf A,W,\rho}^{\Omega_k}$, the same zero extension and \eqref{eq:one-scale}, together with $\|u\|_{Y_\rho}^2\le\rho_+\|u\|_2^2$ and $\mathfrak h_{\mathbf A,W}[u]\ge(a_-/2)\|\nabla u\|_2^2$, give the stated uniform lower bound. If $\widetilde u=0$ quasi-everywhere on $S$, the observation term in \eqref{eq:one-scale} vanishes for every $k$; letting $k\to\infty$ gives $u=0$.

Finally fix $x\notin S$. For each $k$, choose $j(k)$ with $x\in B_{k,j(k)}$. The capacity and Green-function estimate used in \cref{lem:one-step-capacity-killing} gives the following lower bound quasi-everywhere; continuity of the equilibrium potential off $K_{k,j(k)}$ extends it to this $x$:
$$
\mathbb P_x\!\left(\tau_{K_{k,j(k)}}<\tau_{B^*_{k,j(k)}}\right)\ge c_0,
$$
independently of $k$. Since $B^*_{k,j(k)}\subset B(x,3R_k)$, writing $\sigma_k$ for the exit time from the latter ball gives $\mathbb P_x(\tau_S^+<\sigma_k)\ge c_0$. The times $\sigma_k$ decrease to zero almost surely, and therefore $\mathbb P_x(\tau_S^+=0)\ge c_0$. This event belongs to the germ $\sigma$-field at time zero, so Blumenthal's zero-one law makes its probability one. For every rational $q>0$, $B_q\notin S$ almost surely because the heat kernel has a density and $S$ has zero volume. The Markov property at $q$ then gives visits to $S$ arbitrarily soon after $q$; intersecting over rational $q$ proves density. Finally,
$$
\mathbb E_x\int_0^T\mathbf 1_S(B_t)\,dt
=\int_0^T\int_Sp_t(x,y)\,d\operatorname{vol}_g(y)\,dt=0,
$$
which proves the last assertion.
\end{proof}

\begin{proposition}[Spectral coupling and its additive functional]
\label{prop:variational-additive-functional}
Let $\mu$ be a finite positive measure such that $H^1(X)\hookrightarrow L^2(\mu)$, and define
$$
\mathcal J_\mu(E)
=
\inf\left\{
\int_X|\widetilde u|^2\,d\mu:
u\in H^1(X),\ \|u\|_2=1,\ \mathcal E(u)\le E
\right\}.
$$
If
$$
\lambda_\mu(\gamma)=\inf\sigma(H+\gamma\mu),
\qquad \gamma>0,
$$
then
\begin{equation}\label{eq:variational-coupling}
\lambda_\mu(\gamma)
=
\inf_{E\ge0}\bigl(E+\gamma\mathcal J_\mu(E)\bigr).
\end{equation}

Extend $\Phi$ in \cref{thm:prescribed-loss} to $[0,T_0^{-1})$ by the constant value $\Phi(T_0^{-1})$. For the measure $\mu_\Phi$ constructed there,
\begin{equation}\label{eq:rayleigh-profile}
\mathcal J_{\mu_\Phi}(E)\asymp\Phi(E)^{-1},
\qquad E\ge0,
\end{equation}
and therefore
\begin{equation}\label{eq:coupling-profile}
\lambda_{\mu_\Phi}(\gamma)
\asymp
\inf_{E\ge0}\left(E+\frac{\gamma}{\Phi(E)}\right),
\qquad \gamma>0.
\end{equation}

The form perturbation $H+\gamma\mu_\Phi$ also has a pathwise interpretation. Let $\mathcal A_t$ be the positive continuous additive functional of Brownian motion with Revuz measure $\mu_\Phi$, which is smooth by the trace bound, and let $d\pi=\mathrm{vol}_g(X)^{-1}d\mathrm{vol}_g$. Then, for $t>0$, $a\ge0$, and $\gamma>0$,
\begin{equation}\label{eq:additive-lower-tail}
\mathbb P_\pi(\mathcal A_t\le a)
\le
\exp\bigl(\gamma a-t\lambda_{\mu_\Phi}(\gamma)\bigr).
\end{equation}
Moreover, for quasi-every starting point $x$, under $\mathbb P_x$ almost every path $t\mapsto\mathcal A_t$ is continuous and strictly increasing, while $d\mathcal A_t$ is singular with respect to $dt$ on every finite time interval.
\end{proposition}

\begin{proof}
For $\|u\|_2=1$, put $q=\mathcal E(u)$. Then
$$
\mathcal E(u)+\gamma\int|\widetilde u|^2\,d\mu
\ge q+\gamma\mathcal J_\mu(q),
$$
which proves one inequality in \eqref{eq:variational-coupling}. Conversely, for each $E$, take normalised $u_n$ with $\mathcal E(u_n)\le E$ and $\int|\widetilde u_n|^2\,d\mu\to\mathcal J_\mu(E)$. Their Rayleigh quotients prove the reverse inequality.

For \eqref{eq:rayleigh-profile}, use the notation
$$
A_k=\sum_{\ell\ge k}a_\ell
=\frac{\Phi(T_0^{-1})}{\Phi(R_k^{-2})}
$$
from the construction of $\mu_\Phi$. If $\mathcal E(u)\le E$ and $\|u\|_2=1$, choose the least $k$ such that $2C_*R_k^2E\le\tfrac12$. For every $\ell\ge k$, \eqref{eq:one-scale} gives
$$
\int_X|\widetilde u|^2\,d\eta_\ell\ge(2C_*)^{-1}.
$$
Hence
$$
\int_X|\widetilde u|^2\,d\mu_\Phi
\ge\frac{A_k}{2C_*}
\ge c\Phi(E)^{-1},
$$
by the doubling condition on $\Phi$. For $E\ge T_0^{-1}$, the definition of $\mathcal C_{\mu_\Phi}(E)$ and its lower bound in \cref{thm:prescribed-loss} give normalised $f_n\in\operatorname{Ran}P_{[0,E]}(H)$ with $\int|f_n|^2\,d\mu_\Phi\le C\Phi(E)^{-1}+o(1)$. Since $\mathcal E(f_n)\le E$, this proves the reverse bound. For $E<T_0^{-1}$, use the normalised constant function. This proves \eqref{eq:rayleigh-profile}, and \eqref{eq:variational-coupling} gives \eqref{eq:coupling-profile}.

It remains to translate the coupling estimate into a statement about Brownian paths. The Feynman--Kac formula and the $L^2$ semigroup norm give
$$
\begin{aligned}
\mathbb E_\pi[e^{-\gamma\mathcal A_t}]
&=\frac1{\mathrm{vol}_g(X)}
\left\langle1,e^{-t(H+\gamma\mu_\Phi)}1\right\rangle\\
&\le e^{-t\lambda_{\mu_\Phi}(\gamma)}.
\end{aligned}
$$
Markov's inequality proves \eqref{eq:additive-lower-tail}. For fixed $L$, \eqref{eq:coupling-profile} gives
$$
\lambda_{\mu_\Phi}(\gamma)
\ge c\min\left\{L,\frac{\gamma}{\Phi(L)}\right\},
$$
so $\lambda_{\mu_\Phi}(\gamma)\to\infty$ as $\gamma\to\infty$. Taking $a=0$ in \eqref{eq:additive-lower-tail} shows that $\mathbb P_\pi(\mathcal A_t=0)=0$ for every $t>0$. Thus $h_t(y):=\mathbb P_y(\mathcal A_t=0)$ vanishes for $\operatorname{vol}_g$-almost every $y$. For quasi-every $x$, the Markov property at time $t/2$ and the heat-kernel density give
$$
\mathbb P_x(\mathcal A_t=0)
\le\mathbb E_x[h_{t/2}(B_{t/2})]=0.
$$
Additivity and the Markov property give the same conclusion for every rational increment. Continuity and monotonicity then prove strict increase.

Finally, choose a zero-volume Borel set $S$ carrying $\mu_\Phi$. The additive functional
$$
\int_0^t\mathbf 1_{X\setminus S}(B_s)\,d\mathcal A_s
$$
has zero Revuz measure and therefore vanishes for quasi-every starting point. On the other hand, the heat-kernel density gives
$$
\mathbb E_x\int_0^T\mathbf 1_S(B_t)\,dt=0
$$
for quasi-every $x$. Thus $d\mathcal A_t$ is carried by a set of times of Lebesgue measure zero. Taking rational intervals and integer horizons places both conclusions on one event of full probability.
\end{proof}

\paragraph{Power and logarithmic examples.}
If $\Phi(E)\asymp(1+T_0E)^\delta$, then, for $\gamma T_0\ge1$,
$$
\lambda_{\mu_\Phi}(\gamma)
\asymp T_0^{-1}(T_0\gamma)^{1/(1+\delta)}.
$$
If $\Phi(E)\asymp\bigl(1+\log(2+T_0E)\bigr)^\theta$, then
$$
\lambda_{\mu_\Phi}(\gamma)
\asymp
\frac{\gamma}
{\bigl(1+\log(2+T_0\gamma)\bigr)^\theta}.
$$

\begin{corollary}[Observation by measures moving at the diffusion scale]
\label{cor:time-dependent-observation}
After increasing $C_*$ if necessary, assume $C_*\ge1$. For $0<t\le T_0$, let $k(t)$ be the least integer such that $2C_*R_{k(t)}^2\le t$.
\begin{enumerate}
\item[\textup{(i)}] For every $u_0\in L^2(X)$ and $0<T\le T_0$,
$$
\|e^{-TH}u_0\|_{L^2(X)}^2\le\frac CT
\int_0^T\!\int_X|e^{-tH}u_0|^2\,d\eta_{k(t)}\,dt.
$$
\item[\textup{(ii)}] For fixed $T$, define a measure on $(0,T)\times X$ by
$$
d\boldsymbol\eta_T(s,x)=ds\,d\eta_{k(T-s)}(x).
$$
For every $y_0\in L^2(X)$ there is $v\in L^2(\boldsymbol\eta_T)$ such that the weak solution of $\partial_sy+Hy=v(s)\eta_{k(T-s)}$, $y(0)=y_0$, satisfies $y(T)=0$ and
$$
\int_0^T\!\int_X|v(s,x)|^2\,d\eta_{k(T-s)}(x)\,ds
\le\frac CT\|y_0\|_{L^2(X)}^2.
$$
Here $\langle v(s)\eta_{k(T-s)},\varphi\rangle =\int_Xv(s)\overline{\widetilde\varphi}\,d\eta_{k(T-s)}$. The control cost is therefore at most $CT^{-1/2}$.
\item[\textup{(iii)}] This order is optimal among moving measures of uniformly bounded mass. More precisely, suppose a measurable family $(\mu_s)$ defines a bounded map
\[
L^2(ds\,d\mu_s)\longrightarrow L^2((0,T);H^{-1}(X)).
\]
If $0<M:=\sup_s\mu_s(X)<\infty$, its control cost is at least $(\mathrm{vol}_g(X)/M)^{1/2}T^{-1/2}$.
\end{enumerate}
Each $\eta_{k(t)}$ is carried by a zero-volume set of Hausdorff dimension $m-2$, and $R_{k(t)}\asymp\sqrt t$.
\end{corollary}

\begin{proof}
Set $u(t)=e^{-tH}u_0$, $F(t)=\|u(t)\|_2^2$, and $I(t)=\int_X|u(t)|^2\,d\eta_{k(t)}$. The Brownian quadratic-variation calculation in the proof of \cref{thm:heat-control} gives $F'(t)=-\|\nabla u(t)\|_2^2$. Thus, for almost every $t>0$, \eqref{eq:one-scale} gives
$$
F(t)\le\frac t2(-F'(t))+CI(t),
$$
so $(t^2F(t))'\le CtI(t)$. Integrating from $\varepsilon$ to $T$ and letting $\varepsilon\downarrow0$ gives
$$
T^2F(T)\le C\int_0^TtI(t)\,dt
\le CT\int_0^TI(t)\,dt,
$$
which proves \textup{(i)}. The function $k(T-s)$ is piecewise constant, and \cref{lem:multiscale-trace} makes the displayed control map bounded from $L^2(\boldsymbol\eta_T)$ to $L^2((0,T);H^{-1}(X))$. After the change of variables $t=T-s$, the estimate in \textup{(i)} is the coercivity estimate for the adjoint of the final-state map. The Riesz representation theorem proves \textup{(ii)}.

For the lower bound, let $(\mu_s)$ satisfy $\sup_s\mu_s(X)\le M$, take $y_0=1$, and test the controlled equation against $1$. Then
$$
\operatorname{vol}_g(X)
=\left|\int_0^T\!\int_Xv\,d\mu_s\,ds\right|
\le (TM)^{1/2}\|v\|_{L^2(ds\,d\mu_s)}.
$$
Division by $\|y_0\|_2=\operatorname{vol}_g(X)^{1/2}$ proves the lower bound. Finally, minimality and $R_{k-1}=2R_k$ give $t/(8C_*)<R_{k(t)}^2\le t/(2C_*)$.
\end{proof}

For these moving measures, the factor $\Phi(T^{-1})$ present for the fixed measures $\mu_\Phi$ disappears. Thus the control cost improves by $\Phi(T^{-1})^{1/2}$ while the observation at every time remains supported on a set of Hausdorff dimension $m-2$.

\begin{proposition}[Necessity of divergence and full support]
\label{prop:divergence-full-support-necessary}
Let $\mu$ be a finite Borel measure on $X$.
\begin{enumerate}
\item[\textup{(i)}] $\sup_{E\ge0}\mathcal C_\mu(E)<\infty$ if and only if $d\operatorname{vol}_g\le C\,d\mu$ for some $C<\infty$. Consequently, if $\mu$ is carried by a zero-volume set, then $\mathcal C_\mu(E)\nearrow\infty$ as $E\to\infty$.
\item[\textup{(ii)}] If $\operatorname{supp}\mu\ne X$, then, for every $N>0$, $\mathcal C_\mu(E)\ge c_NE^N$ for all sufficiently large $E$, and $\mathcal H_\mu(T)\ge c e^{c/T}$ for all sufficiently small $T>0$.
\end{enumerate}
\end{proposition}

\begin{proof}
\textup{(i)} The measure inequality immediately gives the uniform bound. Conversely, suppose $\mathcal C_\mu(E)\le C$ for all $E$. If $u$ is smooth, then $P_{[0,E]}u\to u$ uniformly, and the finiteness of $\mu$ gives
$$
\|u\|_{L^2(X)}^2\le C\int_X|u|^2\,d\mu.
$$
Uniform approximation extends this to continuous $u$. Approximating indicators of open sets by continuous functions proves $d\operatorname{vol}_g\le C\,d\mu$. A measure carried by a zero-volume set cannot satisfy this inequality; monotonicity of $\mathcal C_\mu$ proves the last assertion.

\textup{(ii)} The result is immediate if $\mu=0$. Otherwise, choose $z\in X$, $\rho>0$, and $\psi\in C_c^\infty(B(z,\rho))$ such that $\|\psi\|_2=1$ and $B(z,2\rho)\cap\operatorname{supp}\mu=\varnothing$. Put $f_E=P_{[0,E]}\psi$. Then $\|f_E\|_2\to1$, while, for every $M>0$ and fixed $s>m/2$,
$$
\|f_E\|_{L^\infty(\operatorname{supp}\mu)}
\le C\|(I-P_{[0,E]})\psi\|_{H^s(X)}
\le C_M(1+E)^{-M}.
$$
Thus $\int_X|f_E|^2\,d\mu\le C_M(1+E)^{-2M}$, proving the spectral claim.

For $u_0=\psi$, the Brownian formula gives, on $\operatorname{supp}\mu$,
$$
|e^{-tH}\psi(x)|=|\mathbb E_x\psi(B_t)|
\le\|\psi\|_\infty\mathbb P_x\{d(B_t,x)\ge\rho\}
\le Ce^{-c\rho^2/t},
$$
by the Gaussian Brownian tail. Therefore
$$
\int_0^T\!\int_X|e^{-tH}\psi|^2\,d\mu\,dt
\le Ce^{-c\rho^2/T}
$$
for small $T$. Since $\|e^{-TH}\psi\|_2\to1$, the definition of $\mathcal H_\mu(T)$ gives $\mathcal H_\mu(T)\ge cT e^{c\rho^2/T}\ge c'e^{c'/T}$ after decreasing $T$.
\end{proof}

\begin{proposition}[Necessary conditions on sets carrying the measure]
\label{prop:control-dim-sharp}
Let $m\ge2$, and let $\mu$ be a nonzero finite measure on $X$ such that
$$
\int_X|\widetilde u|^2\,d\mu\le N\|u\|_{H^1(X)}^2,
\qquad u\in H^1(X).
$$
There are $C,r_0>0$ such that
$$
\mu(B(x,\rho))\le CN
\begin{cases}
\rho^{m-2},&m\ge3,\\
\bigl(1+\log(r_0/\rho)\bigr)^{-1},&m=2,
\end{cases}
\qquad 0<\rho\le r_0/4.
$$
Every Borel set carrying $\mu$ is nonpolar. If $m\ge3$, every such set has Hausdorff dimension at least $m-2$. In dimension $2$, $\mu$ has no atoms, and the construction in \cref{def:multiscale-measure,lem:multiscale-trace} shows that the lower bound $\dim_{\mathrm H}S\ge0$ for a Borel set $S$ carrying $\mu$ is sharp.
\end{proposition}

\begin{proof}
For $m\ge3$, use a cutoff equal to $1$ on $B(x,\rho)$, supported in $B(x,2\rho)$, and with gradient bounded by $C/\rho$. Its squared $H^1$ norm is at most $C\rho^{m-2}$.

For $m=2$, take $r_0$ below the uniform injectivity radius and use the logarithmic cutoff
$$
u_\rho(y)=
\begin{cases}
1,&d(x,y)\le\rho,\\
\dfrac{\log(r_0/d(x,y))}{\log(r_0/\rho)},
&\rho<d(x,y)<r_0,\\
0,&d(x,y)\ge r_0.
\end{cases}
$$
Bounded geometry and direct integration give
$$
\|u_\rho\|_{H^1(X)}^2
\le\frac{C}{1+\log(r_0/\rho)}.
$$
Approximation by smooth cutoffs and the trace inequality prove the ball bounds and, in dimension $2$, the absence of atoms.

If a Borel set $S$ carrying $\mu$ were polar, there would be smooth functions $u_n\ge1$ near $S$ with $\|u_n\|_{H^1}\to0$. Since $\mu$ is carried by $S$,
$$
0<\mu(X)\le\int_X|u_n|^2\,d\mu
\le N\|u_n\|_{H^1(X)}^2\longrightarrow0,
$$
a contradiction. For $m\ge3$, covering such a set by balls and using the first ball bound gives its Hausdorff-dimension lower bound.
\end{proof}

The trace condition has an equivalent control-theoretic formulation. By taking adjoints, it holds exactly when
\[
\langle Bv,u\rangle=\int_Xv\,\overline u\,d\mu,
\qquad u\in C^\infty(X),
\]
extends to a bounded map $B:L^2(\mu)\to H^{-1}(X)$.

\begin{remark}[Comparison with related work]\label{rem:concurrent-comparison}
Filbir and Mhaskar \cite[Definition~5.4\textup{(b)} and Theorem~5.5\textup{(b)}]{FilbirMhaskar2011} prove lower sampling inequalities for diffusion-polynomial spaces with respect to finite Borel measures satisfying a uniform lower bound on balls. Applied to the Laplace spectrum of a fixed closed manifold, their theorem gives, for all sufficiently small $R$, the fixed-manifold counterparts of both inequalities in \cref{thm:capacity-sampling}, with constants depending on $X$, for $E\le c_XR^{-2}$ independently of $\Lambda$. For the second inequality, apply their theorem to $R^m\nu_R$. For the first, when $f$ is real-valued, the intermediate value theorem gives $y_j\in B_j$ such that $f(y_j)=\int f\,d\nu_j$; apply their theorem to $R^m\sum_j\delta_{y_j}$. For complex $f$, apply the argument separately to its real and imaginary parts. Maximality of the net makes each of these measures uniformly $2R$-dominant. For equilibrium measures, $\tau_R=R^2\mu_R$ satisfies $\tau_R(B(x,2R))\ge\kappa R^m$, so their result also gives the fixed-manifold $L^2$ conclusion of \cref{cor:zero-volume} through a window $E\le c_XR^{-2}$ independent of $\kappa$.

What is used additionally here is the local estimate for arbitrary $H^1$ functions, the equilibrium-measure form inequality and its Brownian killing proof, and constants stated in terms of the displayed bounded-geometry data. The lower ball bounds for $\mu_\theta$ and $\mu_\Phi$ also allow the theorem of Filbir and Mhaskar to recover the upper spectral estimates on each fixed manifold. It does not provide the construction or bounded $H^1$ trace of these measures, the matching lower spectral estimates, or the heat-observability, null-control, coupling, and additive-functional conclusions.

Point sampling of low-frequency spaces by atomic measures indexed by the spectral cutoff was established in \cite{OrtegaCerdaPridhnani2012}. For $m\ge3$, Stollmann--Stolz \cite{StollmannStolz2021} treat positive-volume observation in convex Euclidean domains when every $R$-ball centred in the domain contains a $\delta$-ball in the observation set; their proof uses a mixed Neumann--Dirichlet Laplacian whose spectral bottom is bounded below by $c_m\delta^{m-2}R^{-m}$. Here the observation sets may have zero volume, and observation is against their relative equilibrium measures. Burq--Germain--Sorella--Zhu
\cite[Theorems~1.1 and~1.2]{BGSZ26} prove that, for a Borel probability
measure on the torus, a uniform trace inequality for Laplace eigenfunctions
forces the measure to be upper $(m-2)$-regular, and hence its support to
have Hausdorff dimension at least $m-2$, while uniform observability forces
the upper Minkowski dimension of the support to be at least $m-2$. Their inequalities concern
individual toral eigenspaces rather than growing spectral subspaces and
do not address the bounded $H^1$-trace or heat-observability framework
used here; their methods use the cluster structure of lattice points on
spheres and decoupling. Foster--Gallegos \cite{FG25} prove, in Euclidean balls, propagation of smallness for gradients of harmonic functions from sets of positive $(m-2+\delta)$-dimensional Hausdorff content.
\end{remark}

\begin{remark}[Comparison with set observation]
\label{rem:heat-control-context}
Green--Le Balc'h--Martin--Orsoni
\cite[Theorem~1.4 and Corollary~3.9]{GLBMO24} prove heat observability
from sets of positive $s$-dimensional Hausdorff content for $s>m-1$
and, on cubes, construct countable observable product sets of Hausdorff
dimension zero. Huang--Wang--Wang \cite{HuangWangWang2025} treat logarithmic contents at dimension $m-1$. Both use spatial suprema and the Lebeau--Robbiano method \cite{LebeauRobbiano1995}; here the norm is $L^2(\mu)$ and the endpoint comes from the bounded trace $H^1(X)\to L^2(\mu)$. It remains open whether the analogous construction for the Dirichlet-to-Neumann operator, or equivalently for the boundary process of reflecting Brownian motion, controls the boundary heat equation from a measure carried by a set of dimension $m-2$.
\end{remark}

\section{The critical Feynman--Kac gauge in the radial model}
\label{sec:critical-gauge}

All solutions are real-valued; for complex solutions, replace $u^2$ by $|u|^2$. Fix $m\ge2$, let $V$ be continuous on $[0,R)$, and suppose
\begin{equation}\label{gauge:eq:pde}
\left(\tfrac12\Delta+V(|x|)\right)u=0
\quad\text{in }B_R\subset\mathbb R^m.
\end{equation}
Let $\tau_r$ be the first exit time of Brownian motion from $B_r$. Let $g$ be the regular radial solution of $(\tfrac12\Delta+2V(|x|))g=0$, $g(0)=1$, and set $R_*:=\sup\{0<r<R:g>0\text{ on }[0,r]\}$. For $0<r<R_*$,
\begin{equation}\label{gauge:eq:height}
H_u(r):=\mathbb E_0\!\left[e^{2\int_0^{\tau_r}V(|B_s|)\,ds}
u(B_{\tau_r})^2\right]
=\frac1{g(r)}\fint_{\partial B_r}u^2\,d\sigma
\end{equation}
by \cref{gauge:lem:gauge}.

The doubled potential in $g$ is forced by squaring the Feynman--Kac martingale, so the quotient in \eqref{gauge:eq:height} is the natural height rather than an auxiliary normalisation. We state the convexity and equality results first and then derive them from the spherical-harmonic decomposition and a radial Riccati identity.

\begin{theorem}[Convexity of the doubled-gauge height]\label{gauge:thm:main}
If $r\mapsto r^2V(r)$ is nondecreasing, then $V\ge0$ and $t\mapsto\log H_u(e^t)$ is convex on $(-\infty,\log R_*)$ for every nonzero solution. Equivalently, $\tfrac r2\partial_r\log H_u(r)$ is nondecreasing. Thus, for $0<r_1<r_2<r_3<R_*$,
\begin{equation}\label{gauge:eq:threesphere}
H_u(r_2)\le H_u(r_1)^\theta H_u(r_3)^{1-\theta},
\qquad
\theta=\frac{\log(r_3/r_2)}{\log(r_3/r_1)}.
\end{equation}
\end{theorem}

If $m=2$, the conclusion of \cref{gauge:thm:main} holds under $V\ge0$ alone.

\begin{theorem}[Failure below the doubled coefficient]\label{gauge:thm:threshold}
Let $0\le b<2$, suppose $V(0)>0$, and let $g_{bV}$ be the regular radial solution of $(\tfrac12\Delta+bV(|x|))g_{bV}=0$, $g_{bV}(0)=1$. If $f_0$ is the regular radial solution of \eqref{gauge:eq:pde}, $f_0(0)=1$, then
\[
\frac{d^2}{dt^2}\log\frac{f_0(e^t)^2}{g_{bV}(e^t)}
=\frac{4(b-2)V(0)}m e^{2t}+o(e^{2t})
\quad\text{as }t\to-\infty.
\]
Thus the radial solution is not log-convex near $0$ for any such $V$ and $b$.
\end{theorem}

\begin{proposition}[Failure of convexity for a compactly supported potential in \texorpdfstring{$m\ge3$}{m >= 3}]
\label{gauge:prop:counterexample}
Let $m\ge3$ and let $V\ge0$ be continuous on $[0,\infty)$, nonzero, supported in $[0,r_1]$, and satisfy
\begin{equation}\label{gauge:eq:smallness}
\frac4{m-2}\int_0^\infty sV(s)\,ds<1.
\end{equation}
Then $R_*=\infty$. For every $r_2>r_1$, some solution $u(r,\omega)=f_\ell(r)Y_\ell(\omega)$, where $\ell\ge0$ and $Y_\ell$ is a spherical harmonic, solves \eqref{gauge:eq:pde} on $\mathbb R^m$ and has $t\mapsto\log H_u(e^t)$ strictly concave near $\log r_2$.
\end{proposition}

\begin{theorem}[Characterisation]\label{gauge:thm:characterisation}
Let $V\ge0$ and $0<\rho\le R_*$. The following are equivalent.
\begin{enumerate}
\item[(a)] For every nonzero solution, $t\mapsto\log H_u(e^t)$ is convex on $(-\infty,\log\rho)$.
\item[(b)] The function $r\mapsto rg'(r)/g(r)$ is nonincreasing on $(0,\rho)$, equivalently $t\mapsto\log g(e^t)$ is concave there.
\end{enumerate}
The hypothesis of \cref{gauge:thm:main} implies \textup{(b)}. The converse already fails in dimension $3$, even on the whole interval $(0,R_*)$.
\end{theorem}

For the equality statement, use the normalised spherical-harmonic expansion
\[
u(r,\omega)=\sum_{\ell,j}c_{\ell j}f_\ell(r)Y_{\ell j}(\omega),
\]
where $f_\ell$ is the regular $\ell$th radial mode of \eqref{gauge:eq:pde}, normalised by $r^{-\ell}f_\ell(r)\to1$ as $r\downarrow0$. Set $A_\ell:=\sum_j|c_{\ell j}|^2$ and put
\[
G_\ell(t):=\frac{d^2}{dt^2}\log\frac{f_\ell(e^t)^2}{g(e^t)},\qquad
\mu_{\ge1}(t):=
\frac{\sum_{\ell\ge1}A_\ell f_\ell(e^t)^2/g(e^t)}{H_u(e^t)},
\]
and let $r_V:=\inf\{r>0:V\not\equiv0\text{ on }(0,r)\}$, with $r_V=\infty$ if $V\equiv0$.

\begin{theorem}[Convexity modulus and equality]\label{gauge:thm:modulus}
Under the hypothesis of \cref{gauge:thm:main}, every nonzero solution and every $t<\log R_*$ satisfy
\begin{equation}\label{gauge:eq:modulus}
\frac{d^2}{dt^2}\log H_u(e^t)\ge G_0(t).
\end{equation}
The integral formula for $G_0$ is \eqref{gauge:eq:defectrep}. Moreover, $G_0(t)>0$ for $e^t>r_V$, and equality in \eqref{gauge:eq:modulus} at one such $t$ holds if and only if $u$ is a constant multiple of $f_0$. Finally,
\begin{equation}\label{gauge:eq:stability}
\frac{d^2}{dt^2}\log H_u(e^t)-G_0(t)
\ge\mu_{\ge1}(t)\bigl(G_1(t)-G_0(t)\bigr).
\end{equation}
\end{theorem}

\subsection{Radial equations and spherical harmonics}
\label{gauge:sec:setup}

Put $\nu=(m-2)/2$. For $\ell\ge0$, set $L_\ell=\ell(\ell+m-2)$ and $\beta_\ell=\ell+\nu$, so that $L_\ell+\nu^2=\beta_\ell^2$. Let $f_\ell$ be the regular solution
\begin{equation}\label{gauge:eq:modeODE}
f_\ell''+\frac{m-1}{r}f_\ell'
+\left(2V(r)-\frac{L_\ell}{r^2}\right)f_\ell=0,
\qquad
\lim_{r\downarrow0}r^{-\ell}f_\ell(r)=1.
\end{equation}
Writing $f_\ell=r^\ell h_\ell$ gives $(r^{2\ell+m-1}h_\ell')'=-2r^{2\ell+m-1}Vh_\ell$; its Volterra form gives this solution uniquely and
\begin{equation}\label{gauge:eq:frobenius}
\frac{rf_\ell'(r)}{f_\ell(r)}=\ell+O(r^2)
\qquad(r\downarrow0).
\end{equation}
The function $g$ satisfies
\begin{equation}\label{gauge:eq:gaugeODE}
g''+\frac{m-1}{r}g'+4V(r)g=0,
\qquad g(0)=1,\quad g'(0)=0,
\end{equation}
and $R_*:=\sup\{0<r<R:g>0\text{ on }[0,r]\}$.

\begin{lemma}[Gauge identity]\label{gauge:lem:gauge}
For $0<r<R_*$, the operator $\tfrac12\Delta+2V(|x|)$ is subcritical on $B_r$, and
\begin{align*}
\mathbb E_x\!\left[e^{2\int_0^{\tau_r}V(|B_s|)\,ds}\right]
&=\frac{g(|x|)}{g(r)},\\
\mathbb E_0\!\left[e^{2\int_0^{\tau_r}V(|B_s|)\,ds}
F(B_{\tau_r})\right]
&=\frac1{g(r)}\fint_{\partial B_r}F\,d\sigma
\end{align*}
for $x\in B_r$ and every $F\in C(\partial B_r)$.
\end{lemma}

\begin{proof}
Choose $r<r'<R_*$. The positive solution $h(x)=g(|x|)$ on $B_{r'}$ and strict domain monotonicity give $\lambda_1(-\tfrac12\Delta-2V;B_r)>0$. Since $h$ is bounded above and away from zero on $\overline{B_r}$, the pair $(B_r,q)$, where $q(x)=2V(|x|)$, is gaugeable by \cite[Theorem~4.17]{CZ95}. The Feynman--Kac representation \cite[Theorem~4.7\textup{(v)}]{CZ95} therefore gives
\[
h(x)=\mathbb E_x\!\left[
e^{2\int_0^{\tau_r}V(|B_s|)\,ds}h(B_{\tau_r})\right].
\]
Since $h=g(r)$ on $\partial B_r$, this proves the first identity. At $x=0$, rotational invariance makes the weighted exit measure uniform; by the first identity its total mass is $1/g(r)$, proving the second.
\end{proof}

Projection onto spherical harmonics orthonormal for normalised surface measure gives the coefficients $c_{\ell j}f_\ell(r)$; Parseval and the notation above yield
\begin{equation}\label{gauge:eq:modesum}
H_u(r)=\sum_{\ell\ge0}A_\ell\Psi_\ell(r),
\qquad
\Psi_\ell(r):=\frac{f_\ell(r)^2}{g(r)}.
\end{equation}

If $V\ge0$, then $f_\ell>0$ on $(0,R_*)$ for every $\ell\ge0$.
Indeed, suppose $f_\ell(r_1)=0$ for some $r_1<R_*$. For a nonzero spherical harmonic $Y_\ell$, the function $w(x):=f_\ell(|x|)Y_\ell(x/|x|)$ solves the equation in $B_{r_1}$, extends regularly across $0$, and vanishes on the boundary. The stopped local martingale
$$
e^{\int_0^{t\wedge\tau_{r_1}}V(|B_s|)\,ds}
w(B_{t\wedge\tau_{r_1}})
$$
is $L^2$-bounded by \cref{gauge:lem:gauge}, hence uniformly integrable. Optional stopping gives $w\equiv0$, a contradiction. Thus $f_\ell$ has no zero before $R_*$ and is positive there by its normalisation at $0$.

\subsection{Riccati equations and an integral formula}
\label{gauge:sec:riccati}

Set $t=\log r$, let a dot denote $d/dt$, and put
\[
T_*:=\log R_*,\qquad w(t):=2e^{2t}V(e^t)=O(e^{2t})
\quad(t\to-\infty).
\]
Thus $w\ge0$ when $V\ge0$. Wherever $f_\ell$ and $g$ are nonzero, set
\[
N_\ell:=\frac{rf_\ell'}{f_\ell},\qquad N_g:=\frac{rg'}g,\qquad
X_\ell:=N_\ell+\nu,\qquad Y:=N_g+\nu.
\]
Substitution in the radial equations gives $\dot N_\ell=-(m-2)N_\ell-N_\ell^2+L_\ell-w$ and $\dot N_g=-(m-2)N_g-N_g^2-2w$. Hence
\begin{equation}\label{gauge:eq:system}
\dot X_\ell=\beta_\ell^2-X_\ell^2-w,
\qquad
\dot Y=\nu^2-Y^2-2w.
\end{equation}
Moreover,
\begin{equation}\label{gauge:eq:initialdata}
X_\ell(t)=\beta_\ell+O(e^{2t}),
\qquad
Y(t)=\nu+O(e^{2t})
\quad(t\to-\infty).
\end{equation}
When $V\ge0$, this positivity makes these functions finite on $(-\infty,T_*)$. Since $\Psi_\ell=f_\ell^2/g$, the previously defined $G_\ell$ satisfies
\begin{equation}\label{gauge:eq:state}
G_\ell(t)=2\dot X_\ell-\dot Y
=Y^2-2X_\ell^2+2\beta_\ell^2-\nu^2.
\end{equation}
The terms containing $w$ cancel: squaring the martingale produces exactly the doubled Feynman--Kac weight used in $g$. In particular, $G_\ell(t)\to0$ as $t\to-\infty$.

We use variation of constants, also in its Stieltjes form; all boundary terms at $-\infty$ vanish by \eqref{gauge:eq:initialdata} and $w(t)=O(e^{2t})$. In the borderline case $\beta_0=0$ of dimension two, the homogeneous factors are bounded rather than decaying, which still suffices, since the terms they multiply vanish at $-\infty$.

\begin{lemma}[Integral formula]\label{gauge:lem:defect}
If $V\ge0$, then on $(-\infty,T_*)$,
\begin{equation}\label{gauge:eq:defectODE}
\dot G_\ell=-2X_\ell G_\ell-2(X_\ell-Y)\dot Y,
\end{equation}
and
\begin{equation}\label{gauge:eq:defectrep}
G_\ell(t)=2\int_{-\infty}^t
e^{-2\int_s^tX_\ell(\tau)\,d\tau}
(X_\ell-Y)(s)(-\dot Y(s))\,ds.
\end{equation}
\end{lemma}

\begin{proof}
Using \eqref{gauge:eq:system} and $2(\beta_\ell^2-X_\ell^2)=G_\ell+\nu^2-Y^2$ gives
\[
\begin{aligned}
\dot G_\ell
&=-2X_\ell G_\ell
+2(X_\ell-Y)(2w-\nu^2+Y^2)\\
&=-2X_\ell G_\ell-2(X_\ell-Y)\dot Y.
\end{aligned}
\]
Apply variation of constants and let $a\to-\infty$. By \eqref{gauge:eq:initialdata} and \eqref{gauge:eq:system}, $G_\ell(a)=O(e^{2a})$, $(2X_\ell)_-$ is integrable near $-\infty$, and $(X_\ell-Y)\dot Y=O(e^{2t})$.
\end{proof}

\subsection{Comparison lemmas}\label{gauge:sec:comparison}

Throughout this subsection $V\ge0$, and hence $w\ge0$.

\begin{lemma}\label{gauge:lem:comparison}
On $(-\infty,T_*)$:
\begin{enumerate}
\item[(i)] $X_\ell\le\beta_\ell$ and $Y\le\nu$;
\item[(ii)] $X_\ell\ge X_{\ell'}$ if $\ell\ge\ell'$, strictly at every finite $t$ if $\ell>\ell'$;
\item[(iii)] $X_\ell\ge Y$ for every $\ell$, and $X_0(t)>Y(t)$ if $w$ is not identically zero on $(-\infty,t]$.
\end{enumerate}
\end{lemma}

\begin{proof}
The relevant differences satisfy
$$
\begin{aligned}
\frac{d}{dt}(X_\ell-\beta_\ell)
  &=-(X_\ell+\beta_\ell)(X_\ell-\beta_\ell)-w,\\
\frac{d}{dt}(X_\ell-X_{\ell'})
  &=\beta_\ell^2-\beta_{\ell'}^2
    -(X_\ell+X_{\ell'})(X_\ell-X_{\ell'}),\\
\frac{d}{dt}(X_0-Y)&=-(X_0+Y)(X_0-Y)+w.
\end{aligned}
$$
Variation of constants and \eqref{gauge:eq:initialdata} yield \textup{(i)} from the first equation and its $Y-\nu$ analogue, and \textup{(ii)} from the second. The third gives
$$
X_0(t)-Y(t)=\int_{-\infty}^t
\exp\!\left(-\int_s^t(X_0(\tau)+Y(\tau))\,d\tau\right)w(s)\,ds,
$$
which proves \textup{(iii)} together with \textup{(ii)}; positivity of the sources gives the strict statements.
\end{proof}

The same comparison gives monotonicity of boundary variance in the radial model. In the setting of \cref{gauge:sec:setup}, assume $V\ge0$ and $u(0)\ne0$. Let $\nu_r$ be the normalised exit distribution weighted by
\[
\exp\left(\int_0^{\tau_r}V(|B_s|)\,ds\right),
\]
and let $\delta_\nu(r)$ be the relative variance of $u|_{\partial B_r}$ under $\nu_r$. Then $r\mapsto\delta_\nu(r)$ is nondecreasing on $(0,R_*)$, strictly unless $u$ is radial.
Indeed, rotational invariance makes $\nu_r$ normalised surface measure. With the notation of \eqref{gauge:eq:modesum},
\[
\frac{\delta_\nu(r)}{1-\delta_\nu(r)}
=\frac{\sum_{\ell\ge1}A_\ell f_\ell(r)^2}{A_0f_0(r)^2}.
\]
By \cref{gauge:lem:comparison}, for $\ell\ge1$,
\[
\frac d{d\log r}\log\frac{f_\ell(r)^2}{f_0(r)^2}
=2(X_\ell-X_0)>0.
\]
The claim follows term by term.

\begin{lemma}[Monotonicity of $Y$]\label{gauge:lem:classmonotone}
Suppose $w$ is nondecreasing, equivalently $r\mapsto r^2V(r)$ is nondecreasing. Then $\dot Y\le0$. Moreover, for every $\ell$,
\begin{equation}\label{gauge:eq:Eell}
E_\ell:=-\dot Y-G_\ell,
\qquad E_\ell(t)=2\int_{(-\infty,t]}
\exp\!\left(-2\int_s^tX_\ell(\tau)\,d\tau\right)\,dw(s)\ge0.
\end{equation}
\end{lemma}

\begin{proof}
For $Z:=\dot Y$, \eqref{gauge:eq:system} gives $dZ=-2YZ\,dt-2\,dw$ and $Z(t)=O(e^{2t})$ as $t\to-\infty$. Thus
$$
\dot Y(t)=-2\int_{(-\infty,t]}
\exp\!\left(-2\int_s^tY(\tau)\,d\tau\right)\,dw(s)\le0.
$$
Since $E_\ell=-\dot Y-G_\ell$, combining $d\dot Y=-2Y\dot Y\,dt-2\,dw$ with \eqref{gauge:eq:defectODE} gives
\[
dE_\ell=-2X_\ell E_\ell\,dt+2\,dw.
\]
Variation of constants, together with $E_\ell(t)\to0$ as $t\to-\infty$, gives \eqref{gauge:eq:Eell}. All integrals against $dw$ are Stieltjes integrals; continuity and monotonicity suffice.
\end{proof}

\begin{proof}[Proof of \cref{gauge:thm:main}]
The hypothesis makes $w$ nondecreasing with limit zero at $-\infty$, and therefore $V\ge0$. By \cref{gauge:lem:classmonotone,gauge:lem:comparison}, $-\dot Y\ge0$ and $X_\ell\ge Y$; hence \eqref{gauge:eq:defectrep} gives $G_\ell\ge0$. Thus every $\Psi_\ell(e^t)$ is log-convex. H\"older's inequality gives the same conclusion for finite nonnegative sums. The nonzero partial sums in \eqref{gauge:eq:modesum} increase pointwise to $H_u$, so their three-point inequalities pass to the limit. The frequency and three-sphere statements follow.

If $m=2$, then $\nu=0$ and $\dot Y=-Y^2-2w\le0$ under $V\ge0$ alone. By \cref{gauge:lem:comparison}, $X_\ell\ge Y$, so \eqref{gauge:eq:defectrep} gives $G_\ell\ge0$ for every $\ell$. H\"older's inequality gives the three-point inequality for each finite mode sum, and monotone convergence in \eqref{gauge:eq:modesum} gives the result.
\end{proof}

\subsection{Convexity modulus and equality}\label{gauge:sec:modulus}

Assume the hypothesis of \cref{gauge:thm:main} throughout this subsection. The preceding proof and \eqref{gauge:eq:Eell} give
\begin{equation}\label{gauge:eq:Gceiling}
0\le G_\ell(t)\le-\dot Y(t),
\qquad \ell\ge0,\quad t<T_*.
\end{equation}

For a nonzero finite sum $h(t)=\sum_{\ell\le L}A_\ell\Psi_\ell(e^t)$, set
$$
\mu_\ell(t):=\frac{A_\ell\Psi_\ell(e^t)}{h(t)},
\qquad \mu_{\ge1,L}(t):=\sum_{\ell=1}^L\mu_\ell(t).
$$
These are the contributions of the spherical harmonic degrees to the Brownian second moment. Twice differentiating gives
\begin{equation}\label{gauge:eq:variance}
(\log h)''
=\sum_\ell\mu_\ell G_\ell
+\operatorname{Var}_{\mu(t)}(2X_\ell-Y-\nu).
\end{equation}

\begin{proof}[Proof of \cref{gauge:thm:modulus}]
We first show that, for every $t<T_*$, $G_\ell(t)\ge G_{\ell'}(t)$ whenever $\ell\ge\ell'$, with strict inequality if $\ell>\ell'$ and $w$ is not identically zero on $(-\infty,t]$. For $D:=G_\ell-G_{\ell'}$, \cref{gauge:lem:defect} gives
$$
\dot D=-2X_\ell D+2(X_\ell-X_{\ell'})E_{\ell'},
\qquad D(t)\to0\quad(t\to-\infty).
$$
The source is nonnegative by \cref{gauge:lem:comparison} and \eqref{gauge:eq:Eell}, so $D\ge0$. If $\ell>\ell'$, then $X_\ell-X_{\ell'}>0$; if $w$ has been nonzero before $t$, then $E_{\ell'}>0$ on a nonempty interval below $t$, proving strictness.

For finite sums, \eqref{gauge:eq:variance} and this monotonicity give
$$
(\log h)''\ge G_0+\mu_{\ge1,L}(G_1-G_0)\ge G_0.
$$
By Parseval, $h_L\uparrow H_u(e^t)$ and $\mu_{\ge1,L}\to\mu_{\ge1}$. Local uniform convergence of $\log h_L$ shows that the inequalities pass to the limit in the sense of distributions; they hold pointwise because $H_u(e^t)$ is positive and $C^2$: by the divergence theorem, $\frac d{dr}\fint_{\partial B_r}u^2\,d\sigma$ equals $|\partial B_r|^{-1}\int_{B_r}\Delta(u^2)$, and $\Delta(u^2)=2|\nabla u|^2-4Vu^2$ is continuous.

If $e^t>r_V$, then $X_0-Y$ and $-\dot Y$ are positive on a nonempty interval below $t$. Thus \eqref{gauge:eq:defectrep} gives $G_0(t)>0$. If equality holds there, the stronger inequality and $G_1(t)>G_0(t)$ force $\mu_{\ge1}(t)=0$; positivity of every $\Psi_\ell$ then gives $A_\ell=0$ for $\ell\ge1$. Hence $u$ is radial. The converse follows directly from the definition of $G_0$.
\end{proof}

\begin{corollary}[Three-sphere inequality for surface means]
\label{gauge:cor:threesphere}
Under the hypotheses of \cref{gauge:thm:main}, if $0<r_1<r_2<r_3<R_*$ and $\vartheta=\log(r_3/r_2)/\log(r_3/r_1)$, then
$$
\fint_{\partial B_{r_2}}u^2\,d\sigma\le
\frac{f_0(r_2)^2}{f_0(r_1)^{2\vartheta}f_0(r_3)^{2(1-\vartheta)}}
\left(\fint_{\partial B_{r_1}}u^2\,d\sigma\right)^{\!\vartheta}
\left(\fint_{\partial B_{r_3}}u^2\,d\sigma\right)^{\!1-\vartheta}.
$$
The constant is optimal: $u=f_0$ gives equality, and equality for any triple with $r_3>r_V$ forces $u$ to be a constant multiple of $f_0$.
\end{corollary}

\begin{proof}
By \cref{gauge:thm:modulus} and \eqref{gauge:eq:height},
$$
t\longmapsto
\log\frac{\fint_{\partial B_{e^t}}u^2\,d\sigma}{f_0(e^t)^2}
$$
is convex. Its three-point inequality is the display. It is constant for $u=f_0$. Conversely, equality makes it affine on $[\log r_1,\log r_3]$; if $r_3>r_V$, its second derivative vanishes at an interior point where $e^t>r_V$, and \cref{gauge:thm:modulus} applies.
\end{proof}

\paragraph{Constant potential.}
If $V\equiv a>0$ and $0<r<R_*$, the Brownian exit-time moment $q_m$ and the product formula for $J_\nu$ give, with $r=e^t$ and $z=\sqrt{2a}\,r$,
$$
G_0(t)=\frac{d^2}{dt^2}\log\frac{q_m(2ar^2)}{q_m(ar^2)^2}
=8z^2\sum_{n\ge1}j_{\nu,n}^2
\left(\frac1{(j_{\nu,n}^2-2z^2)^2}
-\frac1{(j_{\nu,n}^2-z^2)^2}\right).
$$
For general $\ell$, \eqref{gauge:eq:state} gives $G_\ell(t)=\frac{16\ell a}{m(m+2\ell)}r^2 +O_{m,\ell}(a^2r^4)$ as $r\downarrow0$. The product formula gives the same coefficient as $16a\sum_{n\ge1}(j_{\nu,n}^{-2}-j_{\nu+\ell,n}^{-2})$, where $j_{\gamma,n}$ is the $n$th positive zero of $J_\gamma$. The equality follows from $\sum_{n\ge1}j_{\gamma,n}^{-2}=1/[4(\gamma+1)]$ \cite[Section 15.51]{WatsonBessel1944}.

\subsection{The characterisation}\label{gauge:sec:characterisation}

\begin{proof}[Proof of \cref{gauge:thm:characterisation}]
Suppose first that \textup{(b)} holds. By \cref{gauge:lem:comparison}\textup{(iii)}, $X_\ell\ge Y$, so \eqref{gauge:eq:defectrep} gives $G_\ell\ge0$ for every $\ell$. Hence each $\Psi_\ell(e^t)$ is log-convex. H\"older's inequality gives the three-point inequality for every finite mode sum, and monotone convergence in \eqref{gauge:eq:modesum} gives \textup{(a)}.

Conversely, \textup{(a)} applied to $f_\ell Y_{\ell j}$ gives $G_\ell\ge0$. Suppose $\dot Y(t_*)>0$ for some $t_*<\log\rho$, and set $W:=\int_{-\infty}^{t_*}w(s)\,ds$. If $\beta_\ell>W$, then \cref{gauge:lem:comparison}\textup{(i)} and a first-crossing argument give
$$
\beta_\ell-W\le X_\ell(s)\le\beta_\ell,
\qquad s\le t_*.
$$
Indeed, for $z_\ell:=\beta_\ell-X_\ell\ge0$, one has $\dot z_\ell=-(\beta_\ell+X_\ell)z_\ell+w$ until a possible first zero of $X_\ell$, and hence $z_\ell\le W<\beta_\ell$, excluding such a zero. Since $Y$ and $\dot Y$ are bounded on $(-\infty,t_*]$ and $\dot Y$ is continuous at $t_*$, \eqref{gauge:eq:defectrep} becomes
$$
G_\ell(t_*)=-\int_{-\infty}^{t_*}\varphi_\ell(s)
\frac{X_\ell(s)-Y(s)}{X_\ell(s)}\dot Y(s)\,ds,
\qquad
\varphi_\ell(s):=2X_\ell(s)
\exp\!\left(-2\int_s^{t_*}X_\ell(\tau)\,d\tau\right).
$$
Here $\varphi_\ell$ is a probability density and, for every $\delta>0$,
$$
\int_{-\infty}^{t_*-\delta}\varphi_\ell(s)\,ds
\le e^{-2(\beta_\ell-W)\delta},
\qquad
\sup_{s\le t_*}\left|\frac{X_\ell(s)-Y(s)}{X_\ell(s)}-1\right|
\longrightarrow0.
$$
Thus $G_\ell(t_*)\to-\dot Y(t_*)<0$, a contradiction. Hence $\dot Y\le0$, proving \textup{(b)}.

The hypothesis of \cref{gauge:thm:main} implies \textup{(b)} by \cref{gauge:lem:classmonotone}, but is not necessary. In dimension $3$, take $V(r)=1+2\sin^2(3r)$ on $[0,\infty)$. Since $V\le3$, Sturm comparison with $\mathbb E_0[e^{6\tau_r}]$ gives $R_*\ge\pi/\sqrt{12}>\pi/4$. Moreover,
$$
\frac{d}{dr}\bigl(r^2V(r)\bigr)
=2r\bigl(1+2\sin^2(3r)\bigr)+6r^2\sin(6r).
$$
This is positive for $0<r\le1/4$, but at $r=\pi/4$ it equals $\pi-3\pi^2/8<0$. On $r\le1/4$, \cref{gauge:lem:classmonotone} gives $\dot Y\le0$; on $1/4\le r<R_*$, $2w=4r^2V\ge1/4=\nu^2$, so $\dot Y=\nu^2-Y^2-2w\le0$. Thus \textup{(b)} holds on $(0,R_*)$ although $r^2V$ is not nondecreasing.
\end{proof}

By \cref{gauge:lem:gauge}, condition \textup{(b)} says that the reciprocal of $\mathbb E_0[\exp(2\int_0^{\tau_r}V(|B_s|)\,ds)]$ is log-concave in $\log r$ for $0<r<\rho$.

\subsection{Coefficients below two}\label{gauge:sec:threshold}

\begin{proof}[Proof of \cref{gauge:thm:threshold}]
By continuity, we may work at small $r$ where $V>0$ and the radial solutions below are positive. Let $Y_b:=N_{g_{bV}}+\nu$. Then
$$
\dot Y_b=\nu^2-Y_b^2-bw
$$
and
$$
\frac{d^2}{dt^2}\log\frac{f_0^2}{g_{bV}}
=2\dot X_0-\dot Y_b
=Y_b^2-2X_0^2+\nu^2+(b-2)w.
$$
By \cref{gauge:lem:gauge}, $f_0^{-1}$ and $g_{bV}^{-1}$ are the corresponding first-exit moments. Brownian scaling and $\mathbb E_0\tau_r=r^2/m$, equivalently one integration of the radial equations, give
$$
X_0=\nu-\frac{V(0)}{\nu+1}r^2+o(r^2),
\qquad
Y_b=\nu-\frac{bV(0)}{\nu+1}r^2+o(r^2),
\qquad
w=2V(0)r^2+o(r^2).
$$
Substitution yields
$$
\frac{d^2}{dt^2}\log\frac{f_0(r)^2}{g_{bV}(r)}
=\frac{2(b-2)V(0)}{\nu+1}r^2+o(r^2),
$$
which is negative for all sufficiently small $r$ when $b<2$.
\end{proof}

\subsection{A counterexample in dimensions \texorpdfstring{$m\ge3$}{m >= 3}}
\label{gauge:sec:counterexample}

\begin{proof}[Proof of \cref{gauge:prop:counterexample}]
Set
$$
I:=\frac4{m-2}\int_0^\infty sV(s)\,ds<1.
$$
Let $[0,\varrho)$ be the maximal interval on which $g>0$. There,
$$
g'(r)=-\frac4{r^{m-1}}\int_0^r s^{m-1}V(s)g(s)\,ds\le0,
$$
so $0<g\le1$. Fubini's theorem gives
$$
1-g(r)
\le\frac4{m-2}\int_0^\infty sV(s)\,ds=I,
$$
and hence $g\ge1-I>0$. Thus $\varrho=\infty$, $R_*=\infty$, and $1-I\le g\le1$.

Since $V\not\equiv0$, $g'(r_1)<0$. For $r>r_1$,
$$
g(r)=A+Br^{2-m},
\qquad
B=-\frac{g'(r_1)r_1^{m-1}}{m-2}>0,
\qquad
A=\lim_{r\to\infty}g(r)\ge1-I>0.
$$
Consequently,
$$
N_g(r)=\frac{-(m-2)Br^{2-m}}{A+Br^{2-m}}\in(-(m-2),0),
$$
so $Y=N_g+\nu\in(-\nu,\nu)$. Since $w=0$ for $r>r_1$,
$$
\dot Y=\nu^2-Y^2>0.
$$
Fix $t_*:=\log r_2$. The approximate-identity estimate in the proof of \cref{gauge:thm:characterisation} gives
\[
G_\ell(t_*)\longrightarrow-\dot Y(t_*)<0
\qquad(\ell\to\infty).
\]
Choose $\ell$ so that $G_\ell(t_*)<0$. By continuity, $\log\Psi_\ell(e^t)$ is strictly concave near $t_*$, and the solution $u=f_\ell Y_{\ell j}$ proves the claim.
\end{proof}

For $r>r_1$, the term $r^{2-m}$ is, up to a constant, the probability that Brownian motion starting at radius $r$ ever hits $B(0,r_1)$. Thus the counterexample is a transience effect. When $R_*<\infty$, the argument does not decide whether convexity fails before $R_*$.

\subsection{The energy representation and the role of radiality}

The energy identity below does not require radiality; exact log-convexity does.

\begin{proposition}[Energy representation and frequency]
\label{gauge:prop:energy}
Let $V\in C(B_R)$ be real, not necessarily radial, and let $u\not\equiv0$ solve $(\tfrac12\Delta+V)u=0$ in $B_R$. For every $r<R$ such that $-\tfrac12\Delta_{B_r}-2V$ has positive first Dirichlet eigenvalue, set
$$
H_u(r):=\mathbb E_0\!\left[
e^{2\int_0^{\tau_r}V(B_s)\,ds}u(B_{\tau_r})^2\right],
$$
and let $G^{2V}_{B_r}$ be the positive Dirichlet Green kernel of this operator. Then
\begin{equation}\label{gauge:eq:energy-representation}
H_u(r)=u(0)^2+\int_{B_r}G^{2V}_{B_r}(0,y)|\nabla u(y)|^2\,dy.
\end{equation}
On every interval of such radii,
$$
N_u(r):=\frac r2\frac{H_u'(r)}{H_u(r)}
=\frac{r}{2H_u(r)}\int_{B_r}
\partial_rG^{2V}_{B_r}(0,y)|\nabla u(y)|^2\,dy,
$$
where $y$ is fixed when the radius is differentiated. The weight is nonnegative by domain monotonicity. If $V=0$, then $\partial_rG^0_{B_r}(0,y)=2/|\partial B_r|$ throughout $B_r$, and the last formula is Almgren's frequency.
\end{proposition}

\begin{proof}
Put $A_t:=2\int_0^tV(B_s)\,ds$ and $\rho:=|\nabla u|^2$. Since $(\tfrac12\Delta+2V)u^2=\rho$, It\^o's formula and optional stopping at $\tau_r\wedge T$ give
$$
\mathbb E_0\!\left[e^{A_{\tau_r\wedge T}}
u(B_{\tau_r\wedge T})^2\right]
=u(0)^2+\mathbb E_0\!\int_0^{\tau_r\wedge T}e^{A_s}\rho(B_s)\,ds.
$$
The part of the left side on $\{T<\tau_r\}$ is the Dirichlet Feynman--Kac semigroup applied to $u^2$ and tends to zero because its first eigenvalue is positive. Monotone convergence for the remaining two terms and the occupation formula give \eqref{gauge:eq:energy-representation}. For an admissible radius, Hadamard's formula for $-\tfrac12\Delta-2V$ gives, for $y\in B_r$,
\[
\partial_rG^{2V}_{B_r}(0,y)
=\frac12\int_{\partial B_r}
 \partial_{n_z}G^{2V}_{B_r}(0,z)\,
 \partial_{n_z}G^{2V}_{B_r}(y,z)\,d\sigma(z).
\]
The two normal derivatives have the same sign, so this quantity is nonnegative. Differentiating \eqref{gauge:eq:energy-representation} gives
\[
H_u'(r)=\int_{B_r}
\partial_rG^{2V}_{B_r}(0,y)|\nabla u(y)|^2\,dy;
\]
the moving-boundary term is zero because $G^{2V}_{B_r}(0,\cdot)$ vanishes on $\partial B_r$. This proves the frequency identity.
\end{proof}

\begin{proposition}[A nonradial counterexample]
\label{gauge:prop:radiality-essential}
Let $m=2$, $R>1$, $n\ge1$, and $0\ne\varphi\in C_c^\infty((0,\infty))$ with $\varphi\ge0$, extended by zero to $\mathbb R$. For all sufficiently small $\varepsilon>0$, set
$$
V_\varepsilon(r,\theta)=\varepsilon\,
\frac{\varphi(-\log r)}{2r^2}(1+\cos 2n\theta)\ge0,
$$
with $V_\varepsilon(0)=0$, and let $u_\varepsilon$ solve $(\tfrac12\Delta+V_\varepsilon)u_\varepsilon=0$ in $B_R$ with boundary value $R^n\sin(n\theta)$. The solution is unique, $-\tfrac12\Delta_{B_R}-2V_\varepsilon$ has positive first eigenvalue, and
\begin{equation}\label{gauge:eq:first-variation}
\frac{d^2}{dt^2}\Big|_{t=0}\log H_{u_\varepsilon}(e^t)
=2n\varepsilon\int_0^\infty
(3-2ns)e^{-2ns}\varphi(s)\,ds+O(\varepsilon^2).
\end{equation}
Consequently, if $\operatorname{supp}\varphi\subset(3/(2n),\infty)$, then $\log H_{u_\varepsilon}(e^t)$ is strictly concave at $t=0$ for every sufficiently small $\varepsilon>0$, and the three-circle inequality with constant $1$ fails on some three nearby radii.
\end{proposition}

\begin{proof}
Write $V_\varepsilon=\varepsilon V_1$ and let $\mathcal R=(-\tfrac12\Delta_{B_R})^{-1}$ with Dirichlet boundary condition. If $w_\varepsilon=u_\varepsilon-u_0$, then
\[
w_\varepsilon
=\varepsilon\mathcal R\bigl(V_1(u_0+w_\varepsilon)\bigr).
\]
For small $\varepsilon$, $I-\varepsilon\mathcal RV_1$ is invertible on the zero-boundary subspace of $C^{2,\alpha}(\overline{B_R})$, and its Neumann series gives
$$
u_\varepsilon=u_0+\varepsilon u_1+O(\varepsilon^2)
\quad\text{in }C^{2,\alpha}(\overline{B_R}),\qquad
u_0=r^n\sin(n\theta),\qquad
\tfrac12\Delta u_1=-V_1u_0,\quad u_1|_{\partial B_R}=0.
$$
The min--max principle also gives
\[
\lambda_1\!\left(-\tfrac12\Delta_{B_R}-2\varepsilon V_1\right)
\ge\lambda_1\!\left(-\tfrac12\Delta_{B_R}\right)
-2\varepsilon\|V_1\|_\infty>0,
\]
and, since $V_1\ge0$,
\[
-\tfrac12\Delta_{B_R}-\varepsilon V_1
\ge-\tfrac12\Delta_{B_R}-2\varepsilon V_1>0.
\]
This proves the asserted positivity and uniqueness.

Brownian scaling gives, for $|t|$ small,
\[
H_{u_\varepsilon}(e^t)
=\mathbb E_0\!\left[
 \exp\!\left(2\varepsilon e^{2t}\int_0^{\tau_1}V_1(e^tB_s)\,ds\right)
 u_\varepsilon(e^tB_{\tau_1})^2\right].
\]
After decreasing $\varepsilon$, the exponential is dominated uniformly by $e^{c\tau_1}$ with $c<\lambda_1(-\tfrac12\Delta_{B_1})$. The exit time therefore has the required exponential moment. Smoothness of $V_1$, the $C^{2,\alpha}$ expansion above, and dominated differentiation justify the expansions below, including two derivatives in $t$. Since $H_{u_0}(r)=\tfrac12r^{2n}$, write $H_{u_\varepsilon}(r)=\tfrac12r^{2n}+\varepsilon h_1(r)+O(\varepsilon^2)$. Differentiating the Brownian exit formula at $\varepsilon=0$ gives
\begin{equation}\label{gauge:eq:h1}
h_1(r)=2\fint_{\partial B_r}u_0u_1\,d\sigma
+2\mathbb E_0\!\left[u_0(B_{\tau_r})^2
\int_0^{\tau_r}V_1(B_p)\,dp\right].
\end{equation}

Put $r=e^t$ and take $|t|$ small, so that $B_r$ contains the support of $V_1$. For $z=(\rho,\theta)$, Brownian exit from the disk gives
$$
\mathbb E_z[u_0(B_{\tau_r})^2]
=\tfrac12\bigl(r^{2n}-\rho^{2n}\cos 2n\theta\bigr),\qquad
G^0_{B_r}(0,z)=\frac1\pi\log\frac r\rho.
$$
The Markov property, the occupation formula, angular integration, and $\rho=e^{-s}$ give
\begin{equation}\label{gauge:eq:weight-variation}
2\mathbb E_0\!\left[u_0(B_{\tau_r})^2
\int_0^{\tau_r}V_1(B_p)\,dp\right]
=\frac12\int_0^\infty
\bigl(2e^{2nt}-e^{-2ns}\bigr)(t+s)\varphi(s)\,ds.
\end{equation}

Only the $n$th sine mode of $u_1$ contributes to the first term of \eqref{gauge:eq:h1}. Write this contribution as
\[
\rho^nq(-\log\rho)\sin(n\theta).
\]
Since $(1+\cos2n\theta)\sin(n\theta) =\tfrac12\sin(n\theta)+\tfrac12\sin(3n\theta)$,
$$
q''-2nq'=-\tfrac12\varphi,\qquad
q(-\log R)=0,\qquad q'(\infty)=0.
$$
Thus, with $J:=\int_0^\infty e^{-2ns}\varphi(s)\,ds$ and $|t|$ small,
$$
q(-t)=\frac{J}{4n}\bigl(e^{-2nt}-R^{-2n}\bigr),
$$
and hence
\begin{equation}\label{gauge:eq:solution-variation}
2\fint_{\partial B_r}u_0u_1\,d\sigma
=\frac{J}{4n}\bigl(1-e^{2nt}R^{-2n}\bigr).
\end{equation}

Set $F(t):=2e^{-2nt}h_1(e^t)$. Combining \eqref{gauge:eq:weight-variation} and \eqref{gauge:eq:solution-variation} gives
$$
F(t)=2\int_0^\infty(t+s)\varphi(s)\,ds
-e^{-2nt}\int_0^\infty e^{-2ns}(t+s)\varphi(s)\,ds
+\frac{J}{2n}\bigl(e^{-2nt}-R^{-2n}\bigr).
$$
The first term is affine in $t$, and the $R$-dependent term is constant, so
$$
F''(0)=2n\int_0^\infty
(3-2ns)e^{-2ns}\varphi(s)\,ds.
$$
Finally, $\log H_{u_\varepsilon}(e^t)=2nt+\mathrm{const} +\varepsilon F(t)+O(\varepsilon^2)$, proving \eqref{gauge:eq:first-variation}. The last assertion follows because $\varphi$ is nonzero and the integrand is strictly negative on its support.
\end{proof}

\section{Frequency, doubling, and Remez inequalities}
\label{sec:freq}

Fix $0<R\le r_{\mathrm{geo}}/2$ and a ball $B(x_0,R)$ with $\overline{B(x_0,R)}\subset X\setminus\partial X$. We study solutions of
\begin{equation}\label{eq:L_a_u}
\bigl(\tfrac12\Delta_g+a\bigr)u=0 \qquad\text{in }B(x_0,R),
\end{equation}
where $a\ge0$ is constant. For an eigenfunction $\Delta_g\varphi+\lambda\varphi=0$, one has $a=\lambda/2$. We write formulas for real $u$; for complex $u$, replace squares by absolute squares and, in every product, conjugate the first factor and take the real part.

\begin{definition}[Height and frequency]\label{def:ost-frequency}
Let $u$ solve \eqref{eq:L_a_u}. For $x\in B(x_0,R)$ and $0<r<\operatorname{dist}(x,\partial B(x_0,R))$, let $\tau_{x,r}=\inf\{t\ge0:B_t\notin B(x,r)\}$, with $B_0=x$, and define
\begin{equation}\label{eq:ost-height}
H_u(x,r):=\mathbb E_x\bigl[e^{2a\tau_{x,r}}u(B_{\tau_{x,r}})^2\bigr]
=\int_{\partial B(x,r)}u(\xi)^2P^{(2a)}_{x,r}(x,\xi)\,d\sigma_g(\xi).
\end{equation}
This is finite when $2a<\mu_1(B(x,r))$ by \cref{lem:survival-gap}\textup{(2)}; $P^{(2a)}_{x,r}$ is the Poisson kernel with exponential weight $e^{2a\tau_{x,r}}$. Whenever $H_u(x,r)>0$ and is differentiable in $r$, define
\begin{equation}\label{eq:ost-frequency}
N_u(x,r):=\frac r2\,\partial_r\log H_u(x,r).
\end{equation}
The exponent $2a$ is the second-moment weight of the martingale $e^{at}u(B_t)$ stopped on leaving the ball.
\end{definition}

In Euclidean balls, \cref{gauge:thm:main,gauge:thm:threshold} show that this height is log-convex and that the coefficient $2a$ is the smallest with this property.

\begin{lemma}[Height monotonicity]\label{lem:height-mono}
Let $u$ solve \eqref{eq:L_a_u}, let $x\in B(x_0,R)$, and suppose $0<r<\operatorname{dist}(x,\partial B(x_0,R))$ and $2a<\mu_1(B(x,r))$. Then $s\mapsto H_u(x,s)$ is nondecreasing on $(0,r]$.
\end{lemma}

\begin{proof}
Since $(\tfrac12\Delta_g+2a)(u^2)=|\nabla u|^2$, It\^o's formula shows that $e^{2at}u(B_t)^2$ is a local submartingale. Stopped at $\tau_{x,r}$, it is dominated by $\|u\|_{L^\infty(\overline{B(x,r)})}^2e^{2a\tau_{x,r}}$, which is integrable by \cref{lem:survival-gap}. Hence, for $0<s\le s'\le r$,
\[
H_u(x,s')-H_u(x,s)
=\mathbb E_x\int_{\tau_{x,s}}^{\tau_{x,s'}}
e^{2at}|\nabla u(B_t)|^2\,dt\ge0.
\]
\end{proof}

\paragraph{Boundary integrals.}
Fix $x$ and $0<r<\operatorname{dist}(x,\partial B(x_0,R))$ with $2a<\mu_1(B(x,r))$. On $S_r=\partial B(x,r)$, put $\rho=d(x,\cdot)$, $\mathbf n=\nabla\rho$, and let $\nabla_T$ be the tangential gradient. With $P_r(\xi)=P^{(2a)}_{x,r}(x,\xi)$, set
\[
\begin{aligned}
H&=\int_{S_r}u^2P_r\,d\sigma_g,&
D&=\int_{S_r}u\,\partial_{\mathbf n}u\,P_r\,d\sigma_g,\\
I&=\int_{S_r}(\partial_{\mathbf n}u)^2P_r\,d\sigma_g,&
T&=\int_{S_r}|\nabla_Tu|^2P_r\,d\sigma_g.
\end{aligned}
\]
Thus $H=H_u(x,r)$.

\begin{lemma}[Weighted Rellich identity]
\label{lem:frozen-angular-rellich}
Let $m\ge2$, assume bounded geometry of order $2$ on an open set $U$, and let $B_r=B(x,r)\Subset U\subset X\setminus\partial X$, $0<r\le r_0$. Let $\kappa$ be the curvature bound in \cref{lem:poisson-comp-small} with $\bar r=r$. Let $u\in C^2(\overline{B_r})$ solve $(\tfrac12\Delta_g+a)u=0$ in $B_r$, where $a\ge0$ and $ar^2\le c_0$. Define $P_r,H,D,I,T$ as above. Set $Z=\rho\mathbf n$ and extend $P_r$ constantly on radial geodesics:
\[
\Phi(\exp_x(s\omega)):=P_r(\exp_x(r\omega)),
\qquad 0<s\le r,\quad
\omega\in S_xX.
\]
Then $\Phi|_{S_r}=P_r$ and $Z\Phi=0$ on $B_r\setminus\{x\}$. Put
\[
U_\Phi:=\int_{B_r}u^2\Phi\,d\operatorname{vol}_g,
\qquad
\mathsf S[u]:=\nabla u\otimes\nabla u-\tfrac12|\nabla u|^2g+au^2g,
\qquad A_Z:=\nabla Z-g.
\]
Then
\begin{equation}\label{eq:frozen-angular-rellich}
T=\frac{m-2}{r}D+I+2aH-\frac{4a}{r}U_\Phi+E_{\mathrm{fr}}(r),
\end{equation}
where
\begin{equation}\label{eq:frozen-angular-error-decomp}
E_{\mathrm{fr}}(r)=-\frac{m-2}{r}Q_\Phi
-\frac2rJ_{\mathrm{geo}}-\frac2rJ_{\mathrm{ang}},
\end{equation}
with
\begin{align}
Q_\Phi&:=\int_{B_r}u\langle\nabla\Phi,\nabla u\rangle
\,d\operatorname{vol}_g, \label{eq:Qphi-def}\\
J_{\mathrm{geo}}&:=\int_{B_r}\Phi\,\mathsf S[u]:A_Z
\,d\operatorname{vol}_g, \label{eq:Jgeo-def}\\
J_{\mathrm{ang}}&:=\int_{B_r}\rho\,\partial_{\mathbf n}u
\langle\nabla_Tu,\nabla_T\Phi\rangle\,d\operatorname{vol}_g.
\label{eq:Jang-def}
\end{align}
In Euclidean space, $\Phi$ is constant, so all three error terms vanish. Moreover, the bounds in \eqref{eq:frozen-weight-calculus} give
\begin{equation}\label{eq:frozen-angular-error-bound}
|E_{\mathrm{fr}}(r)|
\le C(\kappa+a)rH^{1/2}I^{1/2}
+C(\kappa+a)H+C(\kappa+a)ar^2H.
\end{equation}
Here $C$ depends only on $m$ and the bounded-geometry data.
\end{lemma}

\begin{proof}
Since $\operatorname{div}\mathsf S[u]=0$,
$$
\operatorname{div}\bigl(\Phi\mathsf S[u](Z,\cdot)\bigr)
=\Phi\mathsf S[u]:\nabla Z+\mathsf S[u](Z,\nabla\Phi).
$$
Excision of $x$ and passage to the limit are justified by $|\nabla_T\Phi|\lesssim\Phi/\rho$; the inner boundary term vanishes. Since $Z\Phi=0$,
$$
\tfrac r2(I-T)+arH
=\int_{B_r}\Phi\mathsf S[u]:\nabla Z\,d\operatorname{vol}_g
+J_{\mathrm{ang}}.
$$
Put $K_\Phi=\int_{B_r}\Phi|\nabla u|^2\,d\operatorname{vol}_g$. From $\nabla Z=g+A_Z$, the trace of $\mathsf S[u]$, and integration by parts against $u\Phi$,
$$
\int_{B_r}\Phi\mathsf S[u]:\nabla Z\,d\operatorname{vol}_g
 =-\tfrac{m-2}{2}K_\Phi
+maU_\Phi+J_{\mathrm{geo}}.
\qquad
K_\Phi=D+2aU_\Phi-Q_\Phi.
$$
Substitution gives \eqref{eq:frozen-angular-rellich}--\eqref{eq:frozen-angular-error-decomp}.

The bounds $|A_Z|\le C\kappa\rho^2$ and $\rho|\nabla_T\Phi|\le C(\kappa+a)r^2\Phi$ give
$$
|J_{\mathrm{geo}}|\le C\kappa r^2(K_\Phi+aU_\Phi),
\qquad
|J_{\mathrm{ang}}|\le C(\kappa+a)r^2K_\Phi.
$$
Surface comparison, coarea, \cref{lem:height-mono}, and the unweighted energy identity yield
\begin{equation}\label{eq:vol-height-bound}
U_\Phi\le CrH,
\qquad
K_\Phi\le C(H^{1/2}I^{1/2}+arH).
\end{equation}
Hence
$$
\frac2r\bigl(|J_{\mathrm{geo}}|+|J_{\mathrm{ang}}|\bigr)
\le C(\kappa+a)rH^{1/2}I^{1/2}
+C(\kappa+a)ar^2H.
$$

If $m=2$, $Q_\Phi$ has zero coefficient in \eqref{eq:frozen-angular-error-decomp}. If $m\ge3$, integration by parts on each $S_s$ and \eqref{eq:frozen-weight-calculus} give
$$
|Q_\Phi|
\le C(\kappa+a)r^{3-m}H\int_0^rs^{m-3}\,ds
\le C(\kappa+a)rH.
$$
These estimates prove \eqref{eq:frozen-angular-error-bound}.
\end{proof}

\begin{theorem}[Almost-monotonicity of the frequency]
\label{thm:ost-freq-almostmono}
Assume bounded geometry of order $2$ on a domain $\Omega\subset X\setminus\partial X$. Let $u\not\equiv0$ solve $(\tfrac12\Delta_g+a)u=0$ in $\Omega$, where $a\ge0$, and let $B(x,r_*)\Subset\Omega$, $r_*\le r_0$, and $ar_*^2\le c_0$. Let $\kappa$ be the curvature bound of \cref{lem:poisson-comp-small} on $B(x,r_*)$. For $0<r<r_*$, write $H(r)=H_u(x,r)$, let $D(r)$ be as above, and set
\[
\mathcal N(r):=\frac{rD(r)}{H(r)}.
\]
Then $\mathcal N$ is locally absolutely continuous and
\begin{equation}\label{eq:ost-freq-almostmono}
\mathcal N'(r)\ge
-C_{\mathrm{am}}(\kappa+a)r\bigl(1+\mathcal N(r)\bigr)
\qquad\text{for a.e. }r\in(0,r_*),
\end{equation}
where $C_{\mathrm{am}}$ depends only on $m$ and the bounded-geometry data. Moreover,
\begin{equation}\label{eq:Nu-compare-Naux}
N_u(x,r)=\mathcal N(r)+O\bigl((\kappa+a)r^2\bigr)
\end{equation}
for almost every $r$, and, after decreasing $r_0$ and $c_0$ if necessary, $1+N_u(x,r)\asymp1+\mathcal N(r)$ for almost every $r$.
\end{theorem}

\begin{proof}
Put $B_r=B(x,r)$ and $\mathcal R_r=\partial_rP_r+(\Delta_g\rho)P_r$. The kernel bounds make $H$ and $D$ locally absolutely continuous. Moreover $H(r)>0$: otherwise the stopping identity for $u$ forces $u=0$ in $B_r$, contrary to unique continuation; see, for example, \cite{Aronszajn1957,GarofaloLin1986}. Thus $\mathcal N$ is locally absolutely continuous.

First variation gives
\begin{equation}\label{eq:EH-bound}
H'=2D+E_H,
\qquad
E_H=\int_{S_r}u^2\mathcal R_r\,d\sigma_g,
\qquad
|E_H|\le C(\kappa+a)rH.
\end{equation}
Thus $N_u=\tfrac r2H'/H=\mathcal N+O((\kappa+a)r^2)$. Writing $\mathbf n=\nabla\rho$, the polar decomposition $\Delta_gu=\partial_{\mathbf n\mathbf n}u+(\Delta_g\rho) \partial_{\mathbf n}u+\Delta_{S_r}u$ and integration by parts on $S_r$ give
\begin{equation}\label{eq:Dprime-main}
D'=I+T-2aH-\frac{m-1}{r}D+E_D,
\end{equation}
where
\[
\begin{aligned}
E_D={}&\int_{S_r}u\partial_{\mathbf n}u
\left(\frac{\mathcal R_r}{P_r}
-\Delta_g\rho+\frac{m-1}{r}\right)P_r\,d\sigma_g\\
&+\int_{S_r}u\langle\nabla_Tu,\nabla_T\log P_r\rangle
P_r\,d\sigma_g.
\end{aligned}
\]
The bounds in \cref{lem:poisson-comp-small}\textup{(ii$'$)} and $|\Delta_g\rho-(m-1)/r|\le C\kappa r$ imply
\begin{equation}\label{eq:ED-bound-root}
|E_D|\le C(\kappa+a)rH^{1/2}(I^{1/2}+T^{1/2}).
\end{equation}

Differentiate $\mathcal N=rD/H$ and use \cref{lem:frozen-angular-rellich}, with $\Phi$ extended constantly along radial geodesics and $U_\Phi=\int_{B_r}u^2\Phi\,d\operatorname{vol}_g$. The terms involving $2aH$ and $(m-2)D/r$ cancel, leaving
\begin{equation}\label{eq:frequency-assembly}
\mathcal N'=\mathcal V-\frac{4a}{H}U_\Phi
+r\frac{E_D}{H}+r\frac{E_{\mathrm{fr}}}{H}
-\mathcal N\frac{E_H}{H},
\qquad
\mathcal V:=\frac{2r}{H}\left(I-\frac{D^2}{H}\right)\ge0.
\end{equation}
By \eqref{eq:vol-height-bound}, $I/H=\mathcal N^2/r^2+\mathcal V/(2r)$, and Young's inequality,
$$
\frac{4aU_\Phi}{H}\le Car,
\qquad
r\frac{|E_{\mathrm{fr}}|}{H}
\le\tfrac18\mathcal V
+C(\kappa+a)r(1+|\mathcal N|).
$$

The Rellich identity also gives
$$
\frac TH\le C\left(
\frac{1+\mathcal N^2}{r^2}+\frac{\mathcal V}{2r}
+a+(\kappa+a)^2r^2\right).
$$
Substitution in \eqref{eq:ED-bound-root} yields
$$
r\frac{|E_D|}{H}
\le\tfrac18\mathcal V+C(\kappa+a)r(1+|\mathcal N|).
$$
Together with \eqref{eq:EH-bound} and \eqref{eq:frequency-assembly}, this gives
$$
\mathcal N'\ge\tfrac34\mathcal V
-C(\kappa+a)r(1+|\mathcal N|).
$$
Finally, \cref{lem:height-mono} gives $N_u\ge0$ almost everywhere, so \eqref{eq:Nu-compare-Naux} and the smallness of $(\kappa+a)r^2$ give $\mathcal N\ge-1/2$. Hence $1+|\mathcal N|\le3(1+\mathcal N)$, proving \eqref{eq:ost-freq-almostmono} and the last assertion.
\end{proof}

\paragraph{Relation with classical frequencies.}
The leading Rellich cancellation does not require the Poisson weight. The latter makes $H(r)=\mathbb E_x[e^{2a\tau_{x,r}}u(B_{\tau_{x,r}})^2]$ the second moment of the stopped martingale, so a frequency bound gives doubling directly. Classical Almgren and Garofalo--Lin frequencies use geometric surface energies rather than this exponential weight; see \cite{Alm79,GarofaloLin1986}. For the Helmholtz equation, at fixed radius in the standard model geometries, any sup-norm three-ball constant is necessarily exponential in the wave number \cite[Theorems~2.2 and~2.9]{BergeMalinnikova2021}.

The three conclusions below have different inputs. Part~\textup{(i)}
combines integration of the defining logarithmic derivative with the
standard norm comparisons; part~\textup{(ii)} uses
\cref{thm:ost-freq-almostmono}; and part~\textup{(iii)} combines the
doubling estimate in part~\textup{(i)} with a ground-state transform and
a Remez theorem.

\begin{proposition}
[Doubling, a weighted three-radius inequality, and a Remez consequence]
\label{prop:remez-from-ost}
Assume $B(x_0,4R)\Subset X\setminus\partial X$, with bounded geometry of order $2$ there, and let $u\not\equiv0$ solve $(\tfrac12\Delta_g+a)u=0$ there, where $4R\le r_0$ and $a(2R)^2\le c_0$. Let $\kappa$ be the curvature bound of \cref{lem:poisson-comp-small} on $B(x_0,4R)$. Then:
\begin{enumerate}
\item[\textup{(i)}] For $x\in B(x_0,R/2)$ and $0<r\le R/4$, set
\[
N_*(x;r)=\operatorname*{ess\,sup}_{s\in[r,2r]}N_u(x,s),
\qquad
\widetilde N_*(x;r)=
\operatorname*{ess\,sup}_{s\in[r/2,2r]}N_u(x,s).
\]
Then
\begin{equation}\label{eq:ost-H-doubling}
H_u(x,2r)\le2^{2N_*(x;r)}H_u(x,r).
\end{equation}
Moreover,
\[
\fint_{\partial B(x,2r)}|u|^2\,d\sigma_g
\le C2^{2N_*(x;r)}
\fint_{\partial B(x,r)}|u|^2\,d\sigma_g,
\]
and
\[
\fint_{B(x,2r)}|u|^2\,d\operatorname{vol}_g
\le C4^{2\widetilde N_*(x;r)}
\fint_{B(x,r)}|u|^2\,d\operatorname{vol}_g.
\]

\item[\textup{(ii)}] For $0<r<R$, set
\[
C_0=C_{\mathrm{am}}\sup_{0<s\le2R}(\kappa+a)s^2
\]
and, if $C_0>0$,
\[
\widetilde\alpha=
\frac{R^{-C_0}-(2R)^{-C_0}}
{r^{-C_0}-(2R)^{-C_0}},
\]
while $\widetilde\alpha=\log2/\log(2R/r)$ if $C_0=0$. Then
\begin{equation}\label{eq:3ball-H}
H_u(x_0,R)\le
C H_u(x_0,r)^{\widetilde\alpha}
H_u(x_0,2R)^{1-\widetilde\alpha}.
\end{equation}

\item[\textup{(iii)}] The following measurable-set estimate is a
consequence of the Remez theorem of Logunov--Malinnikova
\cite[Lemma~4.2 and Remark~4.3]{LogunovMalinnikova2018}. Let
\[
N_\sharp=
\sup_{x\in B(x_0,R)}
\operatorname*{ess\,sup}_{\substack{0<r\le R\\
B(x,2r)\subset B(x_0,2R)}}N_u(x,r).
\]
If $N_\sharp<\infty$, $0<\gamma\le1$, and the measurable set $E\subset B(x_0,R/2)$ satisfies
\[
\operatorname{vol}_g(E)\ge
\gamma\operatorname{vol}_g(B(x_0,R/2)),
\]
then
\begin{equation}\label{eq:remez-from-ost}
\int_{B(x_0,R/2)}u^2\,d\operatorname{vol}_g
\le\left(\frac C\gamma\right)^{C(1+N_\sharp)}
\int_Eu^2\,d\operatorname{vol}_g.
\end{equation}
\end{enumerate}
Here $C$ depends only on $m$ and the bounded-geometry data.
\end{proposition}

\begin{proof}
For \textup{(i)},
\[
\log\frac{H_u(x,2r)}{H_u(x,r)}
=\int_r^{2r}\frac{2N_u(x,s)}s\,ds
\le2N_*(x;r)\log2.
\]
The surface statement follows from \cref{eq:H-comp-surfmean}. Moreover, surface comparison, coarea, and \cref{lem:height-mono} give
\[
\int_{B(x,2r)}u^2\,d\operatorname{vol}_g\le Cr^mH_u(x,2r),
\qquad
\int_{B(x,r)}u^2\,d\operatorname{vol}_g\ge cr^mH_u(x,r/2).
\]
Integrating the frequency from $r/2$ to $2r$ proves the volume statement.

For \textup{(ii)}, put $t_1=\log r$, $t_2=\log R$, $t_3=\log(2R)$, $L(t)=\log H_u(x_0,e^t)$, and
\[
q(t)=N_u(x_0,e^t)-\mathcal N(e^t),
\qquad |q(t)|\le C(\kappa+a)e^{2t}.
\]
Since $1+\mathcal N\ge1/2$, \eqref{eq:ost-freq-almostmono} implies that $F(t)=e^{C_0t}(1+\mathcal N(e^t))$ is nondecreasing. Define
\[
G(t)=L(t)-L(t_1)+2(t-t_1)-2\int_{t_1}^tq(s)\,ds.
\]
Then $G'=2(1+\mathcal N)$, so for $C_0>0$ the function $G$ is convex in $z=e^{-C_0t}$; for $C_0=0$ it is convex in $t$. Evaluation at $t_1,t_2,t_3$ gives the exponent $\widetilde\alpha$. The linear correction is at most $2\log2$, while
\[
\int_{t_1}^{t_3}|q(s)|\,ds\le C(\kappa+a)R^2\le C,
\]
which proves \eqref{eq:3ball-H}.

For \textup{(iii)}, let $D=B(x_0,2R)$ and
\[
h(y)=\mathbb E_y[e^{a\tau_D}].
\]
Then $h\asymp1$ with uniform rescaled $C^1$ bounds on $B(x_0,3R/2)$, and $v=u/h$ satisfies
\[
\operatorname{div}_g(h^2\nabla_gv)=0.
\]
Choose a cover $\mathcal Q$ of $B(x_0,R/2)$ by a uniformly bounded number of congruent coordinate cubes with fixed overlap, centred in that ball and fine enough that $K_mQ\Subset B(x_0,3R/2)$ for a fixed sufficiently large $K_m>20m$. Choose it so that any two members are joined by a uniformly bounded chain $Q_0,\ldots,Q_J$ in $\mathcal Q$ with $Q_{j+1}\subset2Q_j$. In these coordinates $v$ satisfies $\operatorname{div}(A\nabla v)=0$ with uniformly elliptic Lipschitz $A$ after rescaling. For every $Q\in\mathcal Q$ there are concentric metric balls of comparable radii such that $B_-\subset Q\subset2Q\subset\frac12B_+$, $2B_+\subset B(x_0,2R)$, and $\operatorname{rad}B_+\le R/2$. Since $h\asymp1$, local boundedness and a fixed number of applications of \textup{(i)} on these admissible intermediate balls give
\[
\frac{\sup_{2Q}|v|^2}{\sup_Q|v|^2}
\le C\frac{\fint_{B_+}|u|^2\,d\operatorname{vol}_g}
{\fint_{B_-}|u|^2\,d\operatorname{vol}_g}
\le2^{C(1+N_\sharp)}.
\]
Thus the doubling index from $Q$ to $2Q$ is at most $C(1+N_\sharp)$.

Set
\[
E'=\left\{y\in E:|v(y)|^2\le
\frac2{\operatorname{vol}_g(E)}\int_E|v|^2\,d\operatorname{vol}_g\right\}.
\]
Then $\operatorname{vol}_g(E')\ge\operatorname{vol}_g(E)/2$. Coordinate-volume comparison and pigeonhole give some $Q\in\mathcal Q$ such that, with $E_Q:=E'\cap Q$, $|E_Q|_{\rm e}\ge c\gamma|Q|_{\rm e}$.

If $v$ is real, the Remez estimate of \cite[Lemma~4.2 and Remark~4.3]{LogunovMalinnikova2018} gives
\[
\sup_Q|v|
\le\left(\frac C\gamma\right)^{C(1+N_\sharp)}\sup_{E_Q}|v|.
\]

For complex $v$, choose $\theta$ so that, for $w=\operatorname{Re}(e^{-i\theta}v)$, $\sup_Q|w|=\sup_Q|v|$. Since $\sup_{2Q}|w|\le\sup_{2Q}|v|$ and $|w|\le|v|$, the one-scale version of the same estimate, applied to $w$, gives the displayed bound.

In either case, propagation along the chains chosen above gives
\[
\sup_{B(x_0,R/2)}|v|
\le\left(\frac C\gamma\right)^{C(1+N_\sharp)}\sup_{E'}|v|.
\]
Integrating and using $h\asymp1$ proves \eqref{eq:remez-from-ost}.
\end{proof}

\section{Superlevel components and nodal rigidity}
\label{sec:rigidity}

\subsection{Unconditional superlevel components}
\label{subsec:superlevel-components}

Assume in this subsection that $X$ is closed. Let $\varphi$ be real-valued and satisfy $\Delta_g\varphi+\lambda\varphi=0$, with $\lambda>0$. For $\ell\ge0$, let $C$ be a component of $\{\varphi>\ell\}$, choose $x_C\in C$ with
\[
M_C:=\varphi(x_C)=\max_C\varphi,
\]
set $\vartheta_C:=\ell/M_C\in[0,1)$, and let $\tau_C$ be the first exit time from $C$.

\begin{theorem}[Brownian exit and volume of a superlevel component]
\label{thm:superlevel-stopping}
For every $x\in C$ and $t>0$,
\begin{equation}\label{eq:superlevel-exit}
\mathbb P_x(\tau_C\le t)
\le
\min\left\{1,
\frac{1-e^{-\lambda t/2}\varphi(x)/M_C}{1-\vartheta_C}\right\}.
\end{equation}
Moreover,
\begin{equation}\label{eq:superlevel-occupation}
\mathbb E_{x_C}\int_0^{\tau_C}
\frac{\varphi(B_s)}{M_C}\,ds
=\frac{2(1-\vartheta_C)}{\lambda}.
\end{equation}
If $\ell>0$, then
\begin{equation}\label{eq:superlevel-exp-moment}
\mathbb E_{x_C}e^{(\lambda/2)\tau_C}=\frac{M_C}{\ell}
\end{equation}
and
\begin{equation}\label{eq:superlevel-mean-exit}
\frac{2(1-\vartheta_C)}{\lambda}
\le\mathbb E_{x_C}\tau_C
\le\frac{2}{\lambda}\log\frac{M_C}{\ell}.
\end{equation}

There are $r_0,c,C>0$, depending only on bounded geometry, such that for $0<r\le r_0$,
\begin{equation}\label{eq:superlevel-local-volume}
\frac{\operatorname{vol}_g(B(x_C,r)\setminus C)}
{\operatorname{vol}_g(B(x_C,r))}
\le
C\min\left\{1,\frac{\lambda r^2}{1-\vartheta_C}\right\}.
\end{equation}
Consequently, for every $\eta\in(0,1)$, if
\[
r_{C,\eta}=c\min\left\{r_0,
\sqrt{\eta(1-\vartheta_C)}\,\lambda^{-1/2}\right\},
\]
then
\begin{equation}\label{eq:superlevel-almost-full}
\operatorname{vol}_g(C\cap B(x_C,r_{C,\eta}))
\ge(1-\eta)\operatorname{vol}_g(B(x_C,r_{C,\eta})).
\end{equation}
Finally,
\begin{equation}\label{eq:superlevel-global-volume}
\operatorname{vol}_g(C)
\ge c_g(1-\vartheta_C)^{m/2}\lambda^{-m/2},
\end{equation}
where $c_g>0$ depends only on $(X,g)$. The same conclusions hold for components of $\{\varphi<-\ell\}$ after replacing $\varphi$ by $-\varphi$.
\end{theorem}

\begin{proof}
Put $a=\lambda/2$ and $\tau=\tau_C$. Since $\partial C\subset\{\varphi=\ell\}$, optional stopping at $t\wedge\tau$ yields
$$
M_C=\mathbb E_{x_C}\!\left[
e^{a(t\wedge\tau)}\varphi(B_{t\wedge\tau})\right].
$$
For $\ell>0$, comparison with $\ell e^{a(t\wedge\tau)}$, followed by monotone convergence and optional stopping at $\tau$, proves \eqref{eq:superlevel-exp-moment}. Starting instead from any $x\in C$, splitting according to $\{\tau\le t\}$, and using $\varphi\le M_C$ gives
$$
e^{-at}\frac{\varphi(x)}{M_C}
\le\mathbb P_x(\tau>t)+\vartheta_C\mathbb P_x(\tau\le t),
$$
hence \eqref{eq:superlevel-exit}. The unweighted stopped formula is
$$
\mathbb E_{x_C}\varphi(B_{t\wedge\tau})
=M_C-a\mathbb E_{x_C}\int_0^{t\wedge\tau}\varphi(B_s)\,ds.
$$
Letting $t\to\infty$ proves \eqref{eq:superlevel-occupation} when $\ell>0$; for $\ell=0$, apply it to the nested components of $\{\varphi>\ell'\}$ and let $\ell'\downarrow0$. The mean-exit bounds follow from $\varphi/M_C\le1$ and Jensen's inequality.

For $0<r\le r_0$, the local heat-kernel lower bound gives
$$
cr^{-m}\operatorname{vol}_g(B(x_C,r)\setminus C)
\le\mathbb P_{x_C}(\tau_C\le r^2).
$$
Taking $x=x_C$ in \eqref{eq:superlevel-exit} proves \eqref{eq:superlevel-local-volume} and \eqref{eq:superlevel-almost-full}.

Finally, taking $x=x_C$ and $t=(1-\vartheta_C)/\lambda$ in the exit estimate gives $\mathbb P_{x_C}(\tau_C>t)\ge1/2$. Since the positive spectrum of the fixed manifold is bounded away from zero, $t$ is uniformly bounded. The ambient heat-kernel bound therefore gives
$$
\frac12
\le\int_Cp_C(t,x_C,y)\,d\operatorname{vol}_g(y)
\le C_gt^{-m/2}\operatorname{vol}_g(C),
$$
which proves \eqref{eq:superlevel-global-volume}.
\end{proof}

At level zero, \eqref{eq:superlevel-occupation} gives $\mathbb E_{x_C}\tau_C\ge2/\lambda$, the mean-exit form of the landscape bound \cite[Section~6.2.1]{GrebenkovNguyen2013}; the same stopping argument gives the Euclidean slab corollary below without assuming that $M_C$ is attained.

\begin{corollary}[Superlevel components in a slab]
\label{cor:superlevel-slab}
Let $\lambda>0$ and let $\varphi\in C^2(\mathbb R^m)$ satisfy $\Delta\varphi+\lambda\varphi=0$. Let $C$ be a component of $\{\varphi>\ell\}$ contained in $\{x:|x\cdot e-b|<a\}$ for some unit vector $e$. If $\ell>0$, $M_C:=\sup_C\varphi<\infty$, and $a\sqrt\lambda<\pi/2$, then
$$
\frac{\ell}{M_C}\ge\cos(a\sqrt\lambda),
\qquad
1-\frac{\ell}{M_C}\le\frac{\lambda a^2}{2}.
$$
At level zero, any nodal domain on which $|\varphi|$ is bounded and which is contained in a slab of width $2a$ satisfies $2a\ge\pi\lambda^{-1/2}$. All constants are sharp for the one-dimensional plane wave.
\end{corollary}

\begin{proof}
Let $S=\{x:|x\cdot e-b|<a\}$ and let $\tau_C,\tau_S$ be the corresponding exit times. If $\ell>0$, then $0<\varphi\le M_C$ in $C$ and $e^{(\lambda/2)\tau_S}$ is integrable, so optional stopping gives, for every $x\in C$,
$$
\frac{\varphi(x)}{\ell}
=\mathbb E_xe^{(\lambda/2)\tau_C}
\le\mathbb E_xe^{(\lambda/2)\tau_S}
=\frac{\cos(\sqrt\lambda(x\cdot e-b))}{\cos(a\sqrt\lambda)}
\le\frac1{\cos(a\sqrt\lambda)}.
$$
Taking the supremum over $x$ and using $1-\cos z\le z^2/2$ proves the first claim. At level zero, if $a\sqrt\lambda<\pi/2$, boundedness on the nodal domain again permits optional stopping at $\tau_C\le\tau_S$, which would give $\varphi=0$ on $C$; apply this to $-\varphi$ in a negative domain.
\end{proof}

Since $\{\varphi\le\ell\}\cap B(x_C,r)\subset B(x_C,r)\setminus C$, the bound \eqref{eq:superlevel-local-volume} holds unconditionally for the relative volume of the sublevel set $\{\varphi\le\ell\}$ in $B(x_C,r)$.

\paragraph{Component counts and nodal extrema.}
For $\ell>0$, let $\mathcal C_\ell$ be the components of $\{\varphi>\ell\}$ and set $M_C=\max_C\varphi$. Then
\begin{equation}\label{eq:weighted-component-count}
\sum_{C\in\mathcal C_\ell}
\left(1-\frac{\ell}{M_C}\right)^{m/2}
\le C_g\lambda^{m/2}\operatorname{vol}_g\{\varphi>\ell\}.
\end{equation}
Hence, for every $0<\eta<1$, the number of components satisfying $M_C\ge\ell/(1-\eta)$ is at most $C_g\eta^{-m/2}\lambda^{m/2}\operatorname{vol}_g\{\varphi>\ell\}$.

If $\{\Omega_i\}$ are the nodal domains and $M_i=\|\varphi\|_{L^\infty(\Omega_i)}$, then for every $p>0$,
\begin{equation}\label{eq:nodal-extrema-moments}
\sum_iM_i^p
\le C_{g,p}\lambda^{m/2}\|\varphi\|_{L^p(X)}^p.
\end{equation}
Indeed, the components are disjoint and fill $\{\varphi>\ell\}$, so summing \eqref{eq:superlevel-global-volume} proves \eqref{eq:weighted-component-count}. For each nodal domain, apply \eqref{eq:superlevel-global-volume} to the component containing a point where $|\varphi|=M_i$, at level $\vartheta M_i$. Since $|\varphi|\ge\vartheta M_i$ there,
\[
\|\varphi\|_{L^p(X)}^p
\ge c_g\vartheta^p(1-\vartheta)^{m/2}\lambda^{-m/2}
\sum_iM_i^p.
\]
Taking $\vartheta=2p/(2p+m)$ proves \eqref{eq:nodal-extrema-moments}.

At level zero, \eqref{eq:superlevel-global-volume} is the Faber--Krahn estimate used in the Courant--Pleijel and Krahn--Szeg\H{o} arguments. Positive levels have no analogue of Courant's nodal-domain count, and \eqref{eq:weighted-component-count} supplies a weighted substitute. In the nodal-domain case, \eqref{eq:superlevel-exit} gives the estimate posed in \cite[Question~3.8]{MukherjeeSaha2026HeatProfile}; at level zero it reduces to \cite[Proposition~3.7]{MukherjeeSaha2026HeatProfile} after accounting for their heat generator $\Delta$ rather than our Brownian generator $\tfrac12\Delta$. Finally, \eqref{eq:nodal-extrema-moments} is the arbitrary-dimensional $L^p$ estimate of \cite[Theorem~1.2.4 and \textup{(2.3.2)}]{Poliquin2017}, extending the surface estimates in \cite[Theorems~1.3 and~1.6]{PolterovichSodin2007}.

\begin{corollary}[Euclidean superlevel-volume bound]
\label{cor:euclidean-superlevel}
Let $\Omega\subset\mathbb R^m$ be open and let $\varphi\in C^2(\Omega)\cap C(\overline\Omega)$ satisfy $\Delta\varphi+\lambda\varphi=0$ in $\Omega$ and $\varphi=0$ on $\partial\Omega$; the case $\Omega=\mathbb R^m$ is allowed. Let $C$ be a bounded component as above. For $0<\vartheta_C<1$, let $\beta_m(\vartheta_C)$ be determined by
\[
q_m(\beta_m(\vartheta_C))=\vartheta_C^{-1},
\qquad
0<\beta_m(\vartheta_C)<\tfrac12j_{\nu,1}^2.
\]
Then
\begin{equation}\label{eq:sharp-euclidean-superlevel}
|C|\ge
\omega_m\left(\frac{2\beta_m(\vartheta_C)}{\lambda}\right)^{m/2}.
\end{equation}
The estimate is sharp for concentric superlevel components of the regular radial Helmholtz solution, equivalently for the first Dirichlet eigenfunction inside its first nodal ball. Moreover,
\[
\beta_m(\vartheta)=m(1-\vartheta)+O_m((1-\vartheta)^2)
\quad\text{as }\vartheta\uparrow1,
\]
while $\beta_m(\vartheta)\to j_{\nu,1}^2/2$ as $\vartheta\downarrow0$. Thus \eqref{eq:sharp-euclidean-superlevel} tends to the sharp Faber--Krahn bound at level zero.
\end{corollary}

\begin{proof}
Let $B_R$ be the ball with $|B_R|=|C|$. Brownian symmetrisation \cite[Theorem~1.4 and (5.3)]{BanuelosMendezHernandez2010} gives
\[
\mathbb P_{x_C}(\tau_C>t)\le\mathbb P_0(\tau_{B_R}>t)
\qquad(t>0).
\]
Integration against $(\lambda/2)e^{\lambda t/2}$ and \eqref{eq:superlevel-exp-moment} give
\[
\vartheta_C^{-1}
\le q_m\left(\frac{\lambda R^2}{2}\right),
\]
with the right side understood as infinite beyond its first pole. Since $q_m$ is increasing, this proves \eqref{eq:sharp-euclidean-superlevel}. The asymptotics follow from \eqref{eq:q_m_small_beta} and the first pole in \eqref{eq:q_m_bessel}.
\end{proof}

\paragraph{Krahn--Szeg\H{o} at nonzero levels.}
In the Euclidean setting, let $C_1,\ldots,C_k$ be distinct bounded components chosen from $\{\varphi>\ell\}$ and $\{\varphi<-\ell\}$, and set
$$
M_i=\max_{C_i}|\varphi|,
\qquad \vartheta_i=\frac{\ell}{M_i},
\qquad V=\sum_{i=1}^k|C_i|.
$$
With $\beta_m(0)=j_{\nu,1}^2/2$,
$$
\lambda\ge2\omega_m^{2/m}V^{-2/m}
\left(\sum_{i=1}^k\beta_m(\vartheta_i)^{m/2}\right)^{2/m}.
$$
In particular, if $M_i\ge\ell/(1-\eta)$ for every $i$, then
$$
\lambda\ge2\beta_m(1-\eta)
\left(\frac{k\omega_m}{V}\right)^{2/m}.
$$
At $\ell=0$, taking the two nodal domains of a second Dirichlet eigenfunction on a bounded connected Euclidean domain gives the sharp Krahn--Szeg\H{o} inequality.
For $\ell>0$, this follows by applying \eqref{eq:sharp-euclidean-superlevel} to $\varphi$ or $-\varphi$ on each component and summing:
$$
V\ge\omega_m\left(\frac2\lambda\right)^{m/2}
\sum_{i=1}^k\beta_m(\vartheta_i)^{m/2}.
$$
For $\ell=0$, apply this estimate at levels $\ell'\downarrow0$ to the components of $\{\pm\varphi>\ell'\}$ containing the maxima of the chosen nodal domains. Their volumes increase to those of the nodal domains, while $\beta_m(\ell'/M_i)\to\beta_m(0)$. The conclusions follow by rearranging and using the monotonicity of $\beta_m$.

\paragraph{No matching unconditional upper bound.}
The reverse of \eqref{eq:superlevel-global-volume} at the same scale is false when $m>1$. On the flat torus, $\varphi_k(x)=\cos(kx_1)$ has $\lambda=k^2$, and every component of $\{\varphi_k>\vartheta\}$ has volume
\[
2(2\pi)^{m-1}k^{-1}\arccos\vartheta
\asymp(1-\vartheta)^{1/2}\lambda^{-1/2}
\quad\text{as }\vartheta\uparrow1.
\]
Thus one retains only a lower volume bound and a ball centred at the component maximum that lies almost entirely in the component. A concentric ball contained in the component requires the extra hypotheses below; so does control of the first eigenvalue.

\subsection{Uniformly elliptic equations and composite membranes}

Let $\mathbf A$ be a measurable symmetric endomorphism satisfying
$$
a_-|\xi|_g^2\le\langle\mathbf A\xi,\xi\rangle_g\le a_+|\xi|_g^2,
$$
and let $(X_t^{\mathbf A})$ be the diffusion associated with $\tfrac12\int\langle\mathbf A\nabla v,\nabla v\rangle_g\, d\operatorname{vol}_g$.

\begin{theorem}[Superlevel components for uniformly elliptic equations]
\label{thm:elliptic-superlevel}
Let $U\subset X$ be open, let $0\le q\le Q$ with $Q>0$, and let $u$ be a nonzero continuous weak solution of
$$
\tfrac12\operatorname{div}_g(\mathbf A\nabla u)+qu=0
\qquad\text{in }U.
$$
Let $C\Subset U$ be a proper component of $\{u>\ell\}$, $\ell\ge0$, with $\partial C\ne\varnothing$, and choose $x_C\in C$ with $M_C=u(x_C)=\max_Cu$, put $\vartheta_C=\ell/M_C$, and let $\tau_C$ be the first exit time from $C$. Then
\begin{align}
\mathbb P_{x_C}^{\mathbf A}(\tau_C\le t)
&\le\min\left\{1,\frac{1-e^{-Qt}}{1-\vartheta_C}\right\},
\label{eq:elliptic-component-exit}\\
\mathbb E_{x_C}^{\mathbf A}\int_0^{\tau_C}
q(X_s^{\mathbf A})\frac{u(X_s^{\mathbf A})}{M_C}\,ds
&=1-\vartheta_C.
\label{eq:elliptic-component-occupation}
\end{align}
If $\ell>0$, then
\begin{equation}\label{eq:elliptic-component-moment}
\mathbb E_{x_C}^{\mathbf A}
\exp\left(\int_0^{\tau_C}q(X_s^{\mathbf A})\,ds\right)
=\frac{M_C}{\ell}.
\end{equation}
Consequently,
$$
\mathbb E_{x_C}^{\mathbf A}\tau_C\ge\frac{1-\vartheta_C}{Q};
$$
if $\ell>0$ and $q\ge q_0>0$, then $\mathbb E_{x_C}^{\mathbf A}\tau_C\le q_0^{-1}\log(M_C/\ell)$.

There are $r_0,c,C>0$, depending only on bounded geometry and $a_-,a_+$, such that
\begin{align}
\frac{\operatorname{vol}_g(B(x_C,r)\setminus C)}
{\operatorname{vol}_g(B(x_C,r))}
&\le C\min\left\{1,\frac{Qr^2}{1-\vartheta_C}\right\},
\qquad 0<r\le r_0,
\label{eq:elliptic-component-local}\\
\operatorname{vol}_g(C)
&\ge c\min\left\{r_0^m,
\left(\frac{1-\vartheta_C}{Q}\right)^{m/2}\right\}.
\label{eq:elliptic-component-volume}
\end{align}
\end{theorem}

\begin{proof}
Put $Y_t=X_{t\wedge\tau_C}^{\mathbf A}$ and $A_t=\int_0^{t\wedge\tau_C}q(X_s^{\mathbf A})\,ds$. The Fukushima decomposition makes $e^{A_t}u(Y_t)$ a martingale quasi-everywhere. For an arbitrary start in $C$, apply the Markov property at time $s>0$: before exit, the transition density does not charge the polar exceptional set, while after exit the stopped value is $\ell$. Letting $s\downarrow0$ and using continuity of the paths and of $u$ extends the identity to every starting point. Choosing a smooth $D$ with $\overline C\subset D\Subset U$, positivity of its first Dirichlet eigenvalue gives $\tau_C\le\tau_D<\infty$ almost surely. Since $A_t\le Qt$, $u\le M_C$ in $C$, and $u=\ell$ on $\partial C$,
$$
1\le e^{Qt}\bigl(\mathbb P_{x_C}^{\mathbf A}(\tau_C>t)
+\vartheta_C\mathbb P_{x_C}^{\mathbf A}(\tau_C\le t)\bigr),
$$
which proves \eqref{eq:elliptic-component-exit}. If $\ell>0$, then $u(Y_t)\ge\ell$ and the martingale identity gives $\ell\,\mathbb E_{x_C}e^{A_t}\le M_C$. Monotone convergence yields $\mathbb E_{x_C}e^{A_{\tau_C}}\le M_C/\ell$, while $e^{A_t}u(Y_t)\le M_Ce^{A_{\tau_C}}$; hence the stopped martingale is uniformly integrable and \eqref{eq:elliptic-component-moment} follows. The unweighted stopped Fukushima formula gives \eqref{eq:elliptic-component-occupation}. The mean-exit bounds are immediate.

The small-time Gaussian lower bound gives
$$
cr^{-m}\operatorname{vol}_g(B(x_C,r)\setminus C)
\le\mathbb P_{x_C}^{\mathbf A}(\tau_C\le r^2),
$$
and hence \eqref{eq:elliptic-component-local}. Taking $t=\min\{r_0^2,(1-\vartheta_C)/(2Q)\}$ in \eqref{eq:elliptic-component-exit} and using the Gaussian upper bound proves \eqref{eq:elliptic-component-volume}.
\end{proof}

Cox and McLaughlin prove existence for the Euclidean Dirichlet composite-membrane problem with $0<\alpha<\beta$ \cite[Proposition~4.4]{CoxMcLaughlin1990I}; their subsequent analysis gives the bang-bang structure \cite[Proposition~6.1 and Corollary~6.2\textup{(i)}]{CoxMcLaughlin1990II}. The related Schr\"odinger-potential problem is reviewed in \cite[Section~1]{ChanilloKenig2008}. The following application permits measurable $\mathbf A$ on a Lipschitz Riemannian domain and also allows $\alpha=0$.

\paragraph{Composite membranes.}
Let $\Omega$ be a bounded connected Lipschitz domain, let $\mathbf A$ satisfy the bounds above near $\overline\Omega$, and fix $0\le\alpha<\beta$ and $m_0$ strictly between $\alpha\operatorname{vol}_g(\Omega)$ and $\beta\operatorname{vol}_g(\Omega)$. Minimise
\[
\Lambda_1(\rho)=\inf_{\substack{v\in H^1_0(\Omega)\\
\int_\Omega\rho v^2\,d\operatorname{vol}_g>0}}
\frac{\tfrac12\int_\Omega
\langle\mathbf A\nabla v,\nabla v\rangle_g\,d\operatorname{vol}_g}
{\int_\Omega\rho v^2\,d\operatorname{vol}_g}
\]
over $\alpha\le\rho\le\beta$ with $\int_\Omega\rho\,d\operatorname{vol}_g=m_0$. There are a minimiser $\rho_{\mathrm{opt}}=\beta\mathbf1_D+\alpha\mathbf1_{\Omega\setminus D}$, a positive first eigenfunction $u$, and $\ell>0$ such that
\[
\{u>\ell\}\subset D\subset\{u\ge\ell\}
\quad\text{up to null sets}.
\]
Write $\Lambda_*=\Lambda_1(\rho_{\mathrm{opt}})$. For every component $C$ of $\{u>\ell\}$, choose $x_C\in C$ with $M_C=u(x_C)=\max_Cu$, put $\vartheta_C=\ell/M_C$, and let $\tau_C$ be its exit time. Then
\[
\mathbb E_{x_C}^{\mathbf A}e^{\beta\Lambda_*\tau_C}
=\frac{M_C}{\ell},
\qquad
\mathbb E_{x_C}^{\mathbf A}\int_0^{\tau_C}
\frac{u(X_s^{\mathbf A})}{M_C}\,ds
=\frac{1-\vartheta_C}{\beta\Lambda_*},
\]
while
\[
\begin{aligned}
\frac{\operatorname{vol}_g(B(x_C,r)\setminus D)}
{\operatorname{vol}_g(B(x_C,r))}
&\le C\min\left\{1,
\frac{\beta\Lambda_*r^2}{1-\vartheta_C}\right\},\\[-2pt]
&\hspace{6em}0<r\le r_0.
\end{aligned}
\]
\[
\sum_C(1-\vartheta_C)^{m/2}
\le C_\Omega(\beta\Lambda_*)^{m/2}\operatorname{vol}_g(D).
\]
Consequently, for $0<\eta<1$, the number of components with $M_C\ge\ell/(1-\eta)$ is at most $C_\Omega\eta^{-m/2}(\beta\Lambda_*)^{m/2}\operatorname{vol}_g(D)$. No regularity of $\partial D$ is assumed.

Indeed, the direct method gives a minimiser and a positive first eigenfunction. Replacing the density by a maximiser of $\int_\Omega\rho u^2\,d\operatorname{vol}_g$ can only decrease the Rayleigh quotient, so the new density is still minimising and $u$ is still a first eigenfunction; the bathtub principle gives the stated bang-bang form. Boundary H\"older regularity makes $u$ continuous on $\overline\Omega$ and zero on $\partial\Omega$, so $\ell>0$ and the components $C$ are relatively compact. On each such component, $\rho_{\mathrm{opt}}=\beta$ almost everywhere; since the uniformly elliptic diffusion has transition densities, it spends zero time in the exceptional null set. Thus \cref{thm:elliptic-superlevel} applies with $q=\Lambda_*\rho_{\mathrm{opt}}$ and $Q=\beta\Lambda_*$ and gives the two stopping identities and, because $C\subset D$ up to a null set, the displayed local estimate. Its volume bound, the inequality $\beta\Lambda_*\ge(a_-/2)\mu_1(-\Delta_g,\Omega)>0$, and summation over the disjoint components contained in $D$ give the weighted count.

\subsection{Boundary measures and quantitative estimates}
\label{subsec:boundary-quantities}

We now fix the analytic setup. Let $m\ge2$, $a\ge0$, and let the real-valued function $u\in C^2(\overline D)$ solve $(\tfrac12\Delta_g+a)u=0$ in $D=B(x_0,r)$, where $u(x_0)>0$, $0<r\le r_0$, and $ar^2\le c_0$. Shrink $r_0,c_0$ if necessary so that $2a<\mu_1(D)$. For $y\in D$, set
$$H_1(y):=\mathbb E_y[e^{a\tau_D}],\qquad
\widehat{\mathbb P}_y(A):=H_1(y)^{-1}\mathbb E_y[e^{a\tau_D}\mathbf{1}_A],$$
where $A\in\mathcal F_{\tau_D}$.  Write $\widehat{\mathbb E}_y$ for expectation under this probability, let $\nu_y$ be the distribution of $B_{\tau_D}$ under it, and set $\nu:=\nu_{x_0}$.

\begin{definition}[Boundary quantities]
\label{def:boundary-quantities}
Set $\bar m:=\int u\,d\nu=u(x_0)/H_1(x_0)>0$ and $S_+:=\sup_{\partial D}u>0$ and $S:=\sup_{\partial D}|u|$. For $1\le p\le\infty$, define
$$\varepsilon_p(x_0,r):=\frac{\|(\bar m-u)_+\|_{L^p(\nu)}}{\bar m},
\qquad
\varepsilon_+:=1-\frac{\bar m}{S_+},
\qquad
\varepsilon_{\partial}:=1-\frac{\bar m}{S},$$
and
\begin{align*}
Q_1(x_0)&:=\mathbb E_{x_0}[e^{a\tau_D}u(B_{\tau_D})^2],&
\delta_\nu&:=\frac{\operatorname{Var}_\nu(u)}{\int u^2\,d\nu}
=1-\frac{u(x_0)^2}{H_1(x_0)Q_1(x_0)},\\
H_2(x_0)&:=\mathbb E_{x_0}[e^{2a\tau_D}u(B_{\tau_D})^2],&
\delta_2&:=1-\frac{u(x_0)^2}{H_1(x_0)H_2(x_0)}\ge\delta_\nu,\\
\varepsilon_{\mathrm{hit}}&:=\widehat{\mathbb P}_{x_0}(\tau_Z<\tau_D),&
Z&:=\{u=0\}\cap D.
\end{align*}
Here $\tau_Z$ is the first hitting time of $Z$.  The quantities $\varepsilon_p$ are nondecreasing in $p$, and $\delta_\nu$ is exactly the relative loss in Cauchy--Schwarz applied to the stopping identity. The signed-boundary condition
\[
\frac{u(x_0)}{S_+H_1(x_0)}\ge1-\varepsilon
\]
is precisely $\varepsilon_+\le\varepsilon$; the corresponding condition with the absolute boundary maximum is $\varepsilon_{\partial}\le\varepsilon$.
\end{definition}

\begin{proposition}[Energy identity for the boundary variance]
\label{prop:boundary-variance-energy}
Put $h=H_1$ and $v=u/h$. Under $\widehat{\mathbb P}_x$, $v(B_{t\wedge\tau_D})$ is a square-integrable martingale, and
\begin{equation}\label{eq:boundary-variance-energy}
\operatorname{Var}_{\nu_x}(u)
=\widehat{\mathbb E}_x\int_0^{\tau_D}|\nabla v(B_s)|^2\,ds
=\frac1{h(x)}\mathbb E_x\int_0^{\tau_D}
e^{as}h(B_s)|\nabla(u/h)(B_s)|^2\,ds.
\end{equation}
Consequently,
\begin{equation}\label{eq:relative-boundary-variance-energy}
\frac{\delta_\nu}{1-\delta_\nu}
=\frac{H_1(x_0)}{u(x_0)^2}
\mathbb E_{x_0}\int_0^{\tau_D}e^{as}H_1(B_s)
\left|\nabla\!\left(\frac{u}{H_1}\right)(B_s)\right|^2\,ds.
\end{equation}
The identity \eqref{eq:boundary-variance-energy} requires only $a<\mu_1(D)$ and no sign assumption on $u$. If $u(x_0)\ne0$, then \eqref{eq:relative-boundary-variance-energy} holds and $\delta_\nu=0$ precisely when $u/H_1$ is constant in $D$. It is the first-moment counterpart of \cref{gauge:prop:energy}.
\end{proposition}

\begin{proof}
The density of $\widehat{\mathbb P}_x$ up to time $t\wedge\tau_D$ is
\[
\frac{e^{a(t\wedge\tau_D)}h(B_{t\wedge\tau_D})}{h(x)}.
\]
Since $(\tfrac12\Delta_g+a)(hv)=0$, It\^o's formula after this change of measure shows that $v(B_{t\wedge\tau_D})$ is a martingale with quadratic variation $\int_0^{t\wedge\tau_D}|\nabla v(B_s)|^2\,ds$. Its terminal value is $u(B_{\tau_D})$, because $h=1$ on $\partial D$. It\^o's isometry gives the first equality in \eqref{eq:boundary-variance-energy}; changing measure inside the time integral gives the second. Finally, $\bar m=u(x_0)/H_1(x_0)$ and $\delta_\nu/(1-\delta_\nu)=\operatorname{Var}_\nu(u)/\bar m^2$. The energy vanishes exactly when $\nabla(u/H_1)=0$ in $D$.
\end{proof}

The boundary quantities satisfy
\begin{enumerate}
\item[\textup{(i)}] $\varepsilon_1\le \varepsilon_+/(1-\varepsilon_+)\le \varepsilon_{\partial}/(1-\varepsilon_{\partial})$;
\item[\textup{(ii)}] $\varepsilon_2\le (\delta_\nu/(1-\delta_\nu))^{1/2}$;
\item[\textup{(iii)}] $\nu(\{u<0\})\le\varepsilon_{\mathrm{hit}} \le\min\{\varepsilon_+,\delta_\nu\}$;
\item[\textup{(iv)}] Let $\delta_2(x_0,\rho)$ denote the same quantity with $D$ replaced by $B(x_0,\rho)$.  Then $\rho\mapsto\delta_2(x_0,\rho)$ is nondecreasing for $0<\rho\le r$. For constant boundary data, $\delta_\nu=0$, whereas
$$\delta_2=1-\frac{H_1(x_0)}{\mathbb E_{x_0}[e^{2a\tau_D}]}>0$$
when $a>0$;
\item[\textup{(v)}] $\delta_\nu\le 2\varepsilon_{\partial}/(1-\varepsilon_{\partial})^2$, hence $\delta_\nu\le8\varepsilon_{\partial}$ if $\varepsilon_{\partial}\le\tfrac12$.
\end{enumerate}
To verify these claims, let $W:=u(B_{\tau_D})$. Since $\mathbb E_\nu[(W-\bar m)_+]= \mathbb E_\nu[(\bar m-W)_+]$,
$$
\mathbb E_\nu[(\bar m-W)_+]
\le S_+-\bar m=\varepsilon_+S_+,
$$
which proves the first inequality in \textup{(i)}; the second follows from $S_+\le S$. Also $\|(\bar m-W)_+\|_2^2\le\operatorname{Var}_\nu(W)$ and $\operatorname{Var}_\nu(W)/\bar m^2=\delta_\nu/(1-\delta_\nu)$, proving \textup{(ii)}.  The inequality $\delta_\nu\le\delta_2$ follows from $Q_1\le H_2$.

For \textup{(iii)}, put $A:=\{\tau_Z<\tau_D\}$.  The strong Markov property and optional stopping from $Z$ give $\widehat{\mathbb E}_{x_0}[W\mathbf{1}_A]=0$.  Hence
$$\bar m=\widehat{\mathbb E}_{x_0}[W\mathbf{1}_{A^c}]
\le S_+(1-\varepsilon_{\mathrm{hit}}),\qquad
\bar m^2\le(1-\varepsilon_{\mathrm{hit}})
\widehat{\mathbb E}_{x_0}[W^2],$$
which gives the two upper bounds.  If $W<0$, continuity forces the Brownian path to hit $Z$ before leaving $D$, proving the lower bound.

For \textup{(iv)}, the exit times of nested balls increase with the radius; thus $H_1$ increases, while $H_2$ increases because $e^{at}u(B_t)$ is a square-integrable martingale.  The formula for $\delta_2$ gives the claim.  The formula for constant boundary data follows directly, and its right-hand side is positive for $a>0$.  Finally,
$$\operatorname{Var}_\nu(W)\le\mathbb E_\nu[(S-W)^2]
\le2S\mathbb E_\nu[S-W]=2S^2\varepsilon_{\partial},$$
which proves \textup{(v)}.

\begin{lemma}[Bounds for weighted exit distributions]
\label{lem:offcentre-kernel}
There are constants depending only on bounded geometry and, where indicated, on $p$ or $q$, such that for $y\in D$ and $\xi\in\partial D$,
$$\frac{d\nu_y}{d\sigma_g}(\xi)\le C\,d(y,\xi)^{1-m}.$$
Moreover, $\operatorname{supp}\nu=\partial D$. If $\operatorname{dist}(y,\partial D)\ge\delta r$ and $1\le q\le\infty$, then
$$\left\|\frac{d\nu_y}{d\nu}\right\|_{L^q(\nu)}
\le C_q\delta^{-(m-1)(1-1/q)}.$$
Equivalently, for $1\le p\le\infty$ and $F\ge0$,
$$\int_{\partial D}F\,d\nu_y\le
C_p\delta^{-(m-1)/p}\|F\|_{L^p(\nu)}.$$
\end{lemma}

\begin{proof}
Let $P^{(0)}(y,\xi)$ denote the Poisson kernel of $D$ for Brownian motion, with respect to $d\sigma_g(\xi)$. Put $s:=\operatorname{dist}(y,\partial D)$ and $t:=d(\xi,\xi_y)$, where $\xi_y$ is a nearest boundary point to $y$. Small-ball geometry gives $d(y,\xi)\asymp s+t$. Standard small-ball Poisson-kernel bounds (compare the Euclidean formulas in \cite[Sections~5.2--5.3, equations~(27) and~(34)]{CZ95}) therefore give
\[
P^{(0)}(y,\xi)\le Cs(s+t)^{-m},\qquad
P^{(0)}(x_0,\xi)\ge cr^{1-m}.
\]
For $h_\xi=P^{(0)}(\cdot,\xi)$, let $\mathbb E_y^\xi$ denote expectation under the corresponding Doob transform. Its Green density gives
$$
\mathbb E_y^\xi\tau_D
=\frac{1}{P^{(0)}(y,\xi)}
\int_DG_D(y,z)P^{(0)}(z,\xi)\,d\operatorname{vol}_g(z).
$$
For $m\ge3$, normal-coordinate and time rescaling make the metric process a uniformly bounded time change of a uniformly elliptic divergence-form diffusion on $B_1$.  Its Green density relative to metric volume is the divergence-form Green kernel, the volume density is uniformly comparable to the Euclidean one, and the boundary Jacobian cancels in the Poisson-kernel quotient.  The boundary $3G$ estimate \cite[Theorem~3.1, estimate~(3.1.2)]{CranstonFabesZhao1988} therefore gives
$$
\frac{G_D(y,z)P^{(0)}(z,\xi)}{P^{(0)}(y,\xi)}
\le C\bigl(d(y,z)^{2-m}+d(z,\xi)^{2-m}\bigr).
$$
Integrating gives $\mathbb E_y^\xi\tau_D\le Cr^2$.  For $m=2$, controlled isothermal coordinates make the metric $h_\xi$-process a time change of its planar counterpart, with time-change rate bounded above and below, and the coordinate domain has area $O(r^2)$.  Thus \cite[Theorem~1]{CranstonMcConnell1983} gives the same bound. Consequently,
$$M:=\sup_{y\in D,\,\xi\in\partial D}\mathbb E_y^\xi\tau_D\le Cr^2.$$
The strong Markov property gives, by induction, $\mathbb E_y^\xi[\tau_D^n]\le n!M^n$.  After decreasing $c_0$ once and for all so that $aM\le\tfrac12$, summing the exponential series gives
$$
\mathbb E_y^\xi[e^{a\tau_D}]
\le(1-aM)^{-1}\le1+Car^2.
$$
Disintegration \cite[Proposition~5.12]{CZ95} gives $P^{(a)}(y,\xi)=P^{(0)}(y,\xi)\mathbb E_y^\xi[e^{a\tau_D}]$ for $\sigma_g$-almost every $\xi$; continuity extends the resulting inequality to every $\xi$.  Integrating it gives $1\le H_1(y)\le1+Car^2$.  Since $d\nu_y/d\sigma_g=P^{(a)}(y,\xi)/H_1(y)$ and $d(y,\xi)\le s+t$, the first bound follows.

At the pole, \cref{lem:poisson-comp-small}\textup{(i)}, applied with parameter $a/2$, and $H_1(x_0)\asymp1$ give $d\nu/d\sigma_g\asymp r^{1-m}$; in particular, $\operatorname{supp}\nu=\partial D$. Hence
$$\frac{d\nu_y}{d\nu}(\xi)\le Cr^{m-1}s(s+t)^{-m},\qquad
\nu\{t\le\rho\}\le C(\rho/r)^{m-1}.$$
For $1\le q<\infty$, integration in $t$ gives
$$\int\left(\frac{d\nu_y}{d\nu}\right)^q d\nu
\le Cr^{(m-1)q}s^q\int_0^{Cr}(s+t)^{-mq}t^{m-2}\frac{dt}{r^{m-1}}
\le C_q\left(\frac rs\right)^{(m-1)(q-1)}.$$
The case $q=\infty$ follows pointwise, and H\"older's inequality gives the last statement.
\end{proof}

\begin{lemma}[Distribution bound for the weighted Poisson extension]
\label{lem:weighted-poisson-distribution-bound}
For $F\ge0$ on $\partial D$, set
$$\mathcal PF(y):=\mathbb E_y[e^{a\tau_D}F(B_{\tau_D})].$$
Then, for $1\le p<\infty$ and $t>0$,
$$\frac{\operatorname{vol}_g\{y\in D:\mathcal PF(y)>t\}}
{\operatorname{vol}_g(D)}
\le C_p\left(\frac{\|F\|_{L^p(\nu)}}{t}\right)^{mp/(m-1)}.$$
\end{lemma}

\begin{proof}
First let $p=1$ and put $d\mu:=F\,d\sigma_g$.  Then $\mu(\partial D)\le Cr^{m-1}\|F\|_{L^1(\nu)}$ and the preceding lemma gives $\mathcal PF(y)\le C_0\int d(y,\xi)^{1-m}\,d\mu(\xi)$.  If $t\le C\|F\|_{L^1(\nu)}$, the claim is immediate.  Otherwise
$$\rho:=\left(4C_0\mu(\partial D)/t\right)^{1/(m-1)}\le cr.$$
The part with $d(y,\xi)\ge\rho$ is at most $t/4$.  Chebyshev's inequality and Fubini's theorem give
$$\operatorname{vol}_g\{\mathcal PF>t\}
\le\frac Ct\int_{\partial D}\int_{d(y,\xi)<\rho}
d(y,\xi)^{1-m}\,d\operatorname{vol}_g(y)d\mu(\xi)
\le\frac{C\rho}{t}\mu(\partial D).$$
Using $\operatorname{vol}_g(D)\asymp r^m$ proves the case $p=1$. Interpolation with $\|\mathcal PF\|_\infty\le C\|F\|_\infty$ gives the remaining cases.
\end{proof}

\begin{lemma}[Volume and eigenvalue bounds]\label{lem:whitney-packing}
Let $1\le p<\infty$, let $\Omega_+$ be the component of $\{u>0\}\cap D$ containing $x_0$, and suppose $\varepsilon_p\le\varepsilon\le\varepsilon_0(p)$.  Then
$$\frac{\operatorname{vol}_g(D\setminus\Omega_+)}{\operatorname{vol}_g(D)}
\le C_p\varepsilon^{mp/(m-1)},$$
and
$$\mu_1(D)\le\mu_1(\Omega_+)\le
\bigl(1+C_p\varepsilon^{mp/(m-1)}\bigr)\mu_1(D).$$
If $\varepsilon_\infty<1$, then $\Omega_+=D$, so both losses vanish.
\end{lemma}

\begin{proof}
Let $F:=(\bar m-u)_+/\bar m$ on $\partial D$ and $v:=\mathcal PF$.  At every $z\in Z$ optional stopping gives
$$v(z)\ge\frac{\mathbb E_z[e^{a\tau_D}(\bar m-u(B_{\tau_D}))]}{\bar m}
=H_1(z)\ge1.$$
Harnack's inequality therefore gives $v\ge c$ on $B(z,c' d_z)$, where $d_z:=\operatorname{dist}(z,\partial D)$.

For $y\in D\setminus\Omega_+$, let $z$ be the first zero on the radial geodesic from $x_0$ to $y$. Then $d(y,z)=d(x_0,y)-d(x_0,z)<r-d(x_0,z)=d_z$. Thus $D\setminus\Omega_+$ is covered by the balls $B(z,d_z)$. Vitali's lemma gives points $z_i\in Z$, with $d_i:=d_{z_i}$, such that the balls $B(z_i,d_i)$ are disjoint and the balls $B(z_i,5d_i)$ cover $D\setminus\Omega_+$.  The smaller Harnack balls are disjoint; hence the preceding distribution bound yields
$$\sum_i d_i^m\le C\operatorname{vol}_g\{v\ge c\}
\le C_p\varepsilon^{mp/(m-1)}r^m.$$
This proves the volume estimate.

For the eigenvalue bound, let $\psi$ be the positive first Dirichlet eigenfunction of $D$, normalised in $L^2(D)$.  Then
$$\psi\le Cr^{-m/2},\qquad
\psi(y)\le Cr^{-m/2}\frac{\operatorname{dist}(y,\partial D)}r.$$
Choose on $D$ functions $0\le\chi_i\le1$, equal to $1$ on $B(z_i,5d_i)\cap D$, supported in $B(z_i,10d_i)\cap D$, with $|\nabla\chi_i|\le C/d_i$, and put $U:=\bigcup_iB(z_i,10d_i)$ and $\eta:=\prod_i(1-\chi_i)$. This product is locally finite in $D$. At a point $y$, every contributing $d_i$ is at least $\operatorname{dist}(y,\partial D)/11$; disjointness bounds their number in each dyadic range, and summing those ranges gives
$$|\nabla\eta(y)|\le
\frac{C}{\operatorname{dist}(y,\partial D)}
\mathbf{1}_U(y).$$
For $\eta_N:=\prod_{i\le N}(1-\chi_i)$ the same bound is uniform in $N$. The boundary estimate for $\psi$, $|\eta_N|\le1$, and $\operatorname{vol}_g(U)\le C\sum_i d_i^m$ give, by dominated convergence, $\psi\eta_N\to\psi\eta$ in $W^{1,2}(D)$.  Since $\psi\eta_N\in W^{1,2}_0(D)$ and $\eta$ vanishes on the open union of the $B(z_i,5d_i)$ containing $D\setminus\Omega_+$, closedness and the standard quasi-everywhere characterisation give $\psi\eta\in W^{1,2}_0(\Omega_+)$. The same domination permits passage to the limit in the identities. The ground-state identity gives
$$\frac12\int_D|\nabla(\psi\eta)|^2
=\mu_1(D)\int_D\psi^2\eta^2+\tfrac12\int_D\psi^2|\nabla\eta|^2.$$
The preceding bounds imply
$$\int_D\psi^2|\nabla\eta|^2
\le C\varepsilon^{mp/(m-1)}r^{-2},\qquad
\int_D\psi^2\eta^2\ge1-C\varepsilon^{mp/(m-1)}\ge\tfrac12.$$
The Rayleigh quotient and $\mu_1(D)\asymp r^{-2}$ prove the upper bound; the lower bound is domain monotonicity. Finally, if $\varepsilon_\infty<1$, then, by the full support of $\nu$ and continuity, $u\ge(1-\varepsilon_\infty)\bar m>0$ on $\partial D$, and the stopping formula gives $u>0$ in $D$.
\end{proof}

\begin{theorem}[Rigidity from an $L^p$ boundary deficit]
\label{thm:boundary-deficit-rigidity}
For $0<\delta\le1$ and every $y\in D$ with $\operatorname{dist}(y,\partial D)\ge\delta r$,
\begin{equation}\label{eq:boundary-deficit-interior-bound}
u(y)\ge H_1(y)\bar m
\bigl(1-C_p\delta^{-(m-1)/p}\varepsilon_p\bigr),
\qquad 1\le p\le\infty .
\end{equation}
For $p=\infty$ the power of $\delta$ is understood to be $0$. Let $1\le p<\infty$.  There are $\varepsilon_0(p)>0$ and constants depending only on the indicated parameters and bounded geometry such that, if $\varepsilon_p(x_0,r)\le\varepsilon\le\varepsilon_0(p)$ and $\Omega_+$ is the component of $\{u>0\}\cap D$ containing $x_0$, then:
\begin{enumerate}
\item[\textup{(i)}] \textup{(Profile.)} For $\theta\in(0,1)$ and every integer $k\ge0$, with constants also depending on the order-$k$ geometry bounds,
$$
\sum_{j=0}^kr^j
\left\|\nabla^j\left(\frac{u}{u(x_0)}
-\frac{H_1}{H_1(x_0)}\right)\right\|_{L^\infty(B(x_0,\theta r))}
\le C_{\theta,k}\varepsilon .
$$
In particular, $\inf_{B(x_0,\theta r)}u\ge\tfrac14u(x_0)$ when $\varepsilon\le\varepsilon_{\theta,p}$.
\item[\textup{(ii)}] \textup{(Depth.)}
$$B\bigl(x_0,(1-C_p\varepsilon^{p/(m-1)})r\bigr)\subset\Omega_+.$$
\item[\textup{(iii)}] \textup{(Volume.)}
$$
\frac{\operatorname{vol}_g(D\setminus\Omega_+)}
{\operatorname{vol}_g(D)}
\le C_p\varepsilon^{mp/(m-1)}.
$$
\item[\textup{(iv)}] \textup{(Spectrum.)}
$$
\mu_1(D)\le\mu_1(\Omega_+)\le
\bigl(1+C_p\varepsilon^{mp/(m-1)}\bigr)\mu_1(D).
$$
\end{enumerate}
For $p=\infty$ the profile estimate holds with $\varepsilon=\varepsilon_\infty$, and $\varepsilon_\infty<1$ implies $\Omega_+=D$.
\end{theorem}

\begin{proof}
Optional stopping from $y$ gives
$$
u(y)=H_1(y)\int u\,d\nu_y
\ge H_1(y)\left(\bar m-\int(\bar m-u)_+\,d\nu_y\right),
$$
and \cref{lem:offcentre-kernel} gives \eqref{eq:boundary-deficit-interior-bound}.  Since $\int(u-\bar m)_+\,d\nu=\int(\bar m-u)_+\,d\nu=\bar m\varepsilon_1$, the same bound applied to $|u-\bar m|$ yields
$$
\left|1-\frac{u}{\bar mH_1}\right|
\le C\delta^{-(m-1)}\varepsilon_1
$$
at relative depth $\delta$.  Set $w:=\bar mH_1-u$.  On a fixed smaller ball this gives $|w|\le C_\theta\varepsilon\bar m$, and interior estimates give
$$
\|\nabla^jw\|_{L^\infty(B(x_0,\theta r))}
\le C_{\theta,j}r^{-j}\varepsilon\bar m .
$$
Since $w(x_0)=0$ and
$$
\frac{u}{u(x_0)}-\frac{H_1}{H_1(x_0)}=-\frac{w}{u(x_0)},
$$
\textup{(i)} follows; the lower bound uses $H_1/H_1(x_0)\ge\tfrac12$ after the fixed choice of $c_0$.

If $\delta\ge(2C_p\varepsilon)^{p/(m-1)}$, then \eqref{eq:boundary-deficit-interior-bound} is positive, proving \textup{(ii)}. Clauses \textup{(iii)} and \textup{(iv)}, and the final assertion for $p=\infty$, follow from \cref{lem:whitney-packing}.
\end{proof}

\paragraph{Signed-boundary near equality.}
Specialise the preceding setup to $a=\lambda/2$ and $u=\varphi|_{\overline D}$, where $\varphi$ is a real eigenfunction with $\Delta_g\varphi+\lambda\varphi=0$, $\lambda>0$. Let $D=B(x_0,r)$ be a wavelength-scale ball as in \cref{rem:cES-choice}, with $r=c_{\mathrm{ES}}\lambda^{-1/2}\le r_0$ and $\varphi(x_0)>0$, and set
\[
S_+:=\sup_{\partial D}\varphi>0,
\qquad
H_1(x):=\mathbb E_x[e^{(\lambda/2)\tau_D}].
\]
Near equality at level $\varepsilon\in(0,1)$ means that
\begin{equation}\label{eq:optional-stopping-near-equality}
\frac{\varphi(x_0)}{S_+H_1(x_0)}\ge1-\varepsilon.
\end{equation}
Optional stopping bounds the left-hand side by $1$; thus the condition says that the boundary value is nearly maximal under the exponentially weighted exit distribution.
Let $\Omega_+$ be the component of $\{\varphi>0\}\cap D$ containing $x_0$. There are $C\ge1$ and $\varepsilon_0>0$, depending only on bounded geometry, with $C\varepsilon_0^{1/(m-1)}\le1/2$, such that if \eqref{eq:optional-stopping-near-equality} holds with $\varepsilon\le\varepsilon_0$, then:
\begin{enumerate}
\item[\textup{(i)}] $B(x_0,(1-C\varepsilon^{1/(m-1)})r)\subset\Omega_+$.
\item[\textup{(ii)}] $\operatorname{vol}_g(D\setminus\Omega_+)/\operatorname{vol}_g(D) \le C\varepsilon^{m/(m-1)}$.
\item[\textup{(iii)}] $\mu_1(D)\le\mu_1(\Omega_+)\le (1+C\varepsilon^{m/(m-1)})\mu_1(D)$.
\item[\textup{(iv)}] For each $\theta\in(0,1)$ there is $\varepsilon_\theta>0$, depending only on $\theta$ and bounded geometry, such that, if $\varepsilon\le\varepsilon_\theta$, then $\inf_{B(x_0,\theta r)}\varphi\ge\tfrac14\varphi(x_0)$.
\end{enumerate}
The inclusion, volume estimate, and lower bound also hold with $\Omega_+$ replaced by the global positive nodal domain containing $x_0$. The spectral comparison concerns the component inside $D$; on a closed manifold, or for Dirichlet boundary conditions, the global nodal domain instead has first eigenvalue $\lambda/2$.
Indeed, if $\bar m=\varphi(x_0)/H_1(x_0)$, then \eqref{eq:optional-stopping-near-equality} and the first boundary-quantity comparison above give
\[
\varepsilon_1\le\frac{\varepsilon_+}{1-\varepsilon_+}
\le\frac{\varepsilon}{1-\varepsilon}\le2\varepsilon.
\]
The local conclusions follow from \cref{thm:boundary-deficit-rigidity} with $p=1$. The global assertions follow because $\Omega_+$ is contained in the corresponding global positive nodal domain; the eigenvalue identity is the ground-state property of a nodal domain.

\paragraph{Comparison with unconditional inner-radius bounds.}
At level zero and for $m\ge3$, optional stopping in \cref{thm:superlevel-stopping} recovers the local volume estimate at a nodal-domain maximum of Georgiev--Mukherjee \cite{MR3707293}. Such a volume estimate does not by itself exclude a codimension-one nodal set from an almost full ball. Unconditional inner-radius bounds use different information: Mangoubi \cite{Mangoubi2008Asymmetry} starts from local volume asymmetry in balls centred on the nodal set, while Charron--Mangoubi \cite{CharronMangoubi2024} obtain, for $m\ge3$ and large $\lambda$, a ball of radius $c\lambda^{-1/2}(\log\lambda)^{-(m-2)/2}$ centred at a nodal-domain maximum. For large $\lambda$, Charron \cite[Theorem~1.1]{Charron2026Capacity} subsequently improved this to
\[
c\lambda^{-1/2}(\log\log\lambda)^{-1/2}\quad(m=3),
\qquad
c\lambda^{-1/2}(\log\lambda)^{-(m-3)/2}\quad(m\ge4),
\]
with the ball centred at any point where the eigenfunction attains its maximum in absolute value on the nodal domain. Charron also constructs, on the flat square torus $\mathbb T^m$ for every $m\ge3$, sequences of eigenfunctions and nodal domains $\Omega_\lambda$ satisfying $\operatorname{inrad}(\Omega_\lambda)=o(\lambda^{-1/2})$
\cite[Theorem~1.2]{Charron2026Capacity}. Thus a uniform
$c_g\lambda^{-1/2}$ lower bound is false in general.

Under \eqref{eq:optional-stopping-near-equality}, the preceding conclusion gives the stronger radius $(1-C\varepsilon^{1/(m-1)})r$, together with volume and eigenvalue control. The construction in \cref{cor:egp-sharpness} below shows that this exponent is sharp, although small $\varepsilon$ need not occur at every nodal-domain maximum.

\subsubsection{Variance, perturbations, and frequency criteria}

\begin{corollary}[Stability in Cauchy--Schwarz]\label{cor:cs-stability}
The stopping identity and Cauchy--Schwarz give
$$u(x_0)^2\le H_1(x_0)Q_1(x_0),$$
with relative loss $\delta_\nu$. Equality holds exactly when the boundary trace is constant $\nu$-almost everywhere, in which case $u=\bar mH_1$ in $D$. If $\delta_\nu\le\varepsilon\le\varepsilon_0$, then \cref{thm:boundary-deficit-rigidity}\textup{(i)}--\textup{(iv)} hold with $p=2$ and with its error parameter replaced by $C\varepsilon^{1/2}$. The square root cannot be removed. The example in \cref{prop:sharpness-capacity} attains the depth exponent. Its volume and spectral exponents are sharp as well.
\end{corollary}

\begin{proof}
At equality the terminal value is $\bar m$ $\nu$-almost everywhere. The weighted exit kernels are strictly positive, so $\nu_y\ll\nu$, and optional stopping gives $u(y)=\bar mH_1(y)$. If $\delta_\nu\le\varepsilon\le\tfrac12$, then the boundary-quantity comparison above gives
\[
\varepsilon_2\le\left(\frac{\delta_\nu}{1-\delta_\nu}\right)^{1/2}
\le(2\varepsilon)^{1/2},
\]
proving the quantitative claims.

For sharpness of the square root, take harmonic boundary data $1+\sqrt\varepsilon\,g$ on a flat ball, where $g$ is smooth, has mean zero, and has nonzero Poisson extension. Then $\delta_\nu\asymp\varepsilon$, while the interior change is $\asymp\sqrt\varepsilon$.
\end{proof}

\paragraph{Small $L^2$ boundary perturbations.}
Suppose that on $\partial D$ one has $u=b+\xi$ with $b>0$, $\int_{\partial D}\xi\,d\nu=0$, and $\|\xi\|_{L^2(\nu)}\le\eta b$, where $\eta$ is sufficiently small. Then $\bar m=b$ and $\delta_\nu/(1-\delta_\nu)=\|\xi\|_{L^2(\nu)}^2/b^2\le\eta^2$, so \cref{cor:cs-stability} applies with error $C\eta$: for each fixed $\theta\in(0,1)$ and $k\ge0$ the profile error $\sum_{j\le k}r^j\|\nabla^j(u/u(x_0)-H_1/H_1(x_0))\|_{L^\infty(B(x_0,\theta r))}$ is at most $C_{\theta,k}\eta$, $B(x_0,(1-C\eta^{2/(m-1)})r)\subset\Omega_+$, and the relative volume and first-eigenvalue errors are at most $C\eta^{2m/(m-1)}$. No quantitative pointwise or derivative bound on $\xi$ enters.

In particular, suppose $u=b_0>0$ off a measurable set $A\subset\partial D$, $|u-b_0|\le Lb_0$ on $A$, and $L\nu(A)\le\tfrac12$ with $L^2\nu(A)$ sufficiently small. Since variance is minimised by the mean, $\operatorname{Var}_\nu(u)\le L^2b_0^2\nu(A)$ and $\bar m\ge b_0(1-L\nu(A))\ge\tfrac12b_0$, so one may take $\eta=2L\nu(A)^{1/2}$: the relative radius loss is at most $C(L^2\nu(A))^{1/(m-1)}$, and the relative volume loss and eigenvalue error are at most $C(L^2\nu(A))^{m/(m-1)}$. Since $\nu$ is comparable to normalised surface measure by \cref{prop:pole_isotropy_explicit}, $\nu(A)$ may be replaced by the relative surface area of $A$.

\begin{proposition}[Boundary-cap example]\label{prop:sharpness-capacity}
Let $D=B(0,r)\subset\mathbb R^m$, $m\ge2$, and let $a\ge0$ satisfy $ar^2\le c_0$.  For $0<\delta\le\delta_0$, let $u_\delta$ solve $(\tfrac12\Delta+a)u_\delta=0$ in $D$, with boundary value $-1$ on a spherical cap of radius $\delta r/2$, value $1$ outside the concentric cap $A_\delta$ of radius $\delta r$, and a fixed rescaled smooth transition between them taking values in $[-1,1]$.  Let $\Omega_+$ be the component of $\{u_\delta>0\}\cap D$ containing $0$.  Then, for every fixed $1\le p<\infty$:
\begin{enumerate}
\item[\textup{(i)}] $\varepsilon_p\asymp\delta^{(m-1)/p}$ and $\varepsilon_{\partial}\asymp\delta_\nu\asymp\delta^{m-1}$;
\item[\textup{(ii)}] $u_\delta<0$ at depth $c\delta r$ below the cap centre, so the depth $\asymp\varepsilon_p^{p/(m-1)}r$ is attained;
\item[\textup{(iii)}] $u_\delta<0$ on a ball of radius $c'\delta r$ there, and hence
$$
\frac{\operatorname{vol}_g(D\setminus\Omega_+)}
{\operatorname{vol}_g(D)}\asymp\delta^m,\qquad
\frac{\mu_1(\Omega_+)}{\mu_1(D)}-1\asymp\delta^m
\asymp\varepsilon_p^{mp/(m-1)}.
$$
\end{enumerate}
\end{proposition}

\begin{proof}
Rotation invariance makes $\nu$ the normalised surface measure.  Thus $1-\bar m\asymp\delta^{m-1}$,
$$
\|(\bar m-u_\delta)_+\|_{L^p(\nu)}^p\asymp\delta^{m-1},
\qquad
\operatorname{Var}_\nu(u_\delta)\asymp\delta^{m-1},
$$
which proves \textup{(i)}.

Let $y$ lie at depth $c\delta r$ below the cap centre.  For ordinary Brownian motion the probability of leaving through the inner cap tends to $1$ as $c\downarrow0$, uniformly for small $\delta$, by the half-space limit.  If $s=\operatorname{dist}(y,\partial D)=c\delta r$, then
$$
H_1(y)-1\le a\|H_1\|_\infty\mathbb E_y\tau_D\le C a r s,
$$
because the Markov property gives $H_1(y)-1=a\mathbb E_y\int_0^{\tau_D}H_1(B_t)\,dt$ and $\mathbb E_y\tau_D\le Crs$.  Hence the weighted inner-cap probability is at least the ordinary one divided by $H_1(y)$. Choosing $c$ and then $\delta_0$ small makes it greater than $3/4$.  If $F_\delta$ is the boundary value, then $u_\delta(y)=H_1(y)\int F_\delta\,d\nu_y\le-\tfrac12H_1(y)$; the same argument holds on $B(y,c'\delta r)$.  This proves \textup{(ii)} and the lower volume bound, while \cref{thm:boundary-deficit-rigidity} gives the matching upper bound. Faber--Krahn gives
$$
\frac{\mu_1(\Omega_+)}{\mu_1(D)}
\ge\left(\frac{\operatorname{vol}_g(D)}
{\operatorname{vol}_g(\Omega_+)}\right)^{2/m}
\ge1+c\delta^m,
$$
and the theorem gives the matching upper bound.
\end{proof}

\begin{corollary}[Sharpness for eigenfunctions on rational flat tori]
\label{cor:egp-sharpness}
Let $\mathbf T^m=\mathbb R^m/\Gamma$ be a flat torus such that the quadratic form $Q_\Gamma$ parametrising its Laplace spectrum is a positive multiple of a quadratic form with integer coefficients, and let $m\ge2$. For every fixed $1\le p<\infty$ and every $\delta\in(0,\delta_0]$, there is a sequence $\lambda\to\infty$ with eigenfunctions $\varphi_\lambda$ and balls $D_\lambda$ of radius comparable to $\lambda^{-1/2}$.  Write $\Omega_{+,\lambda}$ for the component of $\{\varphi_\lambda>0\}\cap D_\lambda$ containing the centre.  Then
$$
\varepsilon_p\asymp\delta^{(m-1)/p},\qquad
\sup_{\substack{z\in D_\lambda\\ \varphi_\lambda(z)=0}}
\operatorname{dist}(z,\partial D_\lambda)
\asymp\delta\,\lambda^{-1/2},
$$
and both the relative volume loss and $\mu_1(\Omega_{+,\lambda})/\mu_1(D_\lambda)-1$ are $\asymp\delta^m$. These rates are attained by genuine eigenfunctions.
\end{corollary}

\begin{proof}
Choose $\rho>0$ with $\rho^2/2\le c_0$ and apply \cref{prop:sharpness-capacity} to the equation $(\Delta+1)h_\delta=0$ on $B(0,\rho)$.  Approximate its smooth boundary data by a finite spherical-harmonic sum, with $C^2$ error smaller than a fixed small multiple of $\delta^{m-1}$. Since $1/2<\mu_1(B(0,\rho))$, the regular radial solution in every spherical-harmonic degree is nonzero at $\rho$. Hence the resulting finite Fourier--Bessel series is an entire Helmholtz solution. Boundary and interior estimates show that this approximation preserves all conclusions of the proposition.

By the inverse-localisation property \cite[Definition~1.1 and Theorem~1.2]{EGP2023}, the hypothesis on $Q_\Gamma$ gives a sequence of real eigenfunctions $-\Delta\varphi_\lambda=\lambda\varphi_\lambda$, with $\lambda\to\infty$, whose $\Gamma$-periodic lifts satisfy
\[
\varphi_\lambda(\,\cdot\,/\sqrt\lambda)\longrightarrow h_\delta
\quad\text{in }C^1(\overline{B(0,2\rho)}).
\]
Here one applies Definition~1.1 on a ball containing $\overline{B(0,2\rho)}$; the eigenfunctions may be taken real by taking real parts. For all sufficiently large $\lambda$, set $D_\lambda=B_{\mathbf T^m}(0,\rho/\sqrt\lambda)$. After lifting and dilation, this ball is identified with $B(0,\rho)$, so the boundary errors and the ball on which the limiting solution is negative persist. Since $a=\lambda/2$ and
\[
a(\rho/\sqrt\lambda)^2=\rho^2/2\le c_0,
\]
\cref{thm:boundary-deficit-rigidity} gives the spectral upper bound, while the ball on which $\varphi_\lambda<0$, together with Faber--Krahn, gives the matching lower bound.
\end{proof}

\begin{corollary}[Frequency and near-maximality criteria for rigidity]
\label{cor:frequency-near-maximality-rigidity}
Each condition below implies the conclusions of \cref{thm:boundary-deficit-rigidity} when the resulting error is sufficiently small:
\begin{enumerate}
\item[\textup{(i)}]
$$
\mathfrak F(x_0,r)+C a r^2\le c\varepsilon,\qquad
\mathfrak F(x_0,r):=\int_0^r\frac{2N_u(x_0,s)}s\,ds
=\log\frac{H_2(x_0)}{u(x_0)^2},
$$
with $p=2$ and error $C\varepsilon^{1/2}$;
\item[\textup{(ii)}] If $u=\varphi|_{\overline D}$ for a bounded eigenfunction $\Delta_g\varphi+\lambda\varphi=0$ on $X$, so that $a=\lambda/2$, and $\varphi(x_0)\ge(1-\eta)\|\varphi\|_{L^\infty(X)}$, then the conclusions hold with $p=1$ and error $C(\eta+\lambda r^2)$.
\end{enumerate}
Moreover, if $u$ has a zero in $B(x_0,(1-\delta)r)$, where $0<\delta\le\tfrac12$, then
\[
\mathfrak F(x_0,r)\ge c\delta^{m-1}-C a r^2.
\]
\end{corollary}

\begin{proof}
For \textup{(i)},
$$
\delta_2
=1-\exp\bigl(-\mathfrak F-\log H_1(x_0)\bigr)
\le\mathfrak F+C a r^2 .
$$
Hence $\delta_\nu\le\delta_2\le c\varepsilon$, and \cref{cor:cs-stability} applies. For \textup{(ii)}, if $M=\|\varphi\|_\infty$, then $\sup_{\partial D}\varphi\le M$ and $H_1(x_0)\le1+C\lambda r^2$, so $\varepsilon_+\le\eta+C\lambda r^2$.

For the final assertion, otherwise part \textup{(i)}, with $\varepsilon=c'\delta^{m-1}$, and \cref{thm:boundary-deficit-rigidity}\textup{(ii)} at $p=2$ would give a zero-free ball larger than $B(x_0,(1-\delta)r)$.
\end{proof}

\subsection{Averaging over the centre}\label{subsec:averaged-variance}

Assume $X$ is closed and connected, and let $\varphi$ be a nonzero real-valued eigenfunction satisfying $\Delta_g\varphi+\lambda\varphi=0$, $\lambda\ge1$. Fix $0<c\le c_1<j_{(m-2)/2,1}$, put $r=c\lambda^{-1/2}$, and, for $x\in X$, suppose $r$ is below the radius in \cref{prop:pole_isotropy_explicit}. Write $\tau_{x,r}$ for the exit time from $D_{x,r}=B(x,r)$ and set
\[
h_{x,r}(y)=\mathbb E_y[e^{(\lambda/2)\tau_{x,r}}],
\qquad
\nu_{x,r}(A)=\frac{\mathbb E_x[e^{(\lambda/2)\tau_{x,r}}
\mathbf1_{\{B_{\tau_{x,r}}\in A\}}]}{h_{x,r}(x)}
\]
and
\[
\delta_{x,r}=\frac{\operatorname{Var}_{\nu_{x,r}}(\varphi)}
{\int\varphi^2\,d\nu_{x,r}},\qquad
\mathcal V_{\lambda,c}(x)
=h_{x,r}(x)^2\operatorname{Var}_{\nu_{x,r}}(\varphi).
\]
Then $0\le\delta_{x,r}\le1$, with $\delta_{x,r}=1$ when $\varphi(x)=0$, and, where $\varphi(x)\ne0$,
\begin{equation}\label{eq:variance-odds-moving}
\mathcal V_{\lambda,c}(x)
=\frac{\delta_{x,r}}{1-\delta_{x,r}}\varphi(x)^2.
\end{equation}
By \cref{prop:boundary-variance-energy}, $\mathcal V_{\lambda,c}(x)$ also equals
\[
h_{x,r}(x)\mathbb E_x\int_0^{\tau_{x,r}}e^{(\lambda/2)s}
h_{x,r}(B_s)\left|\nabla\left(\frac{\varphi}{h_{x,r}}\right)(B_s)\right|^2ds.
\]

\begin{theorem}[Averaged boundary variance]
\label{thm:averaged-boundary-variance}
For every $\chi\in C^2(X)$,
\begin{equation}\label{eq:averaged-boundary-variance}
\left|\int_X\chi\,\mathcal V_{\lambda,c}\,d\operatorname{vol}_g
-Q_c\int_X\chi\varphi^2\,d\operatorname{vol}_g\right|
\le C_{g,c_1}\frac{c^2}{\lambda}\|\chi\|_{C^2}
\|\varphi\|_{L^2(X)}^2,
\qquad
Q_c:=q_m(c^2/2)^2-1.
\end{equation}
Moreover, as $c\downarrow0$, $Q_c=c^2/m+O_m(c^4)$. On a flat torus, the left-hand side of \eqref{eq:averaged-boundary-variance} vanishes for $\chi=1$.
\end{theorem}

\begin{proof}
Optional stopping gives
\[
\mathcal V_{\lambda,c}(x)
=h_{x,r}(x)\mathbb E_x[e^{(\lambda/2)\tau_{x,r}}
\varphi(B_{\tau_{x,r}})^2]-\varphi(x)^2.
\]
Let $A_r$ be spherical averaging at radius $r$. By \cref{prop:pole_isotropy_explicit}, with $q=q_m(c^2/2)$,
\[
\mathcal V_{\lambda,c}=q^2A_r(\varphi^2)-\varphi^2
+O_{g,c_1}\bigl(r^2A_r(\varphi^2)\bigr).
\]
Fubini and the normal-coordinate expansion give $A_r^*\chi=\chi+O_g(r^2\|\chi\|_{C^2})$. Since $r^2=c^2/\lambda$, this proves \eqref{eq:averaged-boundary-variance}; the expansion of $Q_c$ follows from \eqref{eq:q_m_small_beta}. On a flat torus $A_r$ preserves the integral and the pole kernel is exactly constant.
\end{proof}

\begin{corollary}[Rigidity outside a set of small $\varphi^2$-mass]\label{cor:typical-rigidity}
For $0<\varepsilon<1$,
\begin{equation}\label{eq:variance-tail}
\int_{\{\delta_{x,r}>\varepsilon\}}\varphi^2\,d\operatorname{vol}_g
\le\frac{1-\varepsilon}{\varepsilon}
\left(Q_c+C_{g,c_1}\frac{c^2}{\lambda}\right)\|\varphi\|_2^2.
\end{equation}
For $c\le c_{\mathrm{rig}}$ small, take $\varepsilon_c=Q_c^{(m-1)/m}$. Outside a set of $\varphi^2$-mass at most $C Q_c^{1/m}\|\varphi\|_2^2$, the conclusions of \cref{cor:cs-stability} hold on $B(x,r)$ after replacing $\varphi$ by $\operatorname{sgn}(\varphi(x))\varphi$. The profile error and relative radius loss are $O(Q_c^{(m-1)/(2m)})$ and $O(Q_c^{1/m})$, while the volume and spectral errors are $O(Q_c)$. Equivalently, these four quantities are $O(c^{(m-1)/m})$, $O(c^{2/m})$, $O(c^2)$, and $O(c^2)$.
\end{corollary}

\begin{proof}
On $\{\delta_{x,r}>\varepsilon\}$, \eqref{eq:variance-odds-moving} gives $\mathcal V_{\lambda,c}\ge\varepsilon(1-\varepsilon)^{-1}\varphi^2$; points with $\delta_{x,r}=1$ have $\varphi(x)=0$. Apply \cref{thm:averaged-boundary-variance} with $\chi=1$, then use \cref{cor:cs-stability}.
\end{proof}

\begin{remark}[Mass near the nodal set and Hardy bounds]
\label{rem:nodal-tube-mass}
Let $Z_\varphi=\{\varphi=0\}$.  Fix $\varepsilon_*>0$ and $A>1$ so that the depth conclusion of \cref{cor:cs-stability}, at radius $At\lambda^{-1/2}$, excludes every zero within distance $t\lambda^{-1/2}$ whenever $\delta_{x,At\lambda^{-1/2}}\le\varepsilon_*$.  The tube $\{d(\,\cdot\,,Z_\varphi)<t\lambda^{-1/2}\}$ then lies, apart from its zero-mass nodal set, in $\{\delta_{x,At\lambda^{-1/2}}>\varepsilon_*\}$, so \eqref{eq:variance-tail} with $c=At$ and $Q_{At}\le Ct^2$ gives $t_0,C>0$, depending only on bounded geometry, with
\begin{equation}\label{eq:nodal-tube-mass}
\int_{\{d(x,Z_\varphi)<t\lambda^{-1/2}\}}\varphi^2\,d\operatorname{vol}_g
\le Ct^2\|\varphi\|_2^2
\qquad(0<t\le t_0).
\end{equation}
Consequently, if $\{\Omega_i\}$ are the nodal domains and $\rho_i:=\operatorname{inrad}(\Omega_i) =\sup_{x\in\Omega_i}d(x,Z_\varphi)$, then the domains with $\rho_i<t\lambda^{-1/2}$ carry together at most
$Ct^2\|\varphi\|_2^2$ of the $\varphi^2$-mass. In particular, if
$a_i:=\|\varphi\|_{L^2(\Omega_i)}^2/\|\varphi\|_2^2$, then
$\rho_i\ge c\sqrt{a_i}\lambda^{-1/2}$. After decreasing $c$ if
necessary, for every $0<\theta\le1$, at least
$(1-\theta)\|\varphi\|_2^2$ of the $\varphi^2$-mass is carried by
nodal domains satisfying
$\rho_i\ge c\sqrt{\theta}\lambda^{-1/2}$.

Domain monotonicity gives the pointwise bound $d(x,Z_\varphi)\le C_g\lambda^{-1/2}$, hence the upper half of the two-sided comparison $\int_Xd(x,Z_\varphi)^p\varphi^2\,d\operatorname{vol}_g \asymp_p\lambda^{-p/2}\|\varphi\|_2^2$ for $p>0$; the lower half follows from \eqref{eq:nodal-tube-mass} at a single fixed small $t$. Layer cake converts \eqref{eq:nodal-tube-mass} into the Hardy bound
\[
\int_X\frac{\varphi^2}{d(x,Z_\varphi)^p}\,d\operatorname{vol}_g
\le C_p\lambda^{p/2}\|\varphi\|_2^2\qquad(0<p<2),
\]
the restriction $p<2$ coming from the convergence of $\int_0^{t_0}t^{1-p}\,dt$; at the endpoint $p=2$, \eqref{eq:nodal-tube-mass} rewritten with $s=t^{-2}$ gives the weak form
\[
\sup_{s>0}s\!\!\int_{\{(\sqrt\lambda\,d(x,Z_\varphi))^{-2}>s\}}\!\!
\varphi^2\,d\operatorname{vol}_g\le C\|\varphi\|_2^2.
\]
\end{remark}

\begin{corollary}[Quantum-ergodic sequences]\label{cor:qe-rigidity}
Let $\|\varphi_j\|_2=1$, $\lambda_j\to\infty$, and suppose
\[
\varphi_j^2\,d\operatorname{vol}_g
\ \rightharpoonup\
\operatorname{vol}_g(X)^{-1}d\operatorname{vol}_g.
\]
For fixed $c<j_{(m-2)/2,1}$,
\[
\mathcal V_{\lambda_j,c}\,d\operatorname{vol}_g
\ \rightharpoonup\
\frac{Q_c}{\operatorname{vol}_g(X)}\,d\operatorname{vol}_g.
\]
Hence, if $\operatorname{vol}_g(\partial U)=0$ and $0<\varepsilon<1$,
\[
\limsup_{j\to\infty}
\int_{U\cap\{\delta_{x,c\lambda_j^{-1/2}}>\varepsilon\}}\varphi_j^2
\le\frac{1-\varepsilon}{\varepsilon}\,
\frac{Q_c\,\operatorname{vol}_g(U)}{\operatorname{vol}_g(X)}.
\]
If $c\le c_{\mathrm{rig}}$, taking $\varepsilon=Q_c^{(m-1)/m}$ shows that the conclusions of \cref{cor:typical-rigidity} hold in every fixed $U$ outside $O(Q_c^{1/m})$ of its limiting eigenfunction mass. If the geodesic flow is ergodic, this applies to a full-density subsequence of every real eigenbasis \cite{ColinDeVerdiere1985,Zelditch1987}.
\end{corollary}

\begin{proof}
The first assertion follows from \eqref{eq:averaged-boundary-variance}; the second follows from \eqref{eq:variance-odds-moving} and weak convergence.
\end{proof}

\paragraph{Comparison with nodal-mass estimates.}
Jakobson--Nadirashvili \cite[Theorem~1]{JakobsonNadirashvili2002} prove that, for every $1\le p\le\infty$, the $L^p$ norms of the positive and negative parts of a nonconstant real eigenfunction are uniformly comparable. Mukherjee \cite[Theorem~2.3 and Remark~3.13]{Mukherjee2021Mass} strengthens this by showing that each sign retains a fixed fraction of its $L^p$ norm outside a sufficiently thin wavelength-scale nodal tube. For $p=2$, \eqref{eq:nodal-tube-mass} supplies a rate: the $\varphi^2$-mass in a tube of width $t\lambda^{-1/2}$ is $O(t^2)\|\varphi\|_2^2$.

Unlike these global mass estimates, \cref{thm:D} gives local profile, nodal-depth, volume, and first-eigenvalue control on typical wavelength-scale balls.

On surfaces, \cite{GeorgievMukherjee2022} gives the same $t^2$ rate in \eqref{eq:nodal-tube-mass}. The rate also follows in every dimension from the local $L^2$ wavelength-scale gradient estimate and Fubini. The unweighted tube volume is studied in \cite{MR3948283}. In dimensions $m\ge3$, Hezari \cite[Corollary~1.3]{Hezari2018InnerRadius} obtains an inner-radius improvement under configuration-space quantum ergodicity. For sequences satisfying the hypothesis of \cref{cor:qe-rigidity}, our corollary instead localises the averaged variance identity and yields the local conclusions of \cref{cor:typical-rigidity}.

\section{Steklov spectra via boundary local time}
\label{sec:steklov}

This section proves \cref{thm:steklov-boundary-collar-supremum} for weak
Steklov eigenfunctions on bounded Lipschitz domains. On smooth compact
manifolds with smooth boundary, Wang--Zhang
\cite[Corollary~1 and Lemma~6]{WangZhang2024} prove qualitative lower
comparisons between boundary and parallel-hypersurface norms in $L^2$ and,
for frequency-localised data, in $L^p$ for $1\le p\le\infty$ at
inverse-frequency distances (in particular, $t\le\sigma^{-1}$ for Steklov
eigenfunctions). The classical pseudodifferential description is not
available at Lipschitz regularity: the approach of
\cite{HislopLutzer2001,WangZhang2024} assumes smoothness, while the
pointwise method of \cite{GalkowskiToth2019} assumes real analyticity. The
stopping identity instead uses only the reflecting process and its boundary
local time. We first treat the $C^2$ case, where an explicit collar barrier
gives constants at the optimal scale $\sigma^{-1}$, and then deduce the
Lipschitz estimate from a quantitative local-time bound.

\subsection{Optional stopping for reflected Brownian motion}

The same stopping argument applies to the Steklov problem on a compact Riemannian manifold $(X,g)$ with $C^2$ boundary:
\begin{equation}\label{eq:steklov}
\tfrac12\Delta_g u=0\quad\text{in }X,\qquad
\partial_{\nu_{\mathrm{out}}}u=\sigma u\quad\text{on }\partial X,
\qquad \sigma\ge1.
\end{equation}
All eigenfunctions in this section are real-valued. Let $(B_t)$ be reflecting Brownian motion and $L_t$ its boundary local time, normalised so that It\^o's formula contains $\partial_{\nu_{\mathrm{in}}}u\,dL_t$. Both the Skorokhod--It\^o and Fukushima decompositions give the basic local martingale
\[
M_t:=e^{\sigma L_t}u(B_t).
\]

\begin{proposition}[Stopping identities for Robin boundary data]
Let $\mathcal C$ be a compact $C^2$ collar whose boundary components are $\partial X$ and $\Gamma$, let $\tau:=\inf\{t:B_t\in\Gamma\}$, and set
$$
\mathcal F_{\mathcal C}
=\{\phi\in H^1(\mathcal C):\operatorname{Tr}_\Gamma\phi=0\}.
$$
Suppose that $u\in H^1(\mathcal C)\cap L^\infty(\mathcal C)$ is real-valued and satisfies
$$
\frac12\int_{\mathcal C}\langle\nabla u,\nabla\phi\rangle_g\,
d\operatorname{vol}_g
=\frac\sigma2\int_{\partial X}u\phi\,d\sigma_g
\qquad(\phi\in\mathcal F_{\mathcal C}).
$$
Assume
$$
\sup_{x\in\mathcal C}\mathbb E_x[e^{2\sigma L_\tau}]<\infty.
$$
Then $(e^{\sigma L_{t\wedge\tau}}\widetilde u(B_{t\wedge\tau}))_{t\ge0}$ is a uniformly integrable martingale, and for quasi-every $x\in\mathcal C$:
\begin{enumerate}
\item[\textup{(i)}] \textup{(Optional stopping.)}
\begin{equation}\label{eq:steklov-rep}
u(x)=\mathbb E_x[e^{\sigma L_\tau}u(B_\tau)].
\end{equation}

\item[\textup{(ii)}] \textup{(Second moment.)} The function
$$
H_u^\sigma(x;\mathcal C):=\mathbb E_x[e^{2\sigma L_\tau}u(B_\tau)^2]
$$
is the unique bounded weak solution of
$$
\tfrac12\Delta_g H=0\quad\text{in }\mathcal C,\qquad
H=u^2\quad\text{on }\Gamma,\qquad
\partial_{\nu_{\mathrm{in}}}H=-2\sigma H\quad\text{on }\partial X.
$$

\item[\textup{(iii)}] \textup{(Variance identity.)}
$$
H_u^\sigma(x;\mathcal C)-u(x)^2
=\operatorname{Var}_x(M_\tau)
=\mathbb E_x\int_0^\tau e^{2\sigma L_t}|\nabla u(B_t)|^2\,dt.
$$
\end{enumerate}
If the representatives involved are continuous, the conclusions hold for every $x$.
\end{proposition}

\begin{proof}
For reflecting Brownian motion killed on hitting $\Gamma$, $\tau$ is the lifetime. In \cite[Example~5.2.2, especially (5.2.50)]{FOT11}, the boundary PCAF $\mathcal L_t$ has Revuz measure $d\sigma$ and enters the Skorokhod formula with coefficient $1/2$; the same calculation using Green's formula gives $d\sigma_g$ on the Riemannian collar. Thus the local time in our It\^o normalisation is $L_t=\mathcal L_t/2$ and has Revuz measure $\tfrac12d\sigma_g|_{\partial X}$. Compactness, uniform ellipticity, and the strong Markov property give an exponential tail for $\tau$, hence $\tau<\infty$ almost surely.

The local Fukushima decomposition \cite[Theorems~5.5.1 and~5.5.2]{FOT11} and the weak equation give, up to $\tau$,
$$
\widetilde u(B_t)-\widetilde u(B_0)
=\widetilde M_t^u-\sigma\int_0^t\widetilde u(B_s)\,dL_s,
\qquad
d\langle\widetilde M^u\rangle_t=|\nabla u(B_t)|_g^2\,dt.
$$
Consequently,
$$
e^{\sigma L_{t\wedge\tau}}\widetilde u(B_{t\wedge\tau})
=u(x)+\int_0^{t\wedge\tau}e^{\sigma L_s}\,d\widetilde M_s^u
$$
is a continuous local martingale. It is dominated by $\|u\|_\infty e^{\sigma L_\tau}\in L^2(\mathbb P_x)$, so localisation and dominated convergence make it a uniformly integrable martingale converging in $L^2$ at $\tau$. Optional stopping gives \textup{(i)}, and the isometry for the stochastic integral gives \textup{(iii)}.

For \textup{(ii)}, the strong Markov property and the Revuz formula give the stated weak mixed problem. If $H$ is another bounded weak solution, the same decomposition makes $e^{2\sigma L_{t\wedge\tau}}H(B_{t\wedge\tau})$ a local martingale. It is dominated by $\|H\|_\infty e^{2\sigma L_\tau}$, so optional stopping gives
$$
H(x)=\mathbb E_x[e^{2\sigma L_\tau}H(B_\tau)]
=\mathbb E_x[e^{2\sigma L_\tau}u(B_\tau)^2].
$$
Thus $H=H_u^\sigma$.
\end{proof}

The relevant collar width is of order $\sigma^{-1}$.  In the flat product collar $0\le s\le\rho$, stopped at $s=\rho$, one has, for $0\le\alpha<\rho^{-1}$,
$$
\mathbb E_s[e^{\alpha L_\tau}]
=\frac{1-\alpha s}{1-\alpha\rho},
$$
and the supremum over starting points is infinite for $\alpha\rho\ge1$. Thus the first moment requires $\sigma\rho<1$, while the second moment above requires $2\sigma\rho<1$. This $\sigma^{-1}$ scale is sharp.  On the unit ball, $u(r,\omega)=r^kY_k(\omega)$ is a Steklov eigenfunction with $\sigma=k$, and at distance $s=1-r$ from the boundary,
$$
\frac{\sup_{|x|=1-s}|u(x)|}{\sup_{|x|=1}|u(x)|}
=(1-s)^\sigma\le e^{-\sigma s}.
$$
Consequently, no comparison with a constant independent of $\sigma$ can hold on collars for which $\sigma\rho\to\infty$.

\subsection{Boundary-to-collar estimates for Steklov eigenfunctions}

Write $d(x):=d(x,\partial X)$ and $\mathcal C_\rho:=\{x:d(x)\le\rho\}$. Choose $\rho_*>0$ so that $d$ is $C^2$ on $\mathcal C_{\rho_*}$, and put $C_H=\|\Delta_gd\|_{L^\infty(\mathcal C_{\rho_*})}$. Throughout this subsection, $u$ satisfies \eqref{eq:steklov} and $\sigma\ge1$. Standard Robin boundary regularity makes $u$ bounded and continuous on the collar, so the preceding proposition applies.

\begin{lemma}[Exponential moment of boundary local time]
\label{lem:steklov_barrier}
Assume $C_H\rho_*\le\tfrac14$ and set $\rho_\sigma:=\min\{\rho_*,(4\sigma)^{-1}\}$ and $\tau:=\inf\{t:d(B_t)=\rho_\sigma\}$.  Then
$$
\sup_{x\in\mathcal C_{\rho_\sigma}}
\mathbb E_x[e^{2\sigma L_\tau}]\le4.
$$
Consequently, $(M_{t\wedge\tau})_{t\ge0}$ is uniformly integrable.
\end{lemma}

\begin{proof}
As in the preceding proposition, $\tau<\infty$ almost surely. Set $\psi:=1-2\sigma d-Kd^2$, where $K:=2\sigma C_H$.  Then $\psi=1$ and $\partial_{\nu_{\mathrm{in}}}\psi=-2\sigma$ on $\partial X$, so the boundary drift of $e^{2\sigma L_t}\psi(B_t)$ vanishes.  In the collar,
$$
\tfrac12\Delta_g\psi
=-K-\sigma\Delta_gd-Kd\Delta_gd
\le\sigma C_H(-1+2C_Hd)\le0.
$$
Moreover, $\psi$ decreases with $d$, and
$$
\psi(\rho_\sigma)
\ge1-\tfrac12-\tfrac18=\tfrac38,
$$
because $2\sigma\rho_\sigma\le\tfrac12$ and $K\rho_\sigma^2\le\tfrac12C_H\rho_*\le\tfrac18$.  Thus $e^{2\sigma L_{t\wedge\tau}}\psi(B_{t\wedge\tau})$ is a nonnegative local supermartingale, hence a supermartingale.  Therefore
$$
\tfrac14\mathbb E_x[e^{2\sigma L_\tau}]
\le\mathbb E_x[e^{2\sigma L_\tau}\psi(B_\tau)]
\le\psi(x)\le1,
$$
which proves the bound.  By Cauchy--Schwarz, $\sup_x\mathbb E_x[e^{\sigma L_\tau}]\le2$, and
$$
|M_{t\wedge\tau}|
\le\|u\|_{L^\infty(\mathcal C_{\rho_\sigma})}e^{\sigma L_\tau},
$$
so the stopped martingale is uniformly integrable.
\end{proof}

\begin{theorem}[Steklov observability at scale $\sigma^{-1}$]
\label{thm:steklov_obs}
With $\rho_\sigma$ as above,
$$
\sup_{\partial X}|u|
\le2\sup_{d=\rho_\sigma}|u|,
$$
and
$$
\|u\|_{L^2(\partial X)}^2
\le4e^{C_H\rho_\sigma}\|u\|_{L^2(\{d=\rho_\sigma\})}^2
\le6\|u\|_{L^2(\{d=\rho_\sigma\})}^2.
$$
\end{theorem}

\begin{proof}
Optional stopping and Cauchy--Schwarz give, for $x\in\partial X$,
$$
u(x)^2
\le\mathbb E_x[e^{2\sigma L_\tau}]\,
\mathbb E_x[u(B_\tau)^2]\le4v(x),
\qquad v(x):=\mathbb E_x[u(B_\tau)^2].
$$
By definition, $v(x)\le\sup_{d=\rho_\sigma}u^2$, so the first estimate follows directly.

The strong Markov property identifies $v$ as the bounded weak harmonic function with boundary value $u^2$ on $\{d=\rho_\sigma\}$ and zero normal derivative on $\partial X$. Set $I(r):=\int_{d=r}v\,dS_r$. The weak divergence theorem and coarea give, for almost every $r$, $\int_{d=r}\partial_{\nabla d}v\,dS_r=0$, and hence
$$
I'(r)=\int_{d=r}v\Delta_gd\,dS_r,
\qquad |I'(r)|\le C_HI(r).
$$
Thus $I(0)\le e^{C_H\rho_\sigma}I(\rho_\sigma)$.  Integrating $u^2\le4v$ over $\partial X$ proves the second estimate; its last constant uses $4e^{1/4}<6$.
\end{proof}

\subsection{Steklov eigenfunctions on bounded Lipschitz domains}

Let $\Omega\subset\mathbb R^m$, $m\ge2$, be a bounded Lipschitz domain.  A weak Steklov eigenfunction is a function $u\in H^1(\Omega)$ satisfying
\begin{equation}\label{eq:weak-steklov}
\int_\Omega\nabla u\cdot\nabla v\,dx
=\sigma\int_{\partial\Omega}uv\,d\mathcal H^{m-1}
\qquad(v\in H^1(\Omega)).
\end{equation}
A weak Steklov eigenfunction is harmonic in $\Omega$, and \eqref{eq:weak-steklov} is a homogeneous Robin condition with coefficient $\beta=-\sigma\in L^\infty(\partial\Omega)$. Hence \cite[Theorem~3.14\textup{(ii)}]{Nittka2011} yields a H\"older-continuous representative on $\overline\Omega$. We denote this H\"older-continuous, and therefore quasi-continuous, representative by $\widetilde u$, and identify $u$ with $\widetilde u$ below.

Let $(B_t)$ be the reflecting Brownian motion for $\mathcal E(f,g)=\tfrac12\int_\Omega\nabla f\cdot\nabla g$.  If $\mathcal L_t$ is the Bass--Hsu boundary local time, whose Revuz measure is $d\mathcal H^{m-1}$, set
$$
L_t:=\tfrac12\mathcal L_t.
$$
With this normalisation, the boundary term in It\^o's formula applied to the Skorokhod equation is $\partial_{\nu_{\mathrm{in}}}u\,dL_t$ \cite[Theorem~1]{BassHsu90}.

For $\rho>0$ set
$$
C_\rho:=\{x\in\overline\Omega:d(x,\partial\Omega)<\rho\},
\qquad
\tau_\rho:=\inf\{t:d(B_t,\partial\Omega)\ge\rho\}.
$$

\begin{lemma}[Expected local time in a Lipschitz collar]
There are $C_{\mathrm{Lip}},\rho_0>0$, depending only on $m$ and the Lipschitz character of $\Omega$, such that, for $0<\rho\le\rho_0$,
\begin{equation}\label{eq:collar-lt-lip}
\sup_{x\in C_\rho}\mathbb E_x[L_{\tau_\rho}]
\le C_{\mathrm{Lip}}\rho.
\end{equation}
Consequently, if $0\le\theta C_{\mathrm{Lip}}\rho<1$, then
\begin{equation}\label{eq:collar-khas-lip}
\sup_{x\in C_\rho}\mathbb E_x[e^{\theta L_{\tau_\rho}}]
\le(1-\theta C_{\mathrm{Lip}}\rho)^{-1}.
\end{equation}
Increasing $C_{\mathrm{Lip}}$ if necessary, assume $C_{\mathrm{Lip}}\rho_0\ge1$.
\end{lemma}

\begin{proof}
Put $\mathcal L=2L$. The Revuz formula, boundary measure growth, and the Neumann heat-kernel upper bound give, for small $t$,
$$
\sup_x\mathbb E_x\mathcal L_t
=\sup_x\int_0^t\int_{\partial\Omega}p_s(x,y)\,
d\mathcal H^{m-1}(y)\,ds\le C\sqrt t
$$
\cite[(6.1)--(6.3)]{Matsuura19}. For each $x\in C_\rho$, local Lipschitz geometry gives a ball $A=A_{x,\rho}\subset\{d(\cdot,\partial\Omega)\ge\rho\}$ of radius $c\rho$, within intrinsic distance $C\rho$ of $x$. For suitable $T>0$, put $\Delta:=T\rho^2$. The Neumann Gaussian lower bound \cite[Theorem~3.10]{GyryaSaloffCoste11} gives $p_\Delta(x,y)\gtrsim\rho^{-m}$ on $A$. Hence, for $0<\rho\le\rho_0$, $\mathbb P_x(\tau_\rho\le\Delta)\ge\int_A p_\Delta(x,y)\,dy\ge p>0$. The strong Markov property therefore gives a geometric tail on successive intervals of length $\Delta$, and the preceding estimate yields
$$
\sup_{x\in C_\rho}\mathbb E_x\mathcal L_{\tau_\rho}
\le\sum_{k\ge0}(1-p)^k\sup_y\mathbb E_y\mathcal L_\Delta
\le C\rho.
$$
This proves \eqref{eq:collar-lt-lip}; \eqref{eq:collar-khas-lip} is Khasminskii's lemma for the process killed at $\tau_\rho$.
\end{proof}

\begin{theorem}[Local-time martingale and boundary-to-collar sup-norm estimate]
\label{thm:steklov-boundary-collar-supremum}
Let $u$ be a weak Steklov eigenfunction with eigenvalue $\sigma$. Then
$$
M_t^\sigma:=e^{\sigma L_t}\widetilde u(B_t)
$$
is a continuous local martingale for quasi-every starting point in $\overline\Omega$. If $\sigma\ge1$, then, for every $c>1$, set
$$
\rho_c:=\frac1{c\sigma C_{\mathrm{Lip}}}.
$$
One then has
$$
\sup_{\partial\Omega}|u|
\le\frac{c}{c-1}\sup_{d(y,\partial\Omega)=\rho_c}|u(y)|.
$$
In particular, at $\rho_2=(2\sigma C_{\mathrm{Lip}})^{-1}$,
$$
\sup_{\partial\Omega}|u|
\le2\sup_{d(y,\partial\Omega)=\rho_2}|u(y)|.
$$
\end{theorem}

\begin{proof}
The Fukushima decomposition \cite[Theorems~5.2.2 and~5.2.3]{FOT11} gives
$$
\widetilde u(B_t)-\widetilde u(B_0)=\widetilde M_t^u+N_t^u,
\qquad
\langle\widetilde M^u\rangle_t
=\int_0^t|\nabla u|^2(B_s)\,ds.
$$
By \eqref{eq:weak-steklov}, $\mathcal E(u,v)=\int\widetilde v\,d\mu$ with $\mu=(\sigma/2)(\operatorname{Tr}u)\,d\mathcal H^{m-1}$. Since $\mu$ is a signed smooth measure, the Revuz correspondence \cite[Corollary~5.4.1]{FOT11} gives
$$
N_t^u=-\sigma\int_0^t\widetilde u(B_s)\,dL_s.
$$
The product rule now yields
$$
dM_t^\sigma=e^{\sigma L_t}\,d\widetilde M_t^u,
$$
because the two local-time terms cancel.

Now assume $\sigma\ge1$. Set $\rho:=\rho_c$ and $\tau:=\tau_\rho$.  Since $C_{\mathrm{Lip}}\rho_0\ge1$, one has $\rho\le\rho_0$, and \eqref{eq:collar-khas-lip} with $\theta=\sigma$ gives
$$
\sup_{y\in C_\rho}\mathbb E_y[e^{\sigma L_\tau}]
\le\frac{c}{c-1}.
$$
Since $u$ is bounded, this estimate makes the stopped martingale uniformly integrable.  Optional stopping gives, for quasi-every $y\in C_\rho$,
$$
|u(y)|
\le\frac{c}{c-1}\sup_{d(z,\partial\Omega)=\rho}|u(z)|.
$$
The exceptional set has zero relative capacity and therefore empty relative interior.  Its complement is dense in $\overline\Omega$, so continuity extends the estimate to every point of $\partial\Omega$.
\end{proof}

\subsection*{Acknowledgements}
The author thanks IIT Bombay for providing ideal working conditions.

\bibliographystyle{amsalpha}
\bibliography{references}

\end{document}